\documentclass[11pt,twoside,a4paper,english,french,reqno]{amsart}
\usepackage{amsmath, amsthm, amsfonts}
\usepackage{hyperref} 
\hypersetup{hidelinks,backref=true,pagebackref=true,hyperindex=true,colorlinks=true,citecolor=red,linkcolor=blue,breaklinks=true,urlcolor=ocre,bookmarks=true,bookmarksopen=false,pdftitle={Title},pdfauthor={Author}}

\usepackage{graphicx}
\usepackage{amsmath,amsthm,amssymb}
\usepackage{amscd,indentfirst,epsfig}
\usepackage{latexsym}
\usepackage{times}
\usepackage{enumerate}
\usepackage{mathrsfs}
\usepackage{amsopn}
\usepackage{amssymb,dsfont}
\usepackage{amssymb}
\usepackage{amsfonts,bm}
\usepackage{amsbsy,amsmath}
\usepackage{amscd}
at 11pt \font\sevenrsf=rsfs7 at 8pt \font\fiversf=rsfs5 at 6pt
\newfam\rsffam
\textfont\rsffam=\tenrsf \scriptfont\rsffam=\sevenrsf
\scriptscriptfont\rsffam=\fiversf

\newtheorem{theo}{Theorem}[section]
\newtheorem{lemme}[theo]{Lemma}
\newtheorem{propo}[theo]{Proposition}
\newtheorem{cor}[theo]{Corollary}
\newtheorem{hyp}[theo]{Assumption}
\newtheorem{hyp-not}[theo]{Assumption-Notations}
\newtheorem{defi}[theo]{Definition}
\newtheorem{exe}[theo]{Example}

\newtheorem{nb}[theo]{Remark}

\font\teneuf=eufm10 at 12pt \font\seveneuf=eufm7 at 8pt
\font\fiveeuf=eufm5 at 6pt
\newfam\euffam
\textfont\euffam=\teneuf \scriptfont\euffam=\seveneuf
\scriptscriptfont\euffam=\fiveeuf

\newfont{\secgoth}{eufm10 at 16pt}
\def \bq {\begin{equation}}
\def \eq {\end{equation}}

\def \leq {\leqslant}
\def \geq {\geqslant}
\def \v {v}
\def \N {\mathbb{N}}
\def \ind {\mathbf{1}}

\numberwithin{equation}{section}

\def\lp {L^1_+}
\def\lm {L^1_-}

\def \d {\mathrm{d}}
\def \D {\mathscr{D}}

\def \Rs {\mathcal{R}}

\def \ml {\mathsf{M}_{\lambda}}

\def \U {\mathcal{U}}
\def \dU {\mathds{U}}
\def \cU {\mathds{V}}

\def \R {\mathbb{R}}
\def \pO {\partial\Omega}

\def \M {\mathcal{}}

\def \X {\mathbb{X}}
\def \Y {\mathbb{Y}}
\def \l+ {L^1_+}
\def \l- {L^1_-}

\renewcommand{\epsilon}{\varepsilon}

\def \ds {\displaystyle}
\def \l {\lambda}
\def \T {\mathsf{T}}
\def \bxi {\bm{\xi}}

\def \M {\mathcal{M}}
\def\Prob{\operatorname{Prob}}
\def \ge {\geq}
\begin{document}
\title[Collisionless kinetic semigroups]{Invariant density and time asymptotics for collisionless kinetic equations with partly diffuse boundary operators}

 \author{B. Lods}

 \address{Universit\`{a} degli
 Studi di Torino \& Collegio Carlo Alberto, Department of Economics and Statistics, Corso Unione Sovietica, 218/bis, 10134 Torino, Italy.}\email{bertrand.lods@unito.it}

 \author{M. Mokhtar-Kharroubi}

 \address{Universit\'e de Bourgogne Franche-Comt\'e, Laboratoire de Math\'ematiques, CNRS UMR 6623, 16, route de Gray, 25030 Besan\c con Cedex, France
}
\email{mustapha.mokhtar-kharroubi@univ-fcomte.fr}

\author{R. Rudnicki}

\address{Institute of Mathematics, Polish Academy of Sciences, Bankowa 14, 40-007 Katowice, Poland}
\email{rudnicki@us.edu.pl}

\thanks{This paper was partially supported by the Polish National
Science Centre Grant No. 2017/27/B/ST1/00100 (RR). This work was written
while R.R. was a visitor to Universit\'e de Franche-Comt\'e and the authors
thank Universit\'e de Franche-Comt\'e for financial support of this visit.}

\maketitle

\begin{abstract}This paper deals with collisionless transport equations
in bounded open domains $\Omega \subset \R^{d}$ $(d\geq 2)$ with $\mathcal{C}^{1}$ boundary $\partial \Omega $, orthogonally
invariant velocity measure $\bm{m}(\d v)$ with support $V\subset \R^{d}$ and stochastic partly diffuse
 boundary operators $\mathsf{H}$  relating the outgoing and
incoming fluxes. Under very general conditions, such equations are governed
by stochastic $C_{0}$-semigroups $\left( U_{\mathsf{H}}(t)\right) _{t\geq 0}$ on $%
L^{1}(\Omega \times V,\d x \otimes \bm{m}(\d v)).$ We give a general criterion of irreducibility of $%
\left( U_{\mathsf{H}}(t)\right) _{t\geq 0}$ and we show that, under very natural assumptions, if an invariant density
exists then $\left( U_{\mathsf{H}}(t)\right) _{t\geq 0}$ converges strongly (not
simply in Cesar\`o means) to its ergodic projection. We show also that if no
invariant density exists then $\left( U_{\mathsf{H}}(t)\right) _{t\geq 0}$ is
\emph{sweeping} in the sense that, for any density $\varphi $, the total mass of $%
U_{\mathsf{H}}(t)\varphi $ concentrates near suitable sets of zero measure as $%
t\rightarrow +\infty .$  We show also a general weak compactness
theorem of interest for the existence of invariant densities. This theorem is based on several results on smoothness and transversality of the dynamical flow associated to $\left( U_{\mathsf{H}}(t)\right) _{t\geq
0}.$%
\end{abstract}
%
 \section{Introduction }

{Kinetic transport equations in bounded geometry is an important field of investigation which can be traced back to the seminal work \cite{bardos} where absorbing boundary conditions have been considered. For more general boundary conditions, relating the incoming and outgoing fluxes at the boundary of the physical domain, the well-posedness of associated transport equations with general force terms -- including Vlasov-like equations -- have been considered in \cite{beals, mjm1,mjm2} while a thorough analysis of the free transport equation with abstract boundary conditions on general domains have been performed in \cite{voigt} (see also \cite[Appendix of $\S$ 2, p. 249]{dautray}). Notice that, for a nonlinear and collisional kinetic equation  such as Boltzmann equation, taking into account general boundary conditions induces notoriously additional difficulties; we just mention here the works \cite{guo03} (dealing with close-to-equilibrium solutions) and \cite{mischler} (for renormalized solutions) and the references therein. 

We aim to emphasize right now that, even though the present contribution is dealing with collisionless  kinetic equations, we hope that the tools developed in the paper will be of some interest in nonlinear kinetic theory with general partly diffuse boundary conditions, especially in the study of the regularity up to the boundary for both the linearized and nonlinear Boltzmann equation as in \cite{chen,kim}.

The object of this paper is to build a general theory of time asymptotics $(t\to\infty)$ for multi-dimensional collisionless kinetic semigroups with partly diffuse boundary operators. Our construction is twofold: 
\begin{enumerate}
\item On the one hand, we continue previous functional analytic works \cite{AL05, AL11, Mkst,voigt} on substochastic semigroups governing collisionless transport equations with conservative boundary operators in $L^{1}$-spaces   and combine them to recent developments on the asymptotics of stochastic partially integral semigroups in $L^{1}$-spaces motivated by piecewise deterministic processes \cite{RT-K-k}.
\item On the other hand, we investigate the problem of the existence of invariant densities for collisionless transport equations. Such existence theory depends heavily on our understanding of compactness properties induced by the diffuse parts of the boundary operators. These compactness properties rely on the fine knowledge of smoothness and transversality properties of the dynamical flow induced by the semigroup.
\end{enumerate}
More precisely, we consider transport equations of the form
\begin{subequations}\label{1}
\begin{equation}\label{1a}
\partial_{t}\psi(x,v,t) + v \cdot \nabla_{x}\psi(x,v,t)=0, \qquad (x,v) \in \Omega \times V, \qquad t \geq 0
\end{equation} 
with initial data
\begin{equation}\label{1c}\psi(x,v,0)=\psi_0(x,v), \qquad \qquad (x,v) \in \Omega \times V,\end{equation}
under abstract ({conservative}) boundary conditions
\begin{equation}\label{1b}
\psi_{|\Gamma_-}=\mathsf{H}(\psi_{|\Gamma_+}),
\end{equation}\end{subequations}
where 
$$\Gamma _{\pm }=\left\{ (x,v)\in \partial \Omega \times V;\ \pm
v \cdot n(x)>0\right\}$$
($n(x)$ {being} the outward unit normal at $v\in \partial \Omega$, see Figure \ref{p:col-1})  and $\mathsf{H}$  {is
a linear 
boundary operator relating the
outgoing and incoming fluxes $\psi _{\mid \Gamma _{+}}$ and $\psi _{\mid
\Gamma _{-}}$ and is \textit{bounded}  on
the trace spaces
$$L^{1}_{\pm}=L^{1}(\Gamma_{\pm}\,;\,|v\cdot n(x)| \pi(\d  x)\otimes \bm{m}(\d v))=L^{1}(\Gamma_{\pm},\d\mu_{\pm}(x,v))$$
where $\pi$ denotes the Lebesgue surface measure on $\partial \Omega.$  We will focus our attention to the case of
\textit{nonnegative} and \textit{conservative} boundary conditions,
i.e.
\begin{equation}\label{eq1}
\mathsf{H}\psi \geq 0 \quad \text{ and } \qquad \|\mathsf{H}\psi\|_{\lm}=\|\psi\|_{\lp}, \quad
\text{ for any nonnegative } \psi \in \lp.\end{equation} }
Here 
$$\Omega
\subset \R^{d}\ \ (d\geq 2) \text{ is an open subset with $\mathcal{C}^{1}$ boundary $\partial \Omega 
$}$$
and our analysis takes place in the functional space%
$$X=L^{1}\left( \Omega \times V;\d x\otimes \bm{m}(\d v)\right)$$
{where $V \subset \R^{d}$ is the support of a nonnegative Borel measure $\bm{m}$  which is orthogonally invariant (i.e. invariant under the action of the orthogonal group of matrices in $\R^{d}$).} Such a measure covers the Lebesgue measure on $\R^{d}$, the surface Lebesgue measure on spheres (one speed or multi-group models) or even combinations of them.

 \begin{figure}
\centerline{\includegraphics{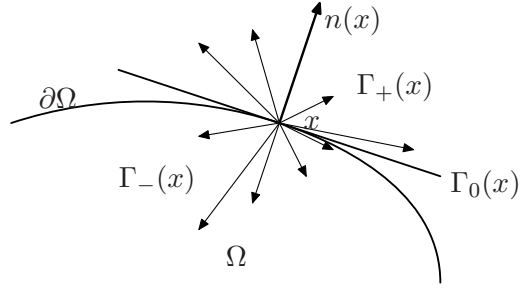}}
\centerline{
\begin{picture}(0,0)(0,0)
\put(84,48){$\Gamma_0(x)$}
\put(29,72){$x$}
\put(37,110){$n(x)$}
\put(49,85){$\Gamma_+(x)$}
\put(-40,50){$\Gamma_-(x)$}
\put(-70,80){$\partial\Omega$}
\put(0,20){$\Omega$}
\end{picture}
}
\caption{$x\in\partial\Omega$; $\Gamma_0(x)$ -- the tangent space to
 $\partial \Omega$ at $x$; $\Gamma_+(x)$ --  outward velocities;
$\Gamma_-(x)$ --  inward velocities.}     
\label{p:col-1}
\end{figure}

Very precise one-dimensional results corresponding to slab geometry have
been obtained in \cite{M-KR}.  Their extension to
multi-dimensional geometries $(d\geq 2)$ is {far from being elementary} and
is completely open. It is the main concern of the present work to provide
such a generalization.

Let
$$W=\left\{ \varphi \in X;\ v \cdot \nabla _{x}\varphi \in X,\ \varphi_{|\Gamma_{\pm}} \in L^{1}\left( \Gamma _{\pm }\right) \right\}$$
where $v \cdot \nabla _{x}\varphi $ is meant in a distributional sense, (see
Section \ref{sec:funct} below for a {reminder of the} trace theory) and let 
$$\mathsf{T}_{\mathsf{H}}\::\:\D(\mathsf{T}_{\mathsf{H}})\subset X \rightarrow X$$
be defined by%
$$\mathsf{T}_{\mathsf{H}}\varphi =-v \cdot \nabla _{x}\varphi ,\ \ \D(\mathsf{T}_{\mathsf{H}}){=\{\varphi \in W\;;\;\varphi_{|\Gamma_{-}}=\mathsf{H}(\varphi_{|\Gamma_{+}})\}.}$$ 
In contrast to the one-dimensional case  \cite{M-KR}, in
general, $\mathsf{T}_{\mathsf{H}}$ need \textit{not} be a generator. However, there exists a
unique \textit{extension} 
\[
A\supset \mathsf{T}_{\mathsf{H}} 
\]%
which generates a positive contraction $C_{0}$-semigroup $\left(
U_{\mathsf{H}}(t)\right) _{t\geq 0}$, see \cite{AL05,Mkst,voigt}. Notice that $\left( U_{\mathsf{H}}(t)\right) _{t\geq 0}$
need \textit{not} be stochastic, i.e. mass-preserving on the positive cone $X_{+}$ {of $X$}. Actually $\left( U_{\mathsf{H}}(t)\right)
_{t\geq 0}$ is stochastic if and only if 
\begin{equation}
A=\overline{\mathsf{T}_{\mathsf{H}}}  \label{le cas stochastique}
\end{equation}%
and different characterizations of this property are also available \cite{AL05,Mkst}. A general sufficient condition for $\left(
U_{\mathsf{H}}(t)\right) _{t\geq 0}$ to be stochastic is given in Proposition \ref{propo:honest}
below. 

Let us {briefly describe} the main contributions of this paper. We restrict ourselves to
the stochastic case (\ref{le cas stochastique}). \ A very important role is
played here \ by the irreducibility of $\left( U_{\mathsf{H}}(t)\right) _{t\geq 0}$
(see Definition \ref{defi:irred} below). \ When $\mathsf{T}_{\mathsf{H}}$ is not a generator, it is not
possible to handle easily its closure $A=\overline{\mathsf{T}_{\mathsf{H}}}.$ Despite this
fact, the resolvent of $A$ is given by an
"explicit" series converging strongly, see \eqref{eq:reso} below. By exploiting this series one
can derive a very general sufficient criterion of irreduciblity of $\left(
U_{\mathsf{H}}(t)\right) _{t\geq 0}$ in terms of properties of the stochastic
boundary operator $H,\ $see Proposition \ref{lem:irre} below. \ It is well known (see \cite{davies}) that if the kernel of the generator of an
irreducible stochastic $C_{0}$-semigroup \ is \textit{not trivial} (and
consequently one-dimensional) then the semigroup is \textit{ergodic} and
converges strongly in \textit{Cesar\`o means} to its one-dimensional
(positive) ergodic projection  (as $t\rightarrow +\infty $). Thus the
existence of an invariant density of $\left( U_{\mathsf{H}}(t)\right) _{t\geq 0}$ is
a \emph{cornerstone} of this construction {and is a fundamental problem for the understanding of the long-time behaviour of \eqref{1}}.

We {mainly consider} (local in space) stochastic boundary\ operators $\mathsf{H}\::\:\lp \to \lm$
which are (locally in space) convex combinations of {reflection} and diffuse operators of the
form
\begin{equation*}\begin{split}
\mathsf{H}\varphi(x,v)&=\alpha(x)\mathsf{R}\varphi(x,v) + (1-\alpha(x))\mathsf{K}\varphi(x,v)\\
&=\alpha (x)\varphi (x,\mathcal{V}(x,v))+(1-\alpha
(x))\int_{\Gamma _{+}(x)}k(x,v,v^{\prime })\varphi (x,v^{\prime })\bm{\mu}
_{x}(\d v^{\prime })\end{split}\end{equation*}
where 
$$\Gamma _{-}\ni (x,v)\longmapsto\left( x,\mathcal{V}(x,v)\right) \in \Gamma
_{+}$$
is a general $\ \mu $-preserving reflection law, $\bm{\mu}_{x}(\d v)=\left\vert
v\cdot n(x)\right\vert \bm{m}(\d v)$ and 
$$\int_{\Gamma _{-}(x)}k(x,v,v^{\prime })\bm{\mu}_{x}(\d v)=1,\ \quad (x,v')\in \Gamma
_{+}$$
where $\alpha\::\:x\in \pO \longmapsto \alpha (x)\in \left[ 0,1%
\right] $ is a measurable function.

{Regarding the long-time behaviour of the solution to \eqref{1}},  when 
\begin{equation}\label{eq:alpha1}
\mathrm{ess}\!\sup_{x\in \partial \Omega }\alpha (x)<1,\end{equation} we show
under quite general assumptions on the kernel $k(x,v,v^{\prime })$ that $%
\left( U_{\mathsf{H}}(t)\right) _{t\geq 0}$ is \textit{partially integral} (i.e. for
each $t>0$, $U_{\mathsf{H}}(t)$ dominates a non trivial integral operator). It
follows that if $\left( U_{\mathsf{H}}(t)\right) _{t\geq 0}$ has an invariant density 
$\Psi _{\mathsf{H}}$ then $\left( U_{\mathsf{H}}(t)\right) _{t\geq 0}$ is \textit{%
asymptotically stable}, i.e.%
\[
\lim_{t\rightarrow +\infty }\left\Vert U_{\mathsf{H}}(t)f -\Psi _{\mathsf{H}}\right\Vert
=0
\]%
for any density $f$; see Theorem \ref{th:asym-stab} for a precise statement. This
result provides us with a much more precise result than the mere Cesar\`o
convergence given by the general theory. Converse results are also given; indeed we show that if $\left( U_{\mathsf{H}}(t)\right) _{t\geq 0}$ has \emph{no
invariant density} then $\left( U_{\mathsf{H}}(t)\right) _{t\geq 0}$ is \emph{%
sweeping} with respect to suitable sets. {In a more precise way}, the total mass of
any trajectory of \eqref{1}
$$t\geq 0\longmapsto U_{\mathsf{H}}(t)\psi_{0}$$ concentrates {for large time $t \to \infty$} near small
(or large) velocities or near the boundary {$\pO \times V$}, see Theorem \ref{prop:sweep} for a precise statement. Such
asymptotics follow from general results on partially integral stochastic
semigroups \cite{PR,PR-JMMA2016,R-b95} which we
recall in Appendix \ref{sec:appB} of the paper. These general theorems on asymptotic
stability or sweeping of stochastic collisionless kinetic semigroups $\left(
U_{\mathsf{H}}(t)\right) _{t\geq 0}$ (and also some related results) are the\textit{\ 
}first object of this paper. Our second object is to deal with the \textit{%
existence} of an invariant density for stochastic collisionless kinetic
semigroups $\left( U_{\mathsf{H}}(t)\right) _{t\geq 0}$. \ As far as we know, the
existence of an invariant density is known \textit{only} for the classical
Maxwell diffuse model (see Example \ref{exe:maxwell} below) for which it is known that $%
\left( U_{\mathsf{H}}(t)\right) _{t\geq 0}$ is asymptotically stable \cite{AN}.

Thus our second object is to provide an {existence} theory of invariant density
for such kinetic models. \ We show first, for general stochastic boundary
operators $\mathsf{H}$, that $0$ is an eigenvalue \ of $\mathsf{T}_{\mathsf{H}}$ associated to a
nonnegative eigenfunction if and only if there exists {a nonnegative solution} $\varphi \in
\lp$ to the eigenvalue problem
\begin{equation}\label{Boundary spectral problem}
\mathsf{M}_{0}\mathsf{H}\varphi =\varphi , \end{equation}
which satisfies the additional condition%
\begin{equation}\label{Additional condition}
\int_{\Gamma _{+}}\varphi (x,v)|v|^{-1}\d\mu _{+}(x,v) <+\infty \end{equation}
where 
$$\mathsf{M_{0}}\::\:\lm \rightarrow \lp$$
is the \emph{stochastic} operator defined by 
$$\left(\mathsf{M}_{0}\varphi \right) (x,v)=\varphi (x-\tau _{-}(x,v)v,v);\quad (x,v)\in
\Gamma _{+}, \qquad {\varphi \in \lm}$$
where $\tau _{-}(x,v)$ is the exit time function (see the definition in
Section \ref{sec:funct} below).

To study the existence of an invariant density, we introduce the sub-class
of \textit{regular} partly diffuse boundary operators such that the diffuse
part is "\textit{weakly compact with respect to velocities}" (see Definition
\ref{defi:regul} below) which enjoys nice approximation properties. The part of the paper
concerned with the existence of an invariant density is very involved and is
based on a series of highly technical results culminating in  {a key spectral
 result} 
\begin{equation}\label{eq:stability}
r_{\mathrm{ess}}(\mathsf{M}_{0}\mathsf{H}) <1
\end{equation}
(see Theorem \ref{theo:ressM0}) where $r_{\mathrm{ess}}$ refers to the essential spectral radius.  {Inequality \eqref{eq:stability} is shown to be true under some smallness assumption on the oscillations of the diffuse parameter $\beta(\cdot)=1-\alpha(\cdot)$
\begin{equation*}\label{eq:oscil}
\big(\|1-\beta(\cdot)\|_{L^{\infty}(\pO)}+\|\beta(\cdot)\|_{L^{\infty}(\pO)}\big)^{2}-\|\beta(\cdot)\|_{L^{\infty}(\pO)}^{2} <1\end{equation*}
 (see Theorem \ref{theo:ressM0}). However, we believe such an assumption to be purely
technical, (see   Remark \ref{nb:weaknb}). }
It is then straightforward to check that the spectral problem (\ref{Boundary
spectral problem}) has a solution {under \eqref{eq:stability}}. 
 If the corresponding eigenfunction satisfies the additional
condition (\ref{Additional condition}) then $\left( U_{\mathsf{H}}(t)\right) _{t\geq
0}$ is asymptotically stable. If \textit{not} we show a more precise
sweeping behaviour: the total mass of any trajectory $t\geq 0\longmapsto
U_{\mathsf{H}}(t)\psi_{0}$ of \eqref{1} concentrates near the zero velocity as $t\rightarrow +\infty $, see Theorem \ref{Foguel-alternative}.

The above spectral  {inequality \eqref{eq:stability}}  is a consequence of a  key weak
compactness theorem namely: for any  {{regular} diffuse operator $\mathsf{K}$} 
$$\mathsf{K} \mathsf{M}_{0}\mathsf{K} \::\:\lp  \rightarrow \lm \quad \text{ is weakly compact}.$$ The proof of this important result (Theorem \ref{propo:weakcompact}), using the
Dunford-Pettis criterion, is highly technical and is given in numerous
steps. {Roughly speaking, the main difficulty lies in the fact that $\mathsf{K}$ induces compactness only in the velocity variables and several iterations and changes of variables are necessary to produce the missing compactness in the space variable $x \in \pO$. Such changes of variables are non trivial and have to be carefully justified. To do this,} we take advantage of the stochastic character of the
various operators involved and we show (see Corollary \ref{cor:ba})  \textit{smoothness properties} of the  ballistic flow
$$\bxi\::\:(x,v)\in \Gamma _{+}\longmapsto \bxi(x,v)=(x-\tau
_{-}(x,v)v,v).$$
and its inverse 
$$\bxi^{-1}\::\:(x,\omega) \in  {\Gamma}_{-} \longmapsto \bxi^{-1}(x,\omega)=(x+\tau_{+}(x,\omega)\omega,\omega) \in \Gamma_{+}.$$
We prove in particular  the following property: for any $x \in \pO$, there exists a set $S(x) \subset \mathbb{S}^{d-1}$ of zero surface Lebesgue measure such that the differential of the mapping
$$\omega \in {\Gamma}_{-}(x) \setminus S(x) \mapsto x+\tau_{+}(x,\omega)\omega$$
has maximal rank (see Proposition \ref{prop:diff}).
These non trivial smoothness and transversality results involve intrinsic tools from differential geometry and are postponed in  Appendix \ref{app:ballistic} for the simplicity of reading but we wish to point out 
that our analysis of the flow induced by $\left( U_{\mathsf{H}}(t)\right) _{t\geq 0}$ is
new (even if results similar to some of ours appear e.g. in \cite{guo03}, see Remark \ref{nb:guo}) and has its own interest independently of
the main motivation of this paper.  \  

As far as we know, most of our results are new and appear here for the first
time. Finally, we note that the assumption that $\partial \Omega $ is of
class $\mathcal{C}^{1}$ plays a role only for the results on smoothness and
transversality of the flow stated in Appendix \ref{app:ballistic}; it is likely that the
results stated there remain valid for $\pO$ which is only
\emph{piecewise of class} $\mathcal{C}^{1}.$

{The paper is organized as follows: in Section \ref{sec:funct}, we introduce the mathematical framework and notations used in the rest of the paper and  establish several properties of the various operators involved in our subsequent analysis. In Section \ref{sec:BC} we introduce and analyse the general class of boundary operators we investigate in the rest of the paper. Section \ref{sec:general} is devoted to general criteria for the ergodic convergence of the semigroup $(U_{\mathsf{H}}(t))_{t\geq 0}$ (see Theorem \ref{theo:irred}) which is related to the  study of the eigenvalue problem \eqref{Boundary spectral problem}  as well as the irreducibility property of $(U_{\mathsf{H}}(t))_{t\geq 0}$. In Section \ref{sec:comp} we establish the main technical result of the paper (Theorem \ref{propo:weakcompact}) as well of its consequence on the essential radius \eqref{eq:stability}, see Theorem \ref{theo:ressM0}. Section \ref{ss:existence-in-dens} is devoted to the main existence result for an invariant density, Theorem \ref{theo:density}. The question of the asymptotic stability of $(U_{\mathsf{H}}(t))_{t\geq0}$ is then discussed in Section \ref{s:semigroups} while the \emph{sweeping properties} of $(U_{\mathsf{H}}(t))_{t\geq0}$, when no invariant density exists, are given in Section \ref{ss:sweeping}. As already mentioned, the paper ends with two Appendices. A first one, Appendix \ref{app:ballistic} contains all the technical results regarding the smoothness and transversality of the ballistic flow while Appendix \ref{sec:appB} recall several important results about partially integral semigroup and sweeping properties used in Section \ref{s:semigroups} and \ref{ss:sweeping}.
 }
 
{We end this Introduction by mentioning that a} related work dealing with rates of convergence to equilibrium, in the spirit of \cite{aoki,liu}, is now in
preparation \cite{L-MK} {extending the results of} \cite{MK-seifert} devoted to slab geometry. Moreover, we hope also to take advantage
of the tools developed here to revisit some {important works} (see e.g. \cite{CPSV, evans}
and references therein) on {stochastic billiards} \cite{L-MK-R}. 
 
\section{Mathematical setting and useful formulae}\label{sec:funct}

\subsection{Functional setting} We introduce the partial Sobolev space 
$$W_1=\{\psi \in X\,;\,v
\cdot \nabla_x \psi \in X\}.$$ It is known \cite{ces1,ces2,dautray}
that any $\psi \in W_1$ admits traces $\psi_{|\Gamma_{\pm}}$ on
$\Gamma_{\pm}$ such that 
$$\psi_{|\Gamma_{\pm}} \in
L^1_{\mathrm{loc}}(\Gamma_{\pm}\,;\,\d \mu_{\pm}(x,v))$$ where
$$\d \mu_{\pm}(x,v)=|v \cdot n(x)|\pi(\d x) \otimes \bm{m}(\d v),$$
denotes the "natural" measure on $\Gamma_{\pm}.$ Notice that, since $\d\mu_{+}$ and $\d\mu_{-}$ share the same expression, we will often simply denote them by 
$$\d \mu(x,v)=|v \cdot n(x)|\pi(\d x) \otimes \bm{m}(\d v),$$
the fact that it acts on $\Gamma_{-}$ or $\Gamma_{+}$ being clear from the context.  Note that
$$\partial
\Omega \times V:=\Gamma_- \cup \Gamma_+ \cup \Gamma_0,$$ where
$$\Gamma_0:=\{(x,v) \in \partial \Omega \times V\,;\,v \cdot
n(x)=0\}.$$
We introduce the space
$$W=\left\{\psi \in W_1\,;\,\psi_{|\Gamma_{\pm}} \in L^1_{{\pm}}\right\}.$$
One can show \cite{ces1,ces2} that $W=\left\{\psi \in
W_1\,;\,\psi_{|\Gamma_+} \in \lp\right\} =\left\{\psi \in
W_1\,;\,\psi_{|\Gamma_-} \in \lm\right\}.$ Then, the \textit{trace
operators} $\mathsf{B}^{\pm}$:
\begin{equation*}\begin{cases}
\mathsf{B}^{\pm}: \:&W_1 \subset X \to L^1_{\mathrm{loc}}(\Gamma_{\pm}\,;\,\d \mu_{\pm})\\
&\psi \longmapsto \mathsf{B}^{\pm}\psi=\psi_{|\Gamma_{\pm}},
\end{cases}\end{equation*}
are such that $\mathsf{B}^{\pm}(W)\subseteq L^1_{\pm}$. Let us define the
{\it maximal transport operator } $\mathsf{T}_{\mathrm{max}}$  as follows:
\begin{equation*}\begin{cases} \mathsf{T}_{\mathrm{max}} :\:& \D(\mathsf{T}_{\mathrm{max}}) \subset X \to X\\
&\psi \mapsto \mathsf{T}_{\mathrm{max}}\psi(x,v)=-v \cdot \nabla_x
\psi(x,v),
\end{cases}\end{equation*}
with domain $\D(\mathsf{T}_{\mathrm{max}})=W_1.$
 Now, for any \textit{
bounded boundary operator} $\mathsf{H} \in\mathscr{B}(L^1_+,L^1_-)$, define
$\mathsf{T}_{\mathsf{H}}$ as
$$\mathsf{T}_{\mathsf{H}}\varphi=\mathsf{T}_{\mathrm{max}}\varphi \qquad \text{ for any }
\varphi \in \D(\mathsf{T}_{\mathsf{H}}),$$ where 
$$\D(\mathsf{T}_{\mathsf{H}})=\{\psi \in
W\,;\,\psi_{|\Gamma_-}=\mathsf{H}(\psi_{|\Gamma_+})\}.$$ In particular, the
transport operator with absorbing conditions (i.e. corresponding
to $\mathsf{H}=0$) will be denoted by $\mathsf{T}_0$. We recall here that there exists a \emph{unique} minimal extension $(A,\D(A))$ of $(\mathsf{T}_{\mathsf{H}},\D(\mathsf{T}_{\mathsf{H}}))$ which generates a nonnegative $C_{0}$-semigroup $(U_{\mathsf{H}}(t))_{t\geq 0}$ in $X$. We note that $\D(A) \subset W_{1}$ and $A\varphi=-v \cdot \nabla_{x}\varphi=\mathsf{T}_{\mathsf{max}}\varphi$ for any $\varphi \in \D(A)$ but the traces $\mathsf{B}^{\pm}\varphi$ need not to belong to $L^{1}(\Gamma_{\pm},\d\mu_{\pm})$. The resolvent of $A$ is given by
\begin{equation}\label{eq:reso}\Rs(\lambda,A)f=\mathsf{R}_{\lambda}f+\sum_{n=0}^{\infty}\mathsf{\Xi}_{\lambda}\mathsf{H}\left(\mathsf{M}_{\lambda}\mathsf{H}\right)^{n}\mathsf{G}_{\lambda}f, \qquad \forall f \in X\,,\,\lambda >0\end{equation}
where the series is \emph{strongly} converging in $X$. See \cite[Theorem 2.8]{AL05} for details. Moreover, $(U_{\mathsf{H}}(t))_{t\geq0}$ is a \emph{stochastic} $C_{0}$-semigroup, i.e.
$$\|U_{\mathsf{H}}(t)f\|_{X}=\|f\|_{X} \qquad \forall f \in X_{+}\;;\;t \geq 0$$
if and only if 
$$A=\overline{\mathsf{T}_{\mathsf{H}}}.$$
Actually, under suitable assumptions on $\mathsf{H}$ (see Prop. \ref{propo:honest}), $A=\overline{\mathsf{T}_{\mathsf{H}}}$ so that $(U_{\mathsf{H}}(t))_{t\geq0}$ is stochastic.

\subsection{Exit time and integration formula} 
Let us now introduce the \textit{exit time} of particles in $\Omega$ (with the notations of \cite{mjm1}),
defined as:
\begin{defi}\label{tempsdevol}
For any $(x,v) \in \overline{\Omega} \times V,$ define
\begin{equation*}
t_{\pm}(x,v)=\inf\{\,s > 0\,;\,x\pm sv \notin \Omega\}.
\end{equation*}
To avoid confusion, we will set $\tau_{\pm}(x,v):=t_{\pm}(x,v)$  if $(x,v) \in
\partial \Omega \times V.$
\end{defi}
 
With the notations of \cite{guo03}, $t_{-}$ is the \emph{backward exit time} $t_{\mathbf{b}}$.
From a heuristic viewpoint, $t_{-}(x,v)$ is the time needed by a
particle having the position $x \in \Omega$ and the velocity $-v
\in V$ to reach the boundary $\partial\Omega$. One can prove
\cite[Lemma 1.5]{voigt} that $t_{\pm}(\cdot,\cdot)$ is measurable on
$\Omega \times V$. Moreover $\tau_{\pm}(x,v)=0 \text{ for any } (x,v)
\in \Gamma_{\pm}$ whereas $\tau_{\mp}(x,v)> 0$ on $\Gamma_{\pm}.$ It holds
$$(x,v) \in \Gamma_{\pm} \Longleftrightarrow \exists y \in \Omega \quad \text{ with } \quad t_{\pm}(y,v) < \infty \quad \text{ and }\quad x=y\pm t_{\pm}(y,v)v.$$
In that case, $\tau_{\mp}(x,v)=t_{\pm}(y,v).$  Notice also that,
\begin{equation}\label{eq:scale}
t_{\pm}(x,v)|v|=t_{\pm}\left(x,\omega\right), \qquad \forall (x,v) \in \overline{\Omega} \times V, \:v \neq 0, \:\omega=|v|^{-1}\,v \in \mathbb{S}^{d-1}.\end{equation}
We have the following integration formulae from \cite{mjm1}.
\begin{propo} For any $h \in X$, it holds
\begin{equation}\label{10.47}
\int_{\Omega \times V}h(x,v)\d x \otimes \bm{m}(\d v)
=\int_{\Gamma_\pm}\d\mu_{\pm}(z,v)\int_0^{\tau_{\mp}(z,v)}h\left(z\mp\,sv,v\right)\d s,
\end{equation}
and
for any $\psi \in L^1(\Gamma_-,\d\mu_{-})$,
\begin{equation}\label{10.51}
\int_{\Gamma_-}\psi(z,v)\d\mu_{-}(z,v)=\int_{\Gamma_+}\psi(x-\tau_{-}(x,v)v,v)\d\mu_{+}(x,v).\end{equation}
\end{propo}
\begin{nb} Notice that with the notations introduced in \cite{mjm1}, 
$$\Gamma_{\pm\infty}=\{(x,v) \in \Gamma_{\pm}\;;\;\tau_{\mp}(x,v)=\infty\}=\{(x,v) \in \Gamma_{\pm}\;;v=0\}$$
so that $\mu_{\pm}(\Gamma_{\pm\infty})=0.$ This explains why the above integration formulae do not involve the sets $\Gamma_{\pm\infty}$. 
Moreover, because $\mu_{-}(\Gamma_{0})=\mu_{+}(\Gamma_{0})=0$, we can extend the above identity \eqref{10.51} as follows: for any $\psi \in L^{1}(\Gamma_{-} \cup \Gamma_{0},\d\mu_{-})$ it holds
\begin{equation}\label{10.52}
\int_{\Gamma_{-}\cup \Gamma_{0}}\psi(z,v)\d\mu_{-}(z,v)=\int_{\Gamma_{+}\cup\Gamma_{0} }\psi(x-\tau_{-}(x,v)v,v)\d\mu_{+}(x,v).\end{equation}
\end{nb}

\subsection{About the resolvent of $\mathsf{T}_{\mathsf{H}}$}

For any $\lambda \in \mathbb{C}$ such that $\mathrm{Re}\lambda
> 0$, define
\begin{equation*}
\begin{cases}
\mathsf{M}_{\lambda} \::\:&L^1_- \longrightarrow L^1_+\\
&u \longmapsto
\mathsf{M}_{\lambda}u(x,v)=u(x-\tau_{-}(x,v)v,v)e^{-\lambda\tau_{-}(x,v)},\:\:\:(x,v) \in \Gamma_+\;;
\end{cases}
\end{equation*}

\begin{equation*}
\begin{cases}
\mathsf{\Xi}_{\lambda} \::\:&L^1_- \longrightarrow X\\
&u \longmapsto \mathsf{\Xi}_{\lambda}u(x,v)=u(x-t_{-}(x,v)v,v)e^{-\lambda
t_{-}(x,v)}\ind_{\{t_{-}(x,v) < \infty\}},\:\:\:(x,v) \in \Omega \times V\;;
\end{cases}
\end{equation*}

\begin{equation*}
\begin{cases}
\mathsf{G}_{\lambda} \::\:&X \longrightarrow L^1_+\\
&\varphi \longmapsto \mathsf{G}_{\lambda}\varphi(x,v)=\displaystyle
\int_0^{\tau_{-}(x,v)}\varphi(x-sv,v)e^{-\lambda s}\d s,\:\:\:(x,v) \in
\Gamma_+\;;\end{cases}
\end{equation*}
and
\begin{equation*}
\begin{cases}
\mathsf{R}_{\lambda} \::\:&X \longrightarrow X\\
&\varphi \longmapsto \mathsf{R}_{\lambda}\varphi(x,v)=\displaystyle
\int_0^{t_{-}(x,v)}\varphi(x-tv,v)e^{-\lambda t}\d t,\:\:\:(x,v) \in
\Omega\times V;
\end{cases}
\end{equation*}
where $\ind_E$ denotes the charateristic function of the measurable
set $E$. All these operators are
bounded on
their respective spaces. More precisely, for any $\mathrm{Re}\lambda
> 0$
\begin{align*}
\|\mathsf{M}_{\lambda}\| &\leq 1, & \|\mathsf{\Xi}_{\lambda}\| &\leq (\mathrm{Re}
\lambda
)^{-1},\\
\|\mathsf{G}_{\lambda}\| &\leq (\,\mathrm{Re} \lambda )^{-1}, &
\|\mathsf{R}_{\lambda}\| &\leq (\mathrm{Re }\lambda)^{-1}.
\end{align*}
The interest of these operators is related to the resolution of the boundary
value problem:
\begin{equation}\label{BVP1}
\begin{cases}
(\lambda- \mathsf{T}_{\mathrm{max}})f=g,\\
\mathsf{B}^-f=u,
\end{cases}
\end{equation}
where $\lambda > 0$, $g \in X$ and $u$ is a given function over
$\Gamma_-.$ Such a boundary value problem, with $u \in \lm$ can be uniquely solved (see \cite{mjm1})
\begin{theo}\label{Theo4.2} Given  $\lambda >0$, $u \in \lm$ and $g \in X$, the function
$$f=\mathsf{R}_{\lambda}g + \mathsf{\Xi}_{\lambda}u$$ is the \textbf{
unique} solution $f \in \D(\mathsf{T}_{\mathrm{max}})$ of the boundary value
problem \eqref{BVP1}.\end{theo}

\begin{nb}\label{nb:lift} Notice that
$\mathsf{\Xi}_{\lambda}$ is a lifting operator which, to a given $u \in
\lm$, associates a function $f=\mathsf{\Xi}_{\lambda}u \in
\D(\mathsf{T}_{\mathrm{max}})$ whose trace on $\Gamma_-$ is exactly $u$.
More precisely,
\begin{equation}\label{propxil1}\T_\mathrm{max}\mathsf{\Xi}_{\lambda}u=\lambda \mathsf{\Xi}_\lambda u, \qquad
\mathsf{B}^-\mathsf{\Xi}_\lambda u=u,\:\:\mathsf{B}^+\mathsf{\Xi}_\lambda u = \ml u, \qquad  \forall u
\in \lm.
\end{equation}
\end{nb}
We can complement the above result with the following whose proof can be extracted from \cite[Theorem 4.2]{mjm2}:
\begin{propo}\phantomsection\label{propo:resolvante}
If $r_{\sigma}(\mathsf{M}_{\lambda}\mathsf{H})< 1$   $(\lambda >0)$, then $A=\mathsf{T}_{\mathsf{H}}$ and
$$\mathcal{R}(\lambda,\mathsf{T}_{\mathsf{H}})=\mathsf{R}_{\lambda}+\mathsf{\Xi}_{\lambda}\mathsf{H}\Rs(1,\mathsf{M}_{\lambda}\mathsf{H})\mathsf{G}_{\lambda}$$
where the series  converges in $\mathscr{B}(X)$.
\end{propo}
 
\subsection{Some auxiliary operators}\label{sec:prelim}

For $\lambda=0$, we can extend the definition of these operators in an obvious way but not all the resulting operators are bounded in their respective spaces. However, we see from  the above integration formula \eqref{10.51}, that 
$$\mathsf{M}_{0} \in \mathscr{B}(\lm,\lp) \qquad \text{ with } \quad \|\mathsf{M}_{0}u\|_{\lp}=\|u\|_{\lm}, \qquad \forall u \in \lm.$$
In the same way, 
one deduces from   \eqref{10.47} that for any nonnegative $\varphi \in X$:
\begin{equation}\label{Eq:G0}
\int_{\Gamma_{+}}\mathsf{G}_{0}\varphi(x,v)\d\mu_{+}(x,v)=\int_{\Gamma_{+}}\d\mu_{+}(x,v)\int_{0}^{\tau_{-}(x,v)}\varphi(x-sv,v)\d s=\int_{\Omega \times V}\varphi(x,v)\d x \otimes \bm{m}(\d v)\end{equation}
which proves that 
$$\mathsf{G}_{0} \in \mathscr{B}(X,\lp) \qquad \text{ with } \quad  \|\mathsf{G}_{0}\varphi\|_{\lp}=\|\varphi\|_{X}, \qquad \forall \varphi\in X.$$
To be able to provide a rigorous definition of the operators $\mathsf{\Xi}_{0}$ and $\mathsf{R}_{0}$ we need the following
\begin{defi} Introduce the function spaces
$$\Y^{\pm}_{1}=L^{1}(\Gamma_{\pm}\,,|v|^{-1} \d\mu_{\pm})$$
with its associated $L^{1}$-norm $\|\cdot\|_{\Y^{\pm}_{1}}$ and
$$\X_{\tau}=L^{1}(\Omega\times V,t_{+}(x,v)\d x \otimes \bm{m}(\d v))$$ 
with the associated $L^{1}$-norm $\|\cdot\|_{\tau}$.\end{defi}
The interest of the above boundary spaces lies in the following:
\begin{lemme}\phantomsection\label{lem:motau} 
For any $u \in \Y^{-}_{1}$  one has $\mathsf{\Xi}_{0}u \in X$ with 
\begin{equation}\label{eq:Xio}
\|\mathsf{\Xi}_{0}u\|_{X}=\int_{\Gamma_{-}}u(x,v)\tau_{+}(x,v)\d\mu_{+}(x,v) \leq D\|u\|_{\Y^{-}_{1}}, \qquad \forall u \in \Y^{-}_{1}\end{equation}
where $D$ is the diameter of $\Omega$, $D=\sup_{x,y \in \pO}|x-y|$.  
Moreover, if $u \in \Y^{-}_{1}$ then $\mathsf{M}_{0}u \in \Y^{+}_{1}$ with
\begin{equation}\label{eq:M0+Y}
\|\mathsf{M}_{0}u\|_{\Y^{+}_{1}}= \|u\|_{\Y^{-}_{1}}.\end{equation}
If $f \in \X_{\tau}$ then $\mathsf{G}_{0}f \in \Y^{-}_{1}$  and $\mathsf{R}_{0} f \in \D(\mathsf{T}_{0}) \subset X$ and $\mathsf{T}_{0}\mathsf{R}_{0}f=-f$.
\end{lemme}
\begin{proof} From \eqref{10.47}, for nonnegative $u \in \lm:$
\begin{equation*}\begin{split}
\int_{\Omega \times V}&\mathsf{\Xi}_{0}u(x,v)\d x \otimes \bm{m}(\d v)=\int_{\Omega \times V}\mathsf{\Xi}_{0}u(x,v)\d x \otimes \bm{m}(\d v)\\
&=\int_{\Gamma_{+}}\d\mu_{+}(z,v)\int_{0}^{\tau_{-}(z,v)}u(z-sv-t_{-}(z-sv,v)v,v)\ind_{\{t_{-}(z-sv,v)<\infty\}}\d s\\
&=\int_{\Gamma_{+}}u(z-\tau_{-}(z,v)v,v)\tau_{-}(z,v)\d\mu_{+}(z,v)
\end{split}\end{equation*}
which, using now \eqref{10.51} yields \eqref{eq:Xio}.
If now $u \in \Y^{-}_{1}$, then
$$\int_{\Gamma_{+}}\mathsf{M}_{0}u(x,v)\,|v|^{-1}\d\mu_{+}(x,v)=\int_{\Gamma_{+}}u(x-\tau_{-}(x,v)v,v)|v|^{-1}\d\mu_{+}(x,v)$$
and  we deduce from \eqref{10.51} that
$$\int_{\Gamma_{+}}\mathsf{M}_{0}u(x,v)|v|^{-1}\d\mu_{+}(x,v)=\int_{\Gamma_{-}}u(z,v)|v|^{-1}\d\mu_{-}(z,v)$$
which is \eqref{eq:M0+Y}. If now $f \in \X_{\tau}$ is nonnegative, one has directly from \eqref{10.47} that
$$\int_{\Gamma_{+}}\,|v|^{-1}\mathsf{G}_{0}f(x,v)\d \mu_{+}(x,v)=\|\,|v|^{-1}f\,\|_{X} \leq D\|f\|_{\tau} < \infty$$
which proves that $\mathsf{G}_{0}f \in X.$ Moreover, using \eqref{10.47},
\begin{equation*}\begin{split}
\int_{\Omega \times V}\mathsf{R}_{0}f(x,v)\d x \otimes \bm{m}(\d v)&=\int_{\Gamma_{+}}\d\mu_{+}(z,v)\int_{0}^{\tau_{-}(z,v)}[\mathsf{R}_{0}f](z-sv,v)\d s\\
&=\int_{\Gamma_{+}}\d\mu_{+}(z,v)\int_{0}^{\tau_{-}(z,v)}\d s \int_{0}^{t_{-}(z-sv,v)}f(z-sv-tv,v)\d t
\end{split}\end{equation*}
and, since $t_{-}(z-sv,v)=\tau_{-}(z,v)-s$ for all $(z,v) \in\Gamma_{+}$ and all $0<s<\tau_{-}(z,v)$ we get
\begin{equation*}\begin{split}
\int_{\Omega \times V}\mathsf{R}_{0}f(x,v)\d x \otimes \bm{m}(\d v)&=\int_{\Gamma_{+}}\d\mu_{+}(z,v)\int_{0}^{\tau_{-}(z,v)}\d s\int_{s}^{\tau_{-}(z,v)}f(z-tv,v)\d t\\
&=\int_{\Gamma_{+}}\d\mu_{+}(z,v)\int_{0}^{\tau_{-}(z,v)}t\,f(z-tv,v)\d t\end{split}\end{equation*}
Now, since $t_{+}(z-tv,v)=t$ for any  $(z,v) \in \Gamma_{+}$, the above reads
$$\int_{\Omega \times V}\mathsf{R}_{0}f(x,v)\d x \otimes \bm{m}(\d v)=\int_{\Gamma_{+}}\d\mu_{+}(z,v)\int_{0}^{\tau_{-}(z,v)}t_{+}(z-tv,v)\,f(z-tv,v)\d t$$
and, using again \eqref{10.47}, one gets 
$$\int_{\Omega \times V}\mathsf{R}_{0}f(x,v)\d x \otimes \bm{m}(\d v)=\int_{\Omega \times V}t_{+}(x,v)f(x,v)\d x \otimes \bm{m}(\d v).$$
This proves that $\mathsf{R}_{0}f \in X$. Now, it is easy to see that actually $g=\mathsf{R}_{0}f$ satisfies $\mathsf{T}_{\mathrm{max}}g=-f$ and $\mathsf{B}^{-}g=0$, i.e. $g \in \D(\mathsf{T}_{0})$ with $\mathsf{T}_{0}g=-f.$
\end{proof}
\begin{nb} Notice that, for any nonnegative $u \in \lp$, 
$$\int_{\Gamma_{+}}\mathsf{M}_{0}u(x,v)\tau_{-}(x,v)\d\mu_{+}(x,v)=\int_{\Gamma_{+}}u(x-\tau_{-}(x,v)v,v)\tau_{-}(x,v)\d\mu_{+}(x,v)$$
and, since $\tau_{+}(x-\tau_{-}(x,v)v,v)=\tau_{-}(x,v)$ for any $(x,v) \in\Gamma_{-}$, we deduce from \eqref{10.51} that
$$\int_{\Gamma_{+}}\mathsf{M}_{0}u(x,v)\tau_{-}(x,v)\d\mu_{+}(x,v)=\int_{\Gamma_{-}}u(z,v)\tau_{+}(z,v)\d\mu_{-}(z,v)$$
This shows that, in \eqref{eq:M0+Y}, we can replace $\Y^{\pm}_{1}$ with $L^{1}(\Gamma_{\pm},\tau_{\mp}(x,v)\d\mu_{\pm}(x,v)).$ In the same way, one see that, for $g \in \X_{\tau}$ it holds $\|\mathsf{G}_{0}g\|_{L^{1}(\Gamma_{-},\tau_{+}\d\mu_{-})}=\|g\|_{\tau}.$\end{nb}


One has the following  result:
\begin{propo}\phantomsection\label{propBV0} Let $g \in\X_{\tau}$ be given and $u \in L^{1}(\Gamma_{-},\tau_{+}(x,v)\d\mu_{-})$. The boundary value problem \begin{equation}\label{BV0}
\begin{cases}
-\mathsf{T}_{\mathrm{max}}f=g\\
\mathsf{B}^{-}f=u \end{cases}\end{equation}
admits a unique solution $f \in X$ given by $f=\mathsf{R}_{0}g+\mathsf{\Xi}_{0}u.$
\end{propo}
\begin{proof} Let $u \in L^{1}(\Gamma_{-},\tau_{+}(x,v)\d\mu_{-})$ and $g \in\X_{\tau}$, since $\mathsf{\Xi}_{0} u \in \D(\mathsf{T}_{\mathrm{max}})$
with  $\mathsf{T}_{\mathrm{max}}\mathsf{\Xi}_{0} u = 0,$ $\mathsf{B}^{+}\mathsf{\Xi}_{0} u =\mathsf{M}_{0}u$
and $\mathsf{B}^{-}\mathsf{\Xi}_{0}u=u$ one sees that $f \in\D(\mathsf{T}_{\mathrm{max}})$ with 
$$\mathsf{T}_{\mathrm{max}}f=\mathsf{T}_{\mathrm{max}}\mathsf{R}_{0}g+\mathsf{T}_{\mathrm{max}}\mathsf{\Xi}_{0}u=\mathsf{T}_{0}\mathsf{R}_{0}g=-g$$
while $\mathsf{B}^{-}f=\mathsf{B}^{-}\mathsf{R}_{0}g+\mathsf{B}^{-}\mathsf{\Xi}_{0}u=\mathsf{B}^{-}\mathsf{\Xi}_{0}u=u.$ This shows that $f=\mathsf{R}_{0}g+\mathsf{\Xi}_{0}u$ is a solution to \eqref{BV0}. To prove the uniqueness, it suffices to assume that $g=u=0$ but then \eqref{BV0} reads $\mathsf{T}_{0}f=0$ which admits the unique solution $f=0$.  \end{proof}

\section{General stochastic partly diffuse boundary conditions}\label{sec:BC}

Let us explicit here the general class of boundary conditions we aim to deal with.  Typical  boundary operators arising in the kinetic theory of gases are {\it local with respect to} $x \in
\partial \Omega$.  
In order to exploit this local nature of the boundary conditions, we introduce the following notations. For any $x \in \partial \Omega$, we define
$$\Gamma_{\pm}(x)=\{v \in V\;;\; \pm v \cdot n(x) > 0\}, \qquad \Gamma_{0}(x)=\{v \in V\;;\; v \cdot n(x)=0\}$$
and we define the measure $\bm{\mu}_{x}(\d v)$ on $\Gamma_{\pm}(x)$ given by
$$\bm{\mu}_{x}(\d v)=|v\cdot n(x)|\bm{m}(\d v).$$
This allows to define the $L^{1}$-space $L^{1}(\Gamma_{\pm}(x),\d\bm{\mu}_{x})$ in an obvious way. We shall denote the $L^{1}(\Gamma_{\pm}(x),\bm{\mu}_{x})$ norm by $\|\cdot\|_{L^{1}(\Gamma_{\pm}(x))}.$ 
Since, for any $\varphi \in L^{1}(\Gamma_{\pm},\mu_{\pm})$ one has
$$\|\varphi\|_{L^{1}_{\pm}}=\int_{\partial\Omega}\left[\int_{\Gamma_{\pm}(x)}|\varphi(x,v)|\bm{\mu}_{x}(\d v)\right]\pi(\d x)=\int_{\partial\Omega}\|\varphi(x,\cdot)\|_{L^{1}(\Gamma_{\pm}(x))}\pi(\d x)$$
we can identify isometrically any $\varphi \in L^{1}_{\pm}$ to the \emph{field}
\begin{equation}\label{eq:field}
x \in \partial \Omega \longmapsto \varphi(x,\cdot) \in L^{1}(\Gamma_{\pm}(x)).\end{equation}

\subsection{Reflection boundary operators}

We begin with the following definition of pure reflection boundary conditions (see \cite[Definition 6.1, p.104]{voigt}):

\begin{defi}\phantomsection\label{defi:reflec} One says that $\mathsf{R} \in \mathscr{B}(\lp,\lm)$ is a pure reflection
boundary operator if 
$$\mathsf{R}(\varphi)(x,\v)=\varphi(x,\mathcal{V}(x,\v)) \qquad
\qquad \forall (x,\v) \in \Gamma_-, \:\varphi \in \lp$$
where $\mathcal{V}\::\:x \in \partial \Omega \mapsto \mathcal{V}(x,\cdot)$ 
is a field of bijective bi-measurable and $\bm{\mu}_{x}$-preserving mappings
$$\mathcal{V}(x,\cdot)\::\:\Gamma_{-}(x) \cup \Gamma_{0}(x) \to \Gamma_{+}(x) \cup \Gamma_{0}(x)$$
such that
\begin{enumerate}[i)]
\item $|\mathcal{V}(x,\v)|=|\v|$ for any $(x,\v) \in \Gamma_-$.
\item If $(x,v) \in \Gamma_{0}$ then $(x,\mathcal{V}(x,v)) \in \Gamma_{0}$, i.e. $\mathcal{V}(x,\cdot)$ maps $\Gamma_{0}(x)$ in $\Gamma_{0}(x).$
\item The mapping 
$$(x,\v) \in \Gamma_- \mapsto (x,\mathcal{V}(x,\v)) \in
\Gamma_+$$
is a $\mathcal{C}^{1}$ diffeomorphism.
\end{enumerate}
\end{defi}
\begin{nb} This last regularity property on $\mathcal{V}$ may require additional regularity of $\pO$ as seen in Example \ref{exe:specu}.\end{nb}

Note that 
$$[\mathsf{R}\varphi](x,v)=[\mathsf{R}(x)\varphi(x,\cdot)](v), \qquad \forall \varphi \in \lp,\:\:(x,v) \in \Gamma_{-}$$
where we identify (isometrically) $\varphi \in \lp$ to the integrable field \eqref{eq:field} and 
$$x \in \partial\Omega \mapsto \mathsf{R}(x) \in \mathscr{B}(L^{1}(\Gamma_{+}(x),\Gamma_{-}(x))$$
is the field of operators defined by
$$\mathsf{R}(x)\psi(v)=\psi(\mathcal{V}(x,v)), \qquad \psi \in L^{1}(\Gamma_{+}(x)), \quad v \in \Gamma_{-}(x).$$
It holds
$$\|\mathsf{R}(x)\psi\|_{L^{1}(\Gamma_{-}(x))}=\|\psi\|_{L^{1}(\Gamma_{+}(x))}, \qquad \forall x \in \partial\Omega, \quad \psi \in L^{1}(\Gamma_{+}(x)),$$
therefore, $\mathsf{R} \in \mathscr{B}(\lp,\lm)$ is stochastic since 
$$\|\mathsf{R}\varphi\|_{\lm}=\int_{\partial\Omega}\left\|\mathsf{R}(x)\varphi(x,\cdot)\right\|_{L^{1}(\Gamma_{-}(x))}\pi(\d x)=\int_{\partial\Omega}\|\varphi(x,\cdot)\|_{L^{1}(\Gamma_{+}(x))}\pi(\d x)=\|\varphi\|_{\lp}.$$
Notice that this last identity is equivalent to the property that the mapping 
$$(x,v) \in \Gamma_{-} \longmapsto (x,\mathcal{V}(x,v)) \in \Gamma_{+}$$ is $\mu$-preserving.

\begin{exe}\phantomsection\label{exe:specu} In practical situations, the most frequently used
pure reflection conditions are
\begin{enumerate}[(a)]
\item the {\it specular reflection boundary
conditions} which corresponds to the case in which $V$ and $\bm{m}$ are invariant under the orthogonal group and $$\mathcal{V}(x,\v)=\v-2(\v \cdot
n(x))\,n(x) \qquad \qquad (x,\v) \in \Gamma_-.$$ 
Notice that, for $\mathcal{V}$ to be a $\mathcal{C}^{1}$ diffeormorphism, we need $\pO$ to be of class $\mathcal{C}^{2}.$
\item The {\it bounce--back
reflection conditions} for which $\mathcal{V}(x,\v)=-\v$, $(x,\v)
\in \Gamma_-$.\end{enumerate}\end{exe}


\begin{figure}
\centerline{\includegraphics{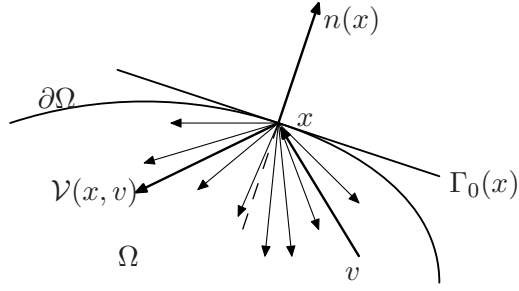}}
\centerline{
\begin{picture}(0,0)(0,0)
\put(84,48){$\Gamma_0(x)$}
\put(27,73){$x$}
\put(37,110){$n(x)$}
\put(-70,80){$\partial\Omega$}
\put(-40,20){$\Omega$}
\put(45,15){$v$}
\put(-65,45){$\mathcal{V}(x,v)$}
\end{picture}
}
\caption{Regular and diffuse reflection: 
$v$ -- an outward vector, $\mathcal{V}(x,v)$ -- the specular reflection, thin vectors -- diffuse reflection.}
\label{p:col-2}
\end{figure}

\subsection{Diffuse boundary operators}\phantomsection\label{sec:diffus} We introduce the following definition
\begin{defi}\phantomsection\label{defi:Hdiff} One says that $\mathsf{K} \in \mathscr{B}(\lp,\lm)$ is a \emph{stochastic diffuse boundary operator} if
\begin{equation}\label{eq:Hdif}
\mathsf{K}\,\psi(x,v)=\int_{\Gamma_{+}(x)}k(x,v,v')\psi(x,v')\bm{\mu}_{x}(\d v'), \qquad (x,v) \in \Gamma_{-}, \quad \psi \in \lp\end{equation}
where the kernel $k(x,v,v')$ induces a field of nonnegative measurable functions
$$x \in \partial\Omega \mapsto k(x,\cdot,\cdot)$$
where 
$$k(x,\cdot,\cdot)\::\:\Gamma_{-}(x) \times \Gamma_{+}(x) \to \R^{+}$$
is such that
$$\int_{\Gamma_{-}(x)}k(x,v,v')\bm{\mu}_{x}(\d v)=1, \qquad \forall (x,v') \in \Gamma_{+}.$$
\end{defi}
As we did for reflection operators, we identify $\mathsf{K} \in \mathscr{B}(\lp,\lm)$ to a field of integral operators
$$x \in \partial\Omega \longmapsto \mathsf{K}(x) \in \mathscr{B}(L^{1}(\Gamma_{+}(x),\Gamma_{-}(x))$$
by the formula
$$[\mathsf{K}\psi](x,v)=\left[\mathsf{K}(x)\psi(x,\cdot)\right](v)$$
where, for any $x \in \partial\Omega$
$$\mathsf{K}(x)\::\:\psi \in L^{1}(\Gamma_{+}(x)) \longmapsto \left[\mathsf{K}(x)\psi\right](v)=\int_{\Gamma_{+}(x)}k(x,v,v')\psi(v')\bm{\mu}_{x}(\d v') \in L^{1}(\Gamma_{-}(x)).$$
Note that $\mathsf{K}(x)\::\:L^{1}(\Gamma_{+}(x)) \to L^{1}(\Gamma_{-}(x))$ is stochastic for any $x \in \partial \Omega$ and therefore so is $\mathsf{K} \in \mathscr{B}(\lp,\lm)$, i.e.
$$\|\mathsf{K}\psi\|_{\lm}=\|\psi\|_{\lp} \qquad \forall \psi \in \lp.$$
We introduce now a useful class of diffuse boundary operators. Before giving the formal definition, let us recall that, if  $\mathsf{K} \in \mathscr{B}(\lp,\lm)$  given by \eqref{eq:Hdif} is such that 
\begin{equation}\label{eq:HxWC}
\mathsf{K}(x) \in \mathscr{B}(L^{1}(\Gamma_{+}(x)),L^{1}(\Gamma_{-}(x))) \quad \text{ is weakly compact  for any } x \in \partial\Omega\end{equation}
then, according to the Dunford-Pettis criterion (see \cite[Theorem 4.30, p. 115 \& Exercise 4.36, p. 129]{brezis}), for any $x \in \pO$ and any $\varepsilon >0$, there is $\delta >0$ such that
\begin{equation*}\label{eq:uniformintegr}\sup_{v'\in \Gamma_{+}(x)}\int_{A}k(x,v,v')\bm{\mu}_{x}(\d v) < \varepsilon \qquad \forall A \subset \Gamma_{-}(x) \text{ such that }  \bm{\mu}_{x}(A) < \delta\end{equation*}
and 
\begin{equation*}\label{eq:nonconce}
\lim_{m \to \infty}\sup_{v'\in \Gamma_{+}(x)}\int_{\Gamma_{-}(x) \setminus A_{m}}k(x,v,v')\bm{\mu}_{x}(\d v) =0\end{equation*}
for any sequence $(A_{m})_{m} \subset \Gamma_{-}(x)$ with $A_{m} \subset A_{m+1}$, $\bm{\mu}_{x}(A_{m}) <\infty$ and $\cup_{m}A_{m}=\Gamma_{-}(x).$ In particular, for any $x \in \partial \Omega$, 
$$\lim_{m \to \infty}\sup_{v'\in \Gamma_{+}(x)}\int_{\{v \in \Gamma_{-}(x)\;;\;|v| \geq m\}}k(x,v,v')\,\bm{\mu}_{x}(\d v)=0.$$
%
Moreover, since 
\begin{equation*}\begin{split}
1&=\int_{\Gamma_{-}(x)}k(x,v,v')\bm{\mu}_{x}(\d v) \geq \int_{\{v \in \Gamma_{-}(x)\;;\;k(x,v,v') \geq m\}}k(x,v,v')\bm{\mu}_{x}(\d v)\\
&\geq m\,\bm{\mu}_{x}\left(\{v \in \Gamma_{-}(x)\;;\;k(x,v,v') \geq m\}\right), \qquad \forall m \in \N, \quad (x,v') \in \Gamma_{+},\end{split}\end{equation*}
we have
\begin{equation*}
\lim_{m\to\infty}\sup_{v' \in \Gamma_{+}(x)}\bm{\mu}_{x}\left(\{v \in \Gamma_{-}(x)\;;\;k(x,v,v') \geq m\}\right)=0.\end{equation*}
In other words, for any $x \in \partial \Omega$, the following holds
\begin{equation}\label{eq:Sm}
\lim_{m\to\infty}\sup_{v'\in \Gamma_{+}(x)}\int_{S_{m}(x,v')}k(x,v,v')\,\bm{\mu}_{x}(\d v)=0\end{equation}
where, for any $m \in \N$ and any $(x,v') \in \Gamma_{+}$
 $$S_{m}(x,v')=\{v \in \Gamma_{-}(x)\;;\;|v| \geq m\} \cup \{v \in \Gamma_{-}(x)\;;\;k(x,v,v') \geq m\}.$$

We introduce then the following class of diffuse boundary operators:
\begin{defi}\phantomsection \label{defi:regul}
We say that a diffuse boundary operator $\mathsf{K} \in \mathscr{B}(\lp,\lm)$ is \emph{regular} if the family of operators 
$$\mathsf{K}(x) \in \mathscr{B}(L^{1}(\Gamma_{+}(x)),L^{1}(\Gamma_{-}(x))), \qquad x \in \partial\Omega$$
is collectively weakly compact in the sense that \eqref{eq:HxWC} holds true for any $x \in \partial\Omega$ and the convergence in \eqref{eq:Sm} is \emph{uniform} with respect to $x \in \partial\Omega$. 
\end{defi}
\begin{nb} A diffuse boundary operator $\mathsf{K}$ is regular for instance whenever there
exists $h:V\rightarrow \R^{+}$ such that $\int_{V}h(v)\left\vert v\right\vert \bm{m}(\d v)<+\infty$ and 
$$ k(x,v,v')\leq h(v)\qquad \quad \forall x\in \partial \Omega ,\ v^{\prime
}\in \Gamma _{+}(x),\ v\in \Gamma _{-}(x).$$
In particular, the classical Maxwell boundary operator  (see Example \ref{exe:maxwell} below) is a regular diffuse boundary operator.
\end{nb}
We have then the following approximation result. 
\begin{lemme}\label{lem:regular} Assume that $\mathsf{K} \in \mathscr{B}(L^{1}_{+},L^{1}_{-})$ is a regular diffuse boundary operator in the sense of the above definition. Then, there exists a sequence $(\mathsf{K}_{m})_{m} \subset \mathscr{B}(L^{1}_{+},L^{1}_{-})$ such that
\begin{enumerate}
\item $0 \leq \mathsf{K}_{m} \leq \mathsf{K}$ for any $m \in \mathbb{N}$;
\item $\lim_{m\to\infty}\|\mathsf{K}-\mathsf{K}_{m}\|_{\mathscr{B}(L^{1}_{+},L^{1}_{-})}=0;$
\item For any $m \in \mathbb{N}$ and any nonnegative $f \in \lp$ it holds 
\begin{equation}\label{eq:Hmpsim}
\mathsf{K}_{m}f(x,v) \leq \psi_{m}(v)\int_{\Gamma_{+}(x)}f(x,v')\,|v'\cdot n(x)|\,\bm{m}(\d v'), \qquad (x,v) \in \Gamma_{-}\end{equation}
with $\psi_{m}=m\ind_{B_{m}}$ where $B_{m}=\{v \in \R^{d}\;;\;|v| \leq m\}$.
\end{enumerate}
\end{lemme} 
\begin{proof} Let $k(x,v,v')$ be the kernel associated to $K$ through \eqref{eq:Hdif}. Introduce then $k_{m}(x,v,v')=\inf \{k(x,v,v')\;;\;m\ind_{B_{m}}(v)\}$ for any $m \in \N$, where $B_{m}$ is the ball of $\R^{d}$ centered in $0$ and with radius $m$, and set 
$$\mathsf{K}_{m}\varphi(x,v)=\int_{\Gamma_{+}(x)} k_{m}(x,v,v')\varphi(x,v')\,|v'\cdot n(x)|\,\bm{m}(\d v'), \qquad \varphi \in \lp,\qquad (x,v) \in \Gamma_{-}.$$
Clearly, $\mathsf{K}_{m} \in \mathscr{B}(\lp,\lm)$ is a diffuse boundary operator with $0 \leq \mathsf{K}_{m} \leq K$ and \eqref{eq:Hmpsim} holds.  Moreover, for any $x \in \partial\Omega$ and any $\varphi \in L^{1}(\Gamma_{+}(x))$, it is easy to check 
 that 
\begin{multline*}
\|\mathsf{K}(x)\varphi-\mathsf{K}_{m}(x)\varphi\|_{L^{1}(\Gamma_{-}(x))} \leq 
\|\varphi\|_{L^{1}(\Gamma_{+}(x))}\\
\times\sup_{v'\in \Gamma_{+}(x)}\int_{\{v \in \Gamma_{-}(x)\;;\,k(x,v,v') \geq m\ind_{B_{m}}(v)\}}k(x,v,v')\bm{\mu}_{x}(\d v),\end{multline*}
i.e.
\begin{equation*}\begin{split}
\|\mathsf{K}(x) -\mathsf{K}_{m}(x)\|_{\mathscr{B}(L^{1}(\Gamma_{+}(x)),L^{1}(\Gamma_{-}(x)))} &\leq \sup_{v'\in \Gamma_{+}(x)}\int_{\{v \in \Gamma_{-}(x)\;;\,k(x,v,v') \geq m\ind_{B_{m}}(v)\}}k(x,v,v')\bm{\mu}_{x}(\d v)\\
&\leq \sup_{v'\in \Gamma_{+}(x)}\int_{S_{m}(x,v')}k(x,v,v')\bm{\mu}_{x}(\d v).\end{split}\end{equation*}
One sees then that
$$\|\mathsf{K}-\mathsf{K}_{m}\|_{\mathscr{B}(\lp,\lm)}=\sup_{x \in \partial\Omega}\|\mathsf{K}(x) -\mathsf{K}_{m}(x)\|_{\mathscr{B}(L^{1}(\Gamma_{+}(x)),L^{1}(\Gamma_{-}(x)))}$$
goes to zero as $m\to \infty$ since the convergence in \eqref{eq:Sm} is uniform with respect to $x \in \partial\Omega.$ \end{proof}

We complement the above result with a different kind of approximation which will turn useful in Section \ref{ss:sweeping}:
\begin{lemme}\phantomsection
\label{lemme-1-add-sweeping} 
Let $\mathsf{K}$ be a regular \emph{stochastic} diffuse boundary operator
with kernel $k(x,v,v^{\prime })$. Let 
$$\bm{\beta}_{n}(x,v^{\prime }) 
=\int_{\Gamma _{-}(x) \cap \{|v| >\frac{1}{n}\}}k(x,v,v^{\prime })\bm{\mu}_{x}(\d v), \qquad (x,v') \in \Gamma_{+}$$
and
$$\bm{k}_{n}(x,v,v^{\prime })=\frac{k(x,v,v^{\prime })}{\bm{\beta}_{n}(x,v')}\ind_{\left\{ \left\vert v\right\vert
>\frac{1}{n}\right\} }, \qquad x \in \pO, v' \in \Gamma_{+}(x), \:v \in \Gamma_{-}(x)$$
Then, denoting by $\mathsf{K}_{n}$ the regular stochastic diffuse boundary operator with
kernel  $\bm{k}_{n}$, it holds
\begin{enumerate}
\item[(i)] $\lim_{n\rightarrow +\infty }\bm{\beta}_{n}(x,v^{\prime })=1$ uniformly in $%
(x,v^{\prime })\in \Gamma _{+}.$
\item[(ii)] $\lim_{n\to\infty}\left\Vert \mathsf{K}_{n}-\mathsf{K}\right\Vert _{\mathscr{B}(\lp,\lm)}=0.$
\end{enumerate}
\end{lemme}
\begin{proof} (i) For any $x \in \partial\Omega$, set ${A}_{n}(x)=\left\{v \in \Gamma_{-}(x)\,;\,|v|  < \tfrac{1}{n}\right\}.$ One has
$$\bm{\mu}_{x}\left(A_{n}(x)\right)=\int_{\Gamma_{-}(x) \cap \{|v| < n^{-1}\}}|v\cdot n(x)| \bm{m}(\d v) \leq  \dfrac{\bm{m}(B_{1})}{n}$$
where $B_{1}$ is the unit ball of $\R^{d}$. Thus, \begin{equation}\label{eq:Anx}
\lim_{n\to\infty}\sup_{x \in \partial\Omega}\bm{\mu}_{x}\left(A_{n}(x)\right)=0.\end{equation}
Now, since $\mathsf{K}$ is regular, \eqref{eq:Sm} holds uniformly with respect to $x \in \pO$, i.e. for any $\varepsilon >0$, there is $m \in \N$ large enough so that
$$\sup_{(x,v') \in \Gamma_{+}}\int_{S_{m}(x,v')}k(x,v,v')\bm{\mu}_{x}(\d v) < \varepsilon.$$
Then, for any $n \in \N$ and any $(x,v') \in \Gamma_{+}$,
\begin{multline*}
\int_{A_{n}(x)}k(x,v,v')\bm{\mu}_{x}(\d v)=\int_{A_{n}(x) \cap S_{m}(x,v')}k(x,v,v')\bm{\mu}_{x}(\d v)\\
+ \int_{A_{n}(x) \setminus S_{m}(x,v')}k(x,v,v')\bm{\mu}_{x}(\d v)
\leq \varepsilon + m\,\bm{\mu}_{x}(A_{n}(x))\end{multline*}
since $k(x,v,v') \leq m$ on $\Gamma_{-}(x) \setminus S_{m}(x,v').$ Using then \eqref{eq:Anx} we get
\begin{equation}\label{eq:kn}
\lim_{n\to\infty}\sup_{(x,v^{\prime}) \in\Gamma_{-}}\int_{\Gamma _{-}(x)\cap \left\{ \left\vert v\right\vert <n^{-1}\right\}
}k(x,v,v^{\prime })\bm{\mu}_{x}(\d v)= 0\end{equation}
which shows (i) since $%
\int_{\Gamma _{-}(x)}k(x,v,v^{\prime })\bm{\mu}_{x}(\d v)=1$ for any $x \in\partial\Omega.$

(ii) Set $\widehat{\mathsf{K}}_{n}$ the boundary operator with kernel $\ind_{\left\{ \left\vert v\right\vert
>n^{-1}\right\} }k(x,v,v').$ One checks easily that
$$\left\Vert \mathsf{K} -\widehat{\mathsf{K}}_{n} \right\Vert_{\mathscr{B}(\lp,\lm)}
\leq  \sup_{(x,v^{\prime })\in \Gamma _{-}}\int_{\Gamma _{-}(x)\cap
\left\{ \left\vert v\right\vert <n^{-1}\right\} }k(x,v,v^{\prime })\bm{\mu}
_{x}(\d v) $$
so that $\lim_{n}\left\| \mathsf{K} -\widehat{\mathsf{K}}_{n} \right\|_{\mathscr{B}(\lp,\lm)}=0$ from \eqref{eq:kn}. 
Since moreover
$$\left\Vert \mathsf{K}_{n}-\widehat{\mathsf{K}}_{n}\right\Vert_{\mathscr{B}(\lp,\lm)} \leq \sup_{(x,v')\in \Gamma_{+}}\left|1-\bm{\beta}_{n}(x,v')\right|\,\|\mathsf{K}\|_{\mathscr{B}(\lp,\lm)}$$
which goes to zero from point (i), we get the desired result.
\end{proof}

\subsection{Stochastic partly diffuse boundary operators} We introduce now the general class of boundary operator we aim at investigating.
\begin{defi}\label{ck} We shall say that a boundary operator $\mathsf{H} \in 
\mathscr{B}(\lp,\lm)$ is stochastic \emph{partly diffuse} if it writes
\begin{equation}\label{eq:partlydif}
\mathsf{H}\psi(x,\v)=\alpha(x)\,\mathsf{R}\psi (x,\v)+(1-\alpha(x))\,\mathsf{K}\psi(x,\v), \qquad  (x,\v) \in \Gamma_-, \psi \in \lp\end{equation} where $\alpha(\cdot)\::\:\partial \Omega \to [0,1]$  is measurable, $\mathsf{R}$ is a reflection operator, and $\mathsf{K} \in \mathscr{B}(\lp,\lm)$ is a \emph{stochastic} diffuse boundary operator given by \eqref{eq:Hdif}. 

If the diffuse part $\mathsf{K}$ is regular we say that $\mathsf{H}$ is a \emph{regular stochastic partly diffuse} boundary operator.
\end{defi}
\begin{nb} Notice that, being a convex combination of stochastic operators, a stochastic partly diffuse operator $\mathsf{H}$  is stochastic.
\end{nb}

\section{General results for abstract stochastic boundary operators}\label{sec:general}

We begin with the following which is a direct consequence of \cite[Theorem 21]{Mkst}: 
\begin{propo}\phantomsection\label{propo:honest}
Let $\mathsf{H} \in\mathscr{B}(\lp,\lm)$ be a stochastic boundary operator. Let there exist $\varphi \in \lp$ such that $\varphi >0$ $\mu$-a.e.  and $\varphi \geq \mathsf{M}_{0}\mathsf{H}\varphi$. Then, $A=\overline{\mathsf{T}_{\mathsf{H}}}$ and $(U_{\mathsf{H}}(t))_{t\geq 0}$ is a stochastic $C_{0}$-semigroup, where we recall that $(A,\D(A))$ is the generator of $(U_{\mathsf{H}}(t))_{t\geq0}$.\end{propo}

We give a general result about the spectrum of $\mathsf{T}_{\mathsf{H}}$:

\begin{propo}\phantomsection\label{propo:general}
Let $\mathsf{H} \in\mathscr{B}(\lp,\lm)$ be a stochastic boundary operator. Then, there is a nonnegative $\Psi \in \D(\mathsf{T}_{\mathsf{H}})$ with $\mathsf{T}_{\mathsf{H}}\psi=0$ if and only if $1$ is an eigenvalue of $\mathsf{M}_{0}\mathsf{H}$ associated to a \emph{nonnegative} eigenfunction $\varphi \in \Y^{+}_{1}$ such that $\mathsf{H}\varphi \in \Y^{-}_{1}.$
\end{propo}
\begin{proof} Assume first that there exists $\varphi \in \Y^{+}_{1}$ such that 
$$\mathsf{M}_{0}\mathsf{H} \varphi=\varphi, \qquad \text{ and } \quad \mathsf{H}\varphi \in \Y^{-}_{1}.$$
Then, as already seen (see \eqref{eq:Xio}), $\mathsf{\Xi}_{0}\mathsf{H}\varphi \in X$. Let $\Psi=\mathsf{\Xi}_{0}\mathsf{H}\varphi$. One has $\mathsf{T}_{\mathrm{max}}\Psi=0$, $\mathsf{B}^{-}\Psi=\mathsf{B}^{-}\mathsf{\Xi}_{0}\mathsf{H}\varphi=\mathsf{H}\varphi$, i.e. $\mathsf{B}^{-}\Psi=\mathsf{H}\mathsf{B}^{+}\Psi.$ This means that $\Psi \in \D(\mathsf{T}_{\mathsf{H}})$ with $\mathsf{T}_{\mathsf{H}}\Psi=0.$ 

Assume now that $0 \in \mathfrak{S}_{p}(\mathsf{T}_{\mathsf{H}})$ is associated to a nonnegative eigenfunction  $\Psi \in \D(\mathsf{T}_{\mathsf{H}})$. Let $\varphi=\mathsf{B}^{+}\Psi$ and $u=\mathsf{B}^{-}\Psi$. It holds $u=\mathsf{H}\varphi$ and, solving the boundary value problem \eqref{BV0} (see Proposition \ref{propBV0}) yields $\Psi=\mathsf{\Xi}_{0}u$. It is easy to check then that $\varphi=\mathsf{M}_{0}\mathsf{H}\varphi$. Since $u \neq 0 \implies \varphi \neq 0$, we get $1 \in \mathfrak{S}_{p}(\mathsf{M}_{0}\mathsf{H}).$
\end{proof}
We recall the definition of irreducible operators or semigroups in $L^{1}$-spaces and refer to \cite{nagel} for more details.
\begin{defi}\phantomsection\label{defi:irred} Let $(E,\Sigma,m)$ be a given $\sigma$-finite measure space.  Let $B \in \mathscr{B}(L^{1}(E,\Sigma,m))$  be given. Then, we say that 
\begin{enumerate}[i)]
\item $B$ is \emph{positive} and write $B \geq 0$, if $B$ leaves invariant the cone of nonnegative functions of $L^{1}(E,\Sigma,m)$ i.e. for any $h \in L^{1}(E,\Sigma,m)$, 
$$h(s) \geq 0 \qquad \text{ for $m$-a. e. $s \in E$ } \implies Bh(s) \geq 0 \qquad \text{ for $m$-a. e. $s\in E$.}$$
\item $B$ is \emph{positivity improving} if for any $h \in L^{1}(E,\Sigma,m)$ non identically zero
$$h(s) \geq 0 \qquad \text{ for $m$-a. e. $s \in E$ } \implies Bh(s) >0 \qquad \text{ for $m$-a. e. $s\in E$.}$$
\item $B$ is \emph{irreducible} if for any non trivial and nonnegative $h \in L^{1}(E,\Sigma,m)$ and any non trivial nonnegative $g \in L^{\infty}(E,\Sigma,m)$, there exists $n \in \mathbb{N}$ such that
$$\langle B^{n}h\,,\,g\rangle_{1,\infty} >0$$
where $\langle \cdot,\cdot \rangle_{1,\infty}$ is the duality bracket between $L^{1}(E,\Sigma,m)$ and $L^{\infty}(E,\Sigma,m)$.
\item A positive $C_{0}$-semigroup $(S(t))_{t\geq0}$ on $L^{1}(E,\Sigma,m)$ with generator $(G,\D(G))$ is irreducible if, for any non trivial nonnegative $h \in L^{1}(E,\Sigma,m)$ and any non trivial nonnegative $g \in L^{\infty}(E,\Sigma,m)$, there exists $t \geq 0$ such that $\langle S(t)h,g\rangle_{1,\infty} >0$. This property is equivalent to the fact that $\Rs(\lambda,G)$ is positivity improving for $\lambda >0$ large enough.
\end{enumerate}
\end{defi} 
We introduce the following: 
\begin{hyp}\phantomsection\label{hyp2}
$\mathsf{H} \in\mathscr{B}(\lp,\lm)$ is a  stochastic operator such that $\mathsf{M}_{0}\mathsf{H} \in \mathscr{B}(\lp)$ is irreducible and there exists $\psi_{0} \in \lp$ such that $\mathsf{H}\psi_{0} >0$ $\mu$-a.e. on $\Gamma_{+}.$
\end{hyp}
\begin{nb} If $\mathsf{H}$  is stochastic partly diffuse operator of the form \eqref{eq:partlydif} and if $\mathsf{M}_{0}(1-\alpha)K$ is irreducible then so is $\mathsf{M}_{0}\mathsf{H}$. This is the case for instance if $\|\alpha\|_{\infty}< 1$ and if, for $\pi$-a.e. $x \in \partial \Omega$ 
\begin{equation}\label{eq:k>0}
k(x,v,v') >0 \qquad \text{ for $\bm{\mu}_{x}$-a.e. } v \in \Gamma_{-}(x),\:v'\in \Gamma_{+}(x).\end{equation}
In addition, if $\alpha(x) >0$ $\pi$-a.e. $x \in \partial \Omega$, the second condition $\mathsf{H}\psi_{0} >0$ $\mu$-a.e. is satisfied by any $\psi_{0} >0$ $\mu$-a.e.
Otherwise, we assume that, for any $x \in\partial\Omega$ such that $\alpha(x)=0$ there is a subset $\gamma_{+}(x) \subset \Gamma_{+}(x)$ of positive $\bm{\mu}_{x}$-measure and such that $k(x,v,v') >0$ for $\bm{\mu}_{x}$-a.e. $v \in \Gamma_{+}(x), v' \in \Gamma_{+}(x).$ Then, again the second condition  is satisfied by any $\psi_{0} >0$ $\mu$-a.e. In particular, the second assumption in Assumptions \ref{hyp2} is always satisfied under \eqref{eq:k>0}.\end{nb}
One has the following (see also \cite[Remark 20]{Mkst}):
\begin{propo}\phantomsection\label{lem:irre} Let $\mathsf{H} \in\mathscr{B}(\lp,\lm)$ be a  stochastic operator and let Assumptions \ref{hyp2} be satisfied. Then, the $C_{0}$-semigroup $(U_{\mathsf{H}}(t))_{t \geq 0}$ is irreducible.\end{propo}
\begin{proof} Let $h \in X$ and $g \in L^{\infty}(\Omega \times V,\d x \otimes \bm{m}(\d v))$ be nonnegative and non trivial. Denoting the duality parity between $X$ and its dual simply by $\langle \cdot,\cdot \rangle$, we have for any $\lambda >0$
\begin{multline*}
\langle \Rs(\lambda,A)h,g\rangle=\langle \mathsf{R}_{\lambda}h,g\rangle + \sum_{k=0}^{\infty}\langle \mathsf{\Xi}_{\lambda}\mathsf{H}\left(\mathsf{M}_{\lambda}\mathsf{H}\right)^{k}\mathsf{G}_{\lambda}g,g\rangle
=\langle \mathsf{R}_{\lambda}h,g\rangle + \sum_{k=0}^{\infty}\langle \left(\mathsf{M}_{\lambda}\mathsf{H}\right)^{k}\mathsf{G}_{\lambda}h,\mathsf{H}^{\star}\mathsf{\Xi}_{\lambda}^{\star}g\rangle\end{multline*}
where $\mathsf{H}^{\star}$ and $\mathsf{\Xi}_{\lambda}^{\star}$ denote the dual operator of $\mathsf{H}$  and $\mathsf{\Xi}_{\lambda}$ respectively. Notice that $g_{\star}=\mathsf{\Xi}_{\lambda}^{\star}g$ is nonnegative and nontrivial and the same holds for $\mathsf{H}^{\star}\mathsf{\Xi}_{\lambda}^{\star}=\mathsf{H}^{\star}g_{\star}$ since, under Assumption \ref{hyp2}
$$\langle \psi_{0},\mathsf{H}^{\star}g_{\star}\rangle=\langle \mathsf{H}\psi_{0},g_{\star} \rangle >0.$$
Now, since the irreducibility of $\mathsf{M}_{0}\mathsf{H}$ is equivalent to that of $\mathsf{M}_{\lambda}\mathsf{H}$ for any $\lambda >0$, we deduce that $\langle \Rs(\lambda,A)h,g\rangle >0$ for any nonnegative and nontrivial $g \in L^{\infty}(\Omega \times V,\d x \otimes \bm{m}(\d v))$ which proves that $\Rs(\lambda,A)h$ is positive a. e. on $\Omega \times V$ and $\Rs(\lambda,A)$ is positivity improving.\end{proof}

The main result of this section is then the following:
\begin{theo}\phantomsection\label{theo:irred} Let $\mathsf{H} \in\mathscr{B}(\lp,\lm)$ be a  stochastic boundary operator and let Assumptions \ref{hyp2} be satisfied. Assume there exists 
$\Psi_{\mathsf{H}} \in \D(\mathsf{T}_{\mathsf{H}})$ (with unit norm) such that
$$\mathsf{T}_{\mathsf{H}}\Psi_{\mathsf{H}}=0 \qquad \text{ and } \quad \Psi_{\mathsf{H}} \geq 0 \quad \text{ a.e. on } \Omega \times V.$$
Then, $\overline{\mathsf{T}_{\mathsf{H}}}$ generates a irreducible and stochastic $C_{0}$-semigroup $(U_{\mathsf{H}}(t))_{t \geq 0}$ on $X$ and $\Psi_{\mathsf{H}}$ is the unique invariant density of $(U_{\mathsf{H}}(t))_{t\geq 0}$. Moreover, $(U_{\mathsf{H}}(t))_{t\geq 0}$ is ergodic with ergodic projection
$$\mathbb{P}f=\varrho_{f}\,{\Psi}_{\mathsf{H}}, \qquad \text{ where } \quad\varrho_{f}=\displaystyle\int_{\Omega\times\R^{d}}f(x,v)\d x \otimes \bm{m}(\d v), \qquad f \in X,$$
i.e. 
\begin{equation}\label{eq:ergodic}
\lim_{t \to \infty}\left\|\frac{1}{t}\int_{0}^{t}U_{\mathsf{H}}(s)f\d s-\varrho_{f}{\Psi}_{\mathsf{H}}\right\|_{X}=0, \qquad  \forall f \in X\end{equation}
and $X=\mathrm{Ker}(\overline{\mathsf{T}_{\mathsf{H}}}) \oplus \overline{\mathrm{Range}(\overline{\mathsf{T}_{\mathsf{H}}})}.$ 
\end{theo}
\begin{proof} According to Proposition \ref{propo:honest}, there exists $\varphi \in \lp$ such that $\varphi \geq 0$ $\mu$-a.e. and $\varphi=\mathsf{M}_{0}\mathsf{H}\varphi.$ Since $\mathsf{M}_{0}\mathsf{H}$ is irreducible, $\varphi >0$ $\mu$-a.e. and   we deduce from Proposition \ref{propo:honest} that $A=\overline{\mathsf{T}_{\mathsf{H}}}.$ Moreover, Lemma \ref{lem:irre} ensures that $(U_{\mathsf{H}}(t))_{t\geq0}$ is irreducible. Since $\mathrm{Ker}(\overline{\mathsf{T}_{\mathsf{H}}}) \neq 0$, the ergodicity of $(U_{\mathsf{H}}(t))_{t\geq0}$ follows from \cite[Theorem 7.3, p. 174 and Theorem 5.1 p. 123]{davies}.\end{proof}

\section{Weak compactness result and existence of an invariant density}\label{sec:comp}

\subsection{Weak compactness result} We prove here the main compactness result of the paper. The proof of the result is based on a series of important geometrical results regarding regularity and transversality of the ballistic flow 
$$\bxi\::\:(x,v) \in \Gamma_{+} \mapsto \bxi(x,v)=(x-\tau_{-}(x,v)v,v) \in \pO \times V.$$
Such  technical results have been postponed to Appendix \ref{app:ballistic} for the clarity of the reading and will be used repeatedly in the proof of the following. 
\begin{theo}\phantomsection\label{propo:weakcompact}
Let $\mathsf{K} \in \mathscr{B}(\lp,\lm)$ be regular diffuse boundary operator. Then, one has
$$\mathsf{K}\mathsf{M}_{0}\mathsf{K} \in \mathscr{B}(\lp,\lm)$$
is weakly compact.
\end{theo} 
\begin{proof} Let $k,\ell \geq 1$ be fixed. Let $(\mathsf{K}_{m})_{m} \subset \mathscr{B}(L^{1}_{+},L^{1}_{-})$ be the sequence of approximation obtained from Lemma \ref{lem:regular},
 which is such that $\lim_{m}\|\mathsf{K}-\mathsf{K}_{m}\|_{\mathscr{B}(L^{1}_{+},L^{1}_{-})}=0$. It is then enough to prove the weak-compactness of $\mathsf{K}_{n}\mathsf{M}_{0}\mathsf{K}_{m}$ for any $n,m\in \N$. Still using the notations of Lemma \ref{lem:regular}, introduce 
$$\overline{\mathsf{K}}_{m}f(x,v)=\psi_{m}(v)\int_{v'\cdot n(x) >0}f(x,v')\bm{\mu}_{x}(\d v'), \qquad f \in L^{1}_{+}, \qquad (x,v) \in \Gamma_{-}, \quad m \in \N$$
where $\psi_{m}(v)=m\ind_{B_{m}}(v).$ Given $n,m \in \N$, using \eqref{eq:Hmpsim} and a domination argument, the weak-compactness of $\overline{\mathsf{K}}_{n} \mathsf{M}_{0}\overline{\mathsf{K}}_{m}$ would imply the result. To avoid too heavy notations, and setting for instance $F(v)=\max(\psi_{n},\psi_{m})$, it suffices to prove that  $\mathsf{K}\mathsf{M}_{0}\mathsf{K}$ is weakly-compact for 
\begin{equation}\label{eq:Kprof}
\mathsf{K}\varphi(x,v)=F(v)\int_{\Gamma_{+}(x)}\varphi(x,v')\bm{\mu}_{x}(\d v').\end{equation}
Since $F$ is compactly supported and bounded we can assume without loss of generality that 
\begin{equation}\label{eq:FB1}
F=\ind_{B_{1}}\end{equation} which amounts to consider only velocities $|v| \leq 1.$ Recall that, thanks to Corollary \ref{cor:ba}, 
$$\bxi^{-1}\::\:(x,\omega) \in \widehat{\Gamma}_{-} \longmapsto \bxi^{-1}(x,\omega)=(x+\tau_{+}(x,\omega)\omega,\omega) \in \Gamma_{+}$$
is a $\mathcal{C}^{1}$ diffeomorphism from $\widehat{\Gamma}_{-}$ onto its image. Denoting here for simplicity
$$\Gamma_{-}^{(0)}=\Gamma_{-}\setminus \widehat{\Gamma}_{-}, \qquad \Gamma_{+}^{(0)}=\bxi^{-1}(\Gamma_{-}^{(0)})$$
one has $\mu_{-}(\Gamma_{-}^{(0)})=\mu_{+}(\Gamma_{+}^{(0)})=0$ (see for instance \eqref{eq:delta}) and 
we may make the  identification
\begin{equation}\label{eq:ide}
\lp=L^{1}(\Gamma^{(1)}_{+},\d\mu), \qquad \qquad \lm=L^{1}(\Gamma^{(0)}_{-},\d\mu)\end{equation}
so that we only have to prove the weak-compactness of 
$$\mathsf{K}\mathsf{M}_{0}\mathsf{K}\::\:L^{1}(\Gamma_{+}^{(0)},\d\mu) \rightarrow L^{1}(\Gamma^{(0)}_{-},\d\mu).$$
 Notice that, by \eqref{eq:Kprof} and \eqref{eq:FB1}, the range of $\mathsf{K}\mathsf{M}_{0}\mathsf{K}$ can be rather considered as $L^{1}(\Gamma_{-}^{(0)},\d\tilde{\mu})$ where
\begin{equation}\label{eq:mutilde}
\d\tilde{\mu}(x,v)=F(v)\d\mu(x,v)=\ind_{B_{1}}(v)|v\cdot n(x)| \pi(\d x)\otimes \bm{m}(\d v)\end{equation}
is nothing but the restriction of $\mu_{-}$ to $\pO \times B_{1}.$ In particular, $\tilde{\mu}$ is a \emph{finite} measure. 
From a simple consequence of the Dunford-Pettis criterion (see \cite[Corollary 4.7.21, p. 288]{bogachev}), we need to prove that, for \emph{any} nonincreasing sequence of measurable subsets $(\mathcal{A}_{j})_{j} \subset \Gamma_{-}^{(0)}$ with $\bigcap_{j}\mathcal{A}_{j}=\varnothing$, it holds
\begin{equation} 
\label{eq:weakcop}
\lim_{j \to \infty}\sup_{\|\varphi\|_{\lp} \leq 1}\int_{\mathcal{A}_{j}}\left|\mathsf{K}\mathsf{M}_{0}\mathsf{K}\varphi(x,v)\right|\d\mu(x,v)=0.\end{equation}
Since $\mathsf{K}$ and $\mathsf{M}_{0}$ are nonnegative operators, it suffices of course to consider nonnegative $\varphi\in \lp$ in \eqref{eq:weakcop}. Let us fix a sequence $(\mathcal{A}_{j})_{j} \subset \Gamma_{-}^{(0)}$ with $\bigcap_{j}\mathcal{A}_{j}=\varnothing$ and consider a nonnegative $\varphi \in \lp$.  We set 
$$I_{j}(\varphi)=\int_{\mathcal{A}_{j}} \mathsf{K}\mathsf{M}_{0}\mathsf{K}\varphi(x,v)\d\mu(x,v).$$
Introduce then  the $\mu$-preserving change of variables
$$(x,v) \in \Gamma_{+}  {\longmapsto} (y,w)=\bxi (x,v)=\left(\bxi_{s}(x,v),v\right)=(x-\tau_{-}(x,v)v,v) \in \Gamma_{-},$$ and denote   the position component of the
inverse  of $\bxi$ by $\mathbf{Y}$, i.e. given   $(y,w) \in \Gamma_{-}$, 
$$\mathbf{Y}(y,w)=\bxi_{s}^{-1}(y,w) \in \pO\qquad \text{ and } \qquad \bxi^{-1}(y,w)=(\mathbf{Y}(y,w),w) \in \Gamma_{+}.$$
Notice that here and everywhere in the text, $\bxi_{s}^{-1}$ denotes the position component of the inverse $\bxi^{-1}$, i.e. $\bxi_{s}^{-1}=(\bxi^{-1})_{s}$.
Defining, for any $y \in \partial \Omega$,
\begin{equation*}\label{eq:BjkGj}
B_{j }(y)=\int_{\Gamma_{-}(y)}G_{j}(\mathbf{Y}(y,w))F(w)\bm{\mu}_{y}(\d w), \qquad G_{j}(x)=\int_{\Gamma_{-}(x)}\ind_{\mathcal{A}_{j}}(x,v)F(v)\bm{\mu}_{x}(\d v),\end{equation*}
one can show (see the Lemma \ref{lemma:Ijphi} hereafter) that 
\begin{equation*}I_{j}(\varphi)=\int_{\Gamma_{+}}B_{j}(y)\varphi(y,w_{1})\d\mu(y,w_{1}).\end{equation*}
It is clear then that \eqref{eq:weakcop} will hold true if 
\begin{equation}\label{eq:limjfunc}
\lim_{j\to\infty}\sup_{y\in \partial\Omega}B_{j}(y)=0.\end{equation}
The proof of this property is given in the next Lemma yielding the desired weak-compactness.\end{proof}
\begin{lemme}\label{lem:final}
With the notations of the proof of Theorem \ref{propo:weakcompact}, given $y \in \partial \Omega$, introduce
$$B_{j}(y)=\int_{\Gamma_{-}(y)}G_{j}(\mathbf{Y}(y,w))F(w)\bm{\mu}_{y}(\d w).$$
Then, $\underset{j \to \infty}{\lim}\underset{y \in \pO}{\sup}B_{j}(y)=0.$
\end{lemme}

The proof of the above will use the the following polar
decomposition theorem (see \cite[Lemma 6.13, p.113]{voigt}):
\begin{lemme}\phantomsection\label{lem:polar} If $\bm{m}$ is a orthogonally invariant Borel measure with support $V \subset \R^{d}$, introduce $\bm{m}_{0}$ as the image of the measure $\bm{m}$ under the transformation $v \in \R^{d} \mapsto |v| \in [0,\infty),$ i.e. $\bm{m}_{0}(I)=\bm{m}\left(\{v \in \R^{d}\;;\;|v| \in I\}\right)$ for any Borel subset $I \subset \R^{+}.$ Then, for any $\psi \in L^{1}(\R^{d},\bm{m})$ it holds
$$\int_{\R^{d}}\psi(v)\bm{m}(\d v)=\frac{1}{|\mathbb{S}^{d-1}|}\int_{0}^{\infty}\bm{m}_{0}(\d\varrho)\int_{\mathbb{S}^{d-1}}\psi(\varrho\,\omega)\sigma(\d\omega)$$
where $\d\sigma$ denotes the Lebesgue measure on $\mathbb{S}^{d-1}$ with total mass $|\mathbb{S}^{d-1}|.$ 
\end{lemme}

\begin{nb}\label{nb:recallSepsi} We shall use in the proof of Lemma \ref{lem:final} that, with the notations of Proposition \ref{prop:diff}, for any $y \in \pO$ we can construct an orthonormal basis $\{\bm{e}_{1}(y),\ldots,\bm{e}_{d}(y)\}$ of $\R^{d}$ depending continuously on $y \in \pO$ with 
$$\bm{e}_{d}(y)=-n(y)$$
and such that, in this basis, any $\varpi \in \mathbb{S}^{d-1}$ can be written as 
$$\varpi=\sum_{i=1}^{d}\omega_{i}\bm{e}_{i}(y)$$ where $\omega=\omega(\bm{\theta})=(\omega_{1},\ldots,\omega_{d})$ is given by \eqref{eq:omeg} in terms of the polar coordinates
$$\bm{\theta}=(\theta_{1},\ldots,\theta_{d-1}) \in U=[0,2\pi] \times [0,\pi]^{d-3}\times \left[0,\frac{\pi}{2}\right].$$
In this case, $\omega$ is independent of $y \in \pO$. We also recall that, according to Remark \ref{nb:Sepsi}, for any $\varepsilon >0$, one can define $S_{\varepsilon}(y)$ as those $\varpi \in \mathbb{S}^{d-1}$ for which 
$$\bm{\theta} \in U_{\varepsilon}:=\left\{(\theta_{1},\ldots,\theta_{d-1}) \in U\;;\; \sin^{d-2}\theta_{d-1}\,\sin^{d-3}\theta_{d-2}\,\ldots\sin\theta_{2} \leq \varepsilon\right\}.$$
and prove that $\lim_{\varepsilon \to 0^{+}}\sup_{y \in \pO}\sigma(S_{\varepsilon}(y))=0.$ See Remark \ref{nb:Sepsi} for more details.\end{nb}

\begin{proof}[Proof of Lemma \ref{lem:final}]
Recall that the identification \eqref{eq:ide} is in force and $(\mathcal{A}_{j})_{j} \subset \Gamma_{-}^{(0)}$ is non-increasing with $\bigcap_{j}\mathcal{A}_{j}=\varnothing$. Introducing the polar coordinates $w=\varrho\,\varpi$,  with $\varrho >0$ and $\varpi  \in \mathbb{S}^{d-1}$ and using Lemma \ref{lem:polar} (recall that $\bm{\mu}_{x}(\d v)=|v\cdot n(x)| \bm{m}(\d v)$) we get
\begin{equation*} 
B_{j}(y)=|\mathbb{S}^{d-1}|^{-1}\int_{0}^{1}\varrho\,\bm{m}_{0}(\d \varrho)\int_{\Gamma_{-}(y)}G_{j}(\mathbf{Y}(y,\varrho\,\varpi))|\varpi \cdot n(y)|\sigma(\d\omega).\end{equation*}
Since $\{0\}$ is not an atom for the measure $\varrho\,\bm{m}_{0}(\d \varrho)$, 
according to the dominated convergence theorem, it is enough to prove that, for any given $\varrho \in (0,1)$, 
\begin{equation}\label{eq:Gjdel}
\lim_{j\to\infty}\sup_{y \in \pO}\int_{\Gamma_{-}(y)}G_{j}(\mathbf{Y}(y,\varrho\,\varpi))|\varpi\cdot n(y)|\sigma(\d \varpi)=0
\end{equation}
and, since $\varrho > 0$, there is no loss of generality in proving the result only for $\varrho=1.$ 
Notice that, for any $\varepsilon >0$ and $y \in \pO$, 
\begin{multline*}
\int_{\Gamma_{-}(y)}G_{j}(\mathbf{Y}(y,\varpi))|\varpi\cdot n(y)|\sigma(\d \varpi)=\int_{\Gamma_{-}(y) \cap S_{\varepsilon}(y)}G_{j}(\mathbf{Y}(y,\varpi))|\varpi\cdot n(y)|\sigma(\d \varpi)\\
+\int_{\Gamma_{-}(y)\setminus S_{\varepsilon}(y)}G_{j}(\mathbf{Y}(y,\varpi))|\varpi\cdot n(y)|\sigma(\d \varpi)
\end{multline*}
where $S_{\varepsilon}(y)$ has been introduced in the above Remark \ref{nb:recallSepsi}. Clearly, since  $\|G_{j}\|_{\infty} \leq \int_{\R^{d}}F(u)|u|\bm{m}(\d u)$, there is $C >0$ independent of $j$  such that
$$\sup_{y \in \pO}\int_{\Gamma_{-}(y) \cap S_{\varepsilon}(y)}G_{j}(\mathbf{Y}(y,\varpi))|\varpi\cdot n(y)|\sigma(\d \varpi) \leq C\,\sup_{y \in \pO}\sigma\left(S_{\varepsilon}(y)\right)$$
which goes to $0$ as $\varepsilon \to 0$ according to Remark \ref{nb:Sepsi}. Therefore, to show \eqref{eq:Gjdel}, we only have to prove that, for any $\varepsilon >0$, 
\begin{equation}\label{eq:Gjfin}\lim_{j\to\infty}\sup_{y \in \pO}\int_{\Gamma_{-}(y)\setminus S_{\varepsilon}(y)}G_{j}(\mathbf{Y}(y,\varpi))\sigma(\d \varpi)=0.\end{equation} Recall that, from Proposition \ref{prop:diff}, for  any $y\in\pO$, the mapping 
\begin{equation}\label{eq:varpiC1}
\varpi \in \widehat{\Gamma}_{-}(y) \setminus S(y) \longmapsto \mathbf{Y}(y,\varpi) \in \pO \quad \text{ is of class $\mathcal{C}^{1}$ with differential of rank $d-1$}\end{equation} 
with moreover  $\sigma(S(y))=0$.
On the other hand, with the notations and parametrization used in Proposition \ref{prop:diff} and recalled in Remark \ref{nb:recallSepsi}, the mapping $(\bm{\theta},y) \in U \times \pO \longmapsto \varpi=\varpi(\bm{\theta},y)$ is continuous while
$$\bm{\theta} \in U \longmapsto \varpi=\varpi(\bm{\theta},y)$$
is of class $\mathcal{C}^{1}$ with a continuous derivative $(\bm{\theta},y) \in A \times \pO \longmapsto \partial_{\bm{\theta}} \varpi(\bm{\theta},y)$. Since the mapping $\bm{\theta} \in U_{\varepsilon} \mapsto \varpi(\bm{\theta},y)$ is a $\mathcal{C}^{1}$ parametrization of $S_{\varepsilon}(y)$, 
by virtue of \eqref{eq:varpiC1} we have that, for any $y \in \pO$,
$$\bm{\theta} \in U \setminus U_{\varepsilon} \longmapsto   \mathbf{Y}(y,\varpi(\bm{\theta},y)) \in \pO$$
is a regular parametrization of 
$$\mathcal{E}_{y}:=\left\{\mathbf{Y}(y,\varpi)\,;\,\varpi \in \Gamma_{-}(y) \setminus S_{\varepsilon}(y)\right\} \subset \pO.$$
Then, according to \cite[Lemma 5.2.11 \& Theorem 5.2.16, pp. 128--131]{stroock}, the Lebesgue surface measure $\pi_{\mathcal{E}_{y}}(\d Y)$ on $\mathcal{E}_{y}$ is given by
$$J_{\mathbf{Y}}(y,\varpi)\d\theta_{1}\ldots\d\theta_{d-1}=J_{\mathbf{Y}}(y,\varpi)\d\bm{\theta}$$
where 
$$J_{\mathbf{Y}}(y,\varpi)=\left[\mathrm{det}\left(\left(\partial_{\theta_{i}}\mathbf{Y}(y,\varpi)\,,\,\partial_{\theta_{\ell}}\mathbf{Y}(y,\varpi)\right)_{1\leq i,\ell \leq d-1}\right)\right]^{1/2} >0$$
on $U \setminus U_{\varepsilon}.$ Since the mapping
$$(\bm{\theta},y) \in U \times \pO \longmapsto \partial_{\theta_{i}}\mathbf{Y}(y,\varpi(\bm{\theta},y))=\d_{\varpi}\mathbf{Y}(y,\varpi(\bm{\theta},y))\partial_{\theta_{i}}\varpi(\bm{\theta},y))$$
is continuous for any $i=1,\ldots,d-1$, then so is the mapping
$$(\bm{\theta},y) \in U \times \pO \longmapsto J_{\mathbf{Y}}(y,\varpi(\bm{\theta},y))$$
and there exists $C_{\varepsilon} >0$ such that 
$$J_{\mathbf{Y}}(y,\varpi(\bm{\theta},y)) \geq C_{\varepsilon}>0, \qquad \forall (\bm{\theta},y) \in (U \setminus U_{\varepsilon}) \times \pO.$$
Hence, for any $y \in \pO$
\begin{multline*}
\int_{\Gamma_{-}(y)\setminus S_{\varepsilon}(y)}G_{j}(\mathbf{Y}(y,\varpi))\sigma(\d \varpi) 
\\
\leq \frac{1}{C_{\varepsilon}}\int_{U \setminus U_{\varepsilon}}G_{j}\left(\mathbf{Y}(y,\varpi(\bm{\theta},y))\right)J_{\mathbf{Y}}(y,\varpi(\bm{\theta},y))\d\bm{\theta}\\
=\frac{1}{C_{\varepsilon}}\int_{\mathcal{E}_{y}}G_{j}(Y){\pi}_{\mathcal{E}_{y}}(\d Y).\end{multline*}
Clearly, recalling the definition of $G_{j}$ -- and because the measures $\pi$ and $\pi_{\mathcal{E}_{y}}$ coincide on $\mathcal{E}_{y}$ -- we get
\begin{equation*}\begin{split}
\int_{\mathcal{E}_{y}}G_{j}(Y){\pi}_{\mathcal{E}_{y}}(\d Y)&=\int_{\mathcal{E}_{y}}\left(\int_{\Gamma_{-}(Y)}F(v)\ind_{\mathcal{A}_{j}}(Y,v)\bm{\mu}_{Y}(\d v)\right){\pi}_{\mathcal{E}_{y}}(\d Y)\\
&\leq \int_{\pO}\left(\int_{\Gamma_{-}(Y)}F(v)\ind_{\mathcal{A}_{j}}(Y,v)\bm{\mu}_{Y}(\d v)\right){\pi}(\d Y)=\tilde{\mu}(\mathcal{A}_{j}),\end{split}\end{equation*}
where $\tilde{\mu}$ is given by \eqref{eq:mutilde}. Thus,
$$\sup_{y \in \pO}\int_{\Gamma_{-}(y)\setminus S_{\varepsilon}(y)}G_{j}(\mathbf{Y}(y,\varpi))\sigma(\d \varpi)  \leq \frac{\tilde{\mu}(\mathcal{A}_{j})}{C_{\varepsilon}} \qquad \forall j \in \N.$$
Since $(\mathcal{A}_{j})_{j}$ is non-increasing with $\bigcap_{j}\mathcal{A}_{j}=\varnothing$ and $\tilde{\mu}$ is a finite measure, we have $\lim_{j}\tilde{\mu}(\mathcal{A}_{j})=0$ which implies \eqref{eq:Gjfin} and proves the Lemma.
\end{proof}

\begin{lemme}\label{lemma:Ijphi} With the notations of the proof of Theorem \ref{propo:weakcompact},  it holds, for any $j \in \N$,
\begin{equation}\label{eq:Ijphi}
I_{j}(\varphi)=\int_{\Gamma_{+}}B_{j}(y)\varphi(y,w_{1})\d\mu(y,w_{1})\end{equation}
where, for any $y \in \partial \Omega$
$$B_{j}(y)=\int_{\Gamma_{-}(y)}G_{j}(\mathbf{Y}(y,w))F(w)\bm{\mu}_{y}(\d w), \qquad G_{j}(x)=\int_{\Gamma_{-}(x)}\ind_{\mathcal{A}_{j}}(x,v)F(v)\bm{\mu}_{x}(\d v).$$
\end{lemme}
 
\begin{proof}
We use the notations of the above proof. In particular, we assume $\mathsf{K}$ to be given by \eqref{eq:Kprof}. Notice that, for any nonnegative $\psi \in \lm$,
\begin{multline*}
\int_{\mathcal{A}_{j}}\mathsf{K}\mathsf{M}_{0}\psi(x,v)\d\mu_{-}(x,v)=\int_{\mathcal{A}_{j}}F(v)\left(\int_{\Gamma_{+}(x)}\mathsf{M}_{0}\psi(x,v')\bm{\mu}_{x}(\d v')\right)\d\mu_{-}(x,v)\\
=:\int_{\mathcal{A}_{j}}F(v)\Psi_{0}(x)\d\mu_{-}(x,v)
\end{multline*}
where one has
$$\Psi_{0}(x)=\int_{\Gamma_{+}(x)}\mathsf{M}_{0}\psi(x,v')\bm{\mu}_{x}(\d v')=\int_{\Gamma_{+}(x)}\psi(\bxi(x,v'))\bm{\mu}_{x}(\d v'), \qquad x \in \partial \Omega.$$
Simple use of Fubini's theorem yields
$$\int_{\mathcal{A}_{j}}F(v)\Psi_{0}(x)\d\mu_{-}(x,v)=\int_{\partial\Omega}G_{j}(x)\Psi_{0}(x) \pi(\d x)
=\int_{\Gamma_{+}}G_{j}(x)\psi(\bxi(x,v'))\d\mu_{+}(x,v').$$
Introduce then the $\mu$-preserving change of variables 
\begin{equation}\label{eq:YV}
(x,v') \in \Gamma_{+}  {\longmapsto} (y,w)=\bxi (x,v')  \in \Gamma_{-}\end{equation}
and recalling that $x=\mathbf{Y}(y,\omega)$
we have, 
\begin{equation*}
\int_{\mathcal{A}_{j}}\mathsf{K}\mathsf{M}_{0}\psi(x,v)\d\mu_{-}(x,v)=\int_{\Gamma_{-}}G_{j}(\mathbf{Y}(y,w))\psi(y,w)\d\mu_{-}(y,w).\end{equation*}
Applying this with $\psi=\mathsf{K}\varphi$ for some nonnegative $\varphi \in \lp$, we get \begin{multline*}
\int_{\mathcal{A}_{j}}\mathsf{K}\mathsf{M}_{0}\mathsf{K}\varphi(x,v)\d\mu_{-}(x,v)=\int_{\Gamma_{-}}G_{j}(\mathbf{Y}(y,w))\mathsf{K}\varphi(y,w)\d\mu_{-}(y,w)\\
=\int_{\Gamma_{-}}G_{j}(\mathbf{Y}(y,w))F(w)\d\mu_{-}(y,w)\int_{\Gamma_{+}(y)}\varphi(y,w_{1})\bm{\mu}_{y}(\d w_{1})
\end{multline*}
which is the desired result.\end{proof}

 \subsection{About the essential spectral radius of $\mathsf{M}_{0}\mathsf{H}$}

We are ready to show:\begin{theo}\phantomsection\label{theo:ressM0}
Let $\mathsf{H} \in\mathscr{B}(\lp,\lm)$ be a stochastic \emph{regular} partly diffuse boundary operator given by \eqref{eq:partlydif} {and denote for simplicity $\beta(x)=1-\alpha(x)$ for $\pi$-a. e. $x \in \pO$ and $\beta_{\infty}:=\mathrm{ess}\sup_{x \in \pO}\beta(x)$.  If 
\begin{equation}\label{eq:oscill}
\mathrm{ess\,inf}_{x \in \pO}\beta(x) >1+ \beta_{\infty}-\sqrt{1+\beta^{2}_{\infty}}\end{equation}
then $r_{\mathrm{ess}}(\mathsf{M}_{0}\mathsf{H}) <1.$}\end{theo}
\begin{proof} For notation simplicity, we simply denote by $\alpha \mathsf{R}$ the operator $\alpha(\cdot)\mathsf{R}$ and by $\beta\,\mathsf{K}$ the operator $\beta(\cdot)\mathsf{K}$. Note first that $\left(\mathsf{
M_{0}}\beta\mathsf{K}\right) ^{2}\leq \left(\mathsf{M_{0}K}\right) ^{2}$ and Theorem \ref{propo:weakcompact}
imply that $\left[\mathsf{M_{0}}\beta\mathsf{K}\right]^{2}$ is weakly compact.  We recall
that in $L^{1}$ spaces, the ideal of strictly singular operators coincides
with the ideal of weakly compact operators \cite{pelci}. Since
\begin{multline*}
\left[\mathsf{M_{0}H}\right]^{2}=\left[\mathsf{M_{0}}\alpha\mathsf{R}+\mathsf{M_{0}}\beta\mathsf{K}\right]^{2}=\left[\mathsf{M_{0}}\beta\mathsf{K}\right]^{2}
+\left[\mathsf{M_{0}}\alpha\mathsf{R}\right]^{2}+\mathsf{M_{0}}\alpha\mathsf{RM_{0}}\beta\mathsf{K}+\mathsf{M_{0}}\beta\mathsf{KM_{0}}\alpha\mathsf{R}\end{multline*}
then the stability of essential spectra by strictly singular perturbations \cite[Proposition 2.c.10, p.79]{linden} shows that 
$\left[\mathsf{M_{0}H}\right]^{2}$ and $\left[\mathsf{M_{0}}\alpha\mathsf{R}\right]^{2}+\mathsf{M_{0}}\alpha\mathsf{RM_{0}}\beta\mathsf{K}+\mathsf{M_{0}}\beta\mathsf{KM_{0}}\alpha\mathsf{R}$ share the same essential spectrum. In particular,
\begin{equation*}\begin{split}
r_{\mathrm{ess}}\left(\left[\mathsf{M_{0}H}\right]^{2}\right)&=r_{\mathrm{ess}}\left(\left[\mathsf{M_{0}}\alpha\mathsf{R}\right]^{2}+\mathsf{M_{0}}\alpha\mathsf{RM_{0}}\beta\mathsf{K}+\mathsf{M_{0}}\beta\mathsf{KM_{0}}\alpha\mathsf{R}\right)\\
&\leq \left\|\left[\mathsf{M_{0}}\alpha\mathsf{R}\right]^{2}+\mathsf{M_{0}}\alpha\mathsf{RM_{0}}\beta\mathsf{K}+\mathsf{M_{0}}\beta\mathsf{KM_{0}}\alpha\mathsf{R}\right\|_{\mathscr{B}(\lp)}\\
&\leq \|\alpha(\cdot)\|_{L^{\infty}(\pO)}^{2} + 2\|\alpha(\cdot)\|_{L^{\infty}(\pO)}\|\beta(\cdot)\|_{L^{\infty}(\pO)}\\
&=\left(\|\alpha(\cdot)\|_{L^{\infty}(\pO)}
+\|\beta(\cdot)\|_{L^{\infty}(\pO)}\right)^{2}-\|\beta(\cdot)\|^{2}_{L^{\infty}(\pO)}
\end{split}\end{equation*}
Since $\alpha(\cdot)=1-\beta(\cdot)$, this means that
$$r_{\mathrm{ess}}\left(\left[\mathsf{M_{0}H}\right]^{2}\right) \leq \left(1+\mathrm{osc}(\beta)\right)^{2}-\|\beta(\cdot)\|_{L^{\infty}(\pO)}^{2}$$
where $\mathrm{osc}(\beta)=\mathrm{ess}\!\sup_{x\in \pO}\beta(x)-\mathrm{ess}\!\inf_{x\in\pO}\beta(x)$ is the oscillation of $\beta(\cdot).$  Finally, the condition $\left(1+\mathrm{osc}(\beta)\right)^{2}-\|\beta(\cdot)\|_{L^{\infty}(\pO)}^{2}<1$ amounts to 
$$\left(\mathrm{osc}(\beta)\right) ^{2}+2\mathrm{osc}(\beta)<\beta _{\infty }^{2}$$
which is equivalent to \eqref{eq:oscill}. This ends the proof since $r_{\mathrm{ess}}\left(\left[\mathsf{M_{0}H}\right]^{2}\right)=\left(r_{\mathrm{ess}}(\mathsf{M_{0}H})\right)^{2}$ by the spectral mapping theorem.
\end{proof}

\begin{nb} We can view \eqref{eq:oscill} as 
$$\mathrm{osc}(\beta )<\sqrt{1+\beta _{\infty }^{2}}-1$$
which expresses a smallness of the oscillation $\mathrm{osc}(\beta )$ relatively to $\mathrm{ess}\!\sup_{x\in \pO}\beta(x).$ Notice that this condition is always satisfied if $\beta >0$
is a constant.
\end{nb}

\begin{nb}\label{nb:weaknb} We strongly believe that the assumption \eqref{eq:oscill} is purely technical and we conjecture the above result to be true with the sole assumption that $\mathrm{ess}\!\inf_{x \in \pO}\beta(x) > 0$, i.e. when the diffuse reflection is active everywhere on $\partial
\Omega $. 

In a previous version of the paper \cite{MKLR}, we erroneously established the inequality $r_{\mathrm{ess}}(\mathsf{M_{0}H}) <1$ from the stability of the essential spectral radius, proving that
$$r_{\mathrm{ess}}(\mathsf{M_{0}H})=r_{\mathrm{ess}}(\mathrm{M}_{0}\alpha\mathsf{R})$$
without any condition on the oscillation of $\beta (\cdot).$
The key point to establish such a stability result was the following (erroneous) property: for any integers $k,\ell \geq 1$, of the operators
\begin{equation}\label{eq:weakKR}
\mathsf{K}(\mathsf{M}_{0}\mathsf{R})^{k}\mathsf{M}_{0}\mathsf{K}(\mathsf{M}_{0}\mathsf{R})^{\ell}\mathsf{M}_{0}\mathsf{K}\::\:\lp  \rightarrow \lm \qquad \text{ is weakly-compact}.\end{equation}
As pointed out by an anonymous referee, the proof of such a result contained a gap and the result cannot be true whenever $\mathsf{R}$ is associated to bounce-back boundary conditions (see Example \ref{exe:bounce} for details). We however conjecture that the above operators are indeed weakly-compact for any $k,\ell \geq 1$ whenever $\mathsf{R}$ is associated to \emph{specular boundary reflection}, i.e.
$$\mathsf{R}\varphi(x,v)=\varphi(x,v-2(v \cdot n(x))n(x)), \qquad \varphi \in \lp, \qquad (x,v) \in \Gamma_{-}.$$ More generally, it would be interesting to characterize the domains $\Omega$ and the class of reflection boundary operators $\mathsf{R}$ for which the above \eqref{eq:weakKR} holds true for any $k,\ell \geq 1.$\end{nb}

\begin{nb} We point out here that the conclusion in Theorem \ref{theo:ressM0} applies to  any stochastic operator $\mathsf{R}$ and not only to reflection boundary  operators. \end{nb}

\section{Kinetic semigroup for regular partly diffuse boundary operators}
\label{ss:existence-in-dens}

We introduce the following set of Assumptions:

\begin{hyp}\phantomsection\label{hyp1} The \emph{regular} stochastic partly diffuse boundary operator 
$$\mathsf{H}=\alpha(\cdot)\mathsf{R} + (1-\alpha(\cdot))\mathsf{K}$$
is such that 
\begin{enumerate} 
\item[\textbf{A1)}] $\mathrm{Range}(\mathsf{K}) \subset \Y^{-}_{1};$
\item[\textbf{A2)}] $\mathsf{R}(\Y^{+}_{1}) \subset \Y^{-}_{1};$
\item[\textbf{A3)}] {Inequality  \eqref{eq:oscill} is satisfied}.
\end{enumerate}
\end{hyp}
\begin{nb} In the above set of Assumptions, it is possible to replace  $\Y^{\pm}_{1}$ with $L^{1}(\Gamma_{\pm},\tau_{\mp}\d\mu_{\pm})$. However, in this case, Assumption $\textbf{A2)}$ is not necessarily  satisfied for practical examples of boundary conditions (see Example \ref{exe:reflec}). \end{nb}
\begin{nb} {The above Assumption $\textbf{A3)}$ can be replaced with 
$$\textbf{A3')} \qquad r_{\mathrm{ess}}(\mathsf{M_{0}H}) < 1 \qquad \text{ and } \qquad \underset{x \in \pO}{\mathrm{esssup}}\;\alpha(x) < 1.$$
Notice that, as seen from Theorem \ref{theo:ressM0}, $\textbf{A3)} \implies \textbf{A3')}.$ In the rest of the analysis, this is only $\textbf{A3')}$ that will be used.}\end{nb}

\begin{exe}\phantomsection\label{exe:maxwell} Consider the classical Maxwell diffuse boundary condition for which
$$[\mathsf{H}f](x,v)=\frac{\M(v)}{\gamma(x)}\int_{v'\cdot n(x) >0}f(x,v')|v'\cdot n(x)|\d v', \qquad \forall (x,v) \in \Gamma_{-}$$
with
$$\mathcal{M}(v)=\frac{1}{(2\pi\theta)^{d/2}}\exp\left(-\frac{|v|^{2}}{2\theta} \right),  \qquad \text{ and } \quad \gamma(x)=\int_{u\cdot n(x) <0}\M(u)|u\cdot n(x)|\d u, \qquad \forall x \in \partial \Omega$$
for some $\theta >0.$  Notice that, actually, $\gamma$ is independent of $x$ and
$$\gamma(x)=\gamma_{d}:=C_{d}\,\int_{\mathbb{R}^{d}}|v|\M(v)\d v, \qquad \forall x \in \partial \Omega$$
for some universal constant $C_{d}=\frac{|\mathbb{S}^{d-2}|}{|\mathbb{S}^{d-1}|}\ds\int_{0}^{1}t(1-t^{2})^{\tfrac{d-3}{2}}\d t.$ One has then $\mathsf{H}(\lp) \subset \Y^{-}_{1}$. Indeed, for $f \in \lp$ nonnegative, one has
\begin{equation*}\begin{split}
\|\mathsf{H}f\|_{\Y^{-}_{1}}&=\frac{1}{\gamma_{d}}\int_{\partial \Omega}\pi(\d x)\int_{v\cdot n(x)<0}|v|^{-1}\M(v)|v\cdot n(x)|\d v\\
&\phantom{+++++++ ++++}\times\int_{v'\cdot n(x)>0}f(x,v')|v'\cdot n(x)|\d v'\\
&\leq \frac{1}{\gamma_{d}}\int_{\partial \Omega}\pi(\d x)\int_{v'\cdot n(x)>0}f(x,v')|v'\cdot n(x)|\d v'=\frac{1}{2\gamma_{d}}\int_{\Gamma_{+}}f\d\mu\end{split}\end{equation*}
where we used that $\int_{\mathbb{R}^{d}}\M(v)\d v=1$ for the first inequality. This shows that $\mathrm{Range}\,\mathsf{H} \subset \Y^{-}_{1}.$ This result extends easily to the case in which  the temperature $\theta$ depends on $x\in\pO$, i.e. $\theta=\theta(x)$ with $\theta(x) \geq \theta_{0} >0$ for any $x \in \pO.$
\end{exe}
\begin{exe}\label{exe:reflec} Recalling that both $V$ and the measure $\bm{m}$ are invariant under the orthogonal group let us consider the pure reflection boundary operator
$$\mathsf{R}\varphi(x,v)=\varphi(x,v-2(v\cdot n(x))n(x)), \qquad (x,v) \in \Gamma_{-}, \quad \varphi \in \lp.$$
and let $\varphi \in \Y^{+}_{1}$. Then, with the change of variables $w=v-2(v\cdot n(x))n(x)$ such that $|w|=|v|$ and $v=w-2(w\cdot n(x))n(x)$ (which preserves the measure $\d\mu_{\pm}$) we get
\begin{equation*}\begin{split}
\|\mathsf{R}\varphi\|_{\Y^{-}_{1}}&=\int_{\Gamma_{-}}\varphi(x,v-2(v\cdot n(x))n(x))\,|v|^{-1}|v\cdot n(x)|\bm{m}(\d v)\pi(\d x) \\
&=\int_{\Gamma_{+}}\varphi(x,w)|w|^{-1}\,|w\cdot n(x)|\pi(\d x)\bm{m}(\d w),
\end{split}\end{equation*}
i.e. $\mathsf{R}(\Y^{+}_{1}) \subset \Y^{-}_{1}$.

Notice that, in this example, in full generality, we can not replace $\Y^{-}_{1}$ with $L^{1}(\Gamma_{-},\tau_{+}\d\mu_{-})$. Actually, requiring  that
$$\mathsf{R}(L^{1}(\Gamma_{+},\tau_{-}\d\mu_{+})) \subset L^{1}(\Gamma_{-},\tau_{+}\d\mu_{-})$$
amounts to assume that there exists $c >0$ such that $\tau_{+}(x,\mathcal{V}(x,v)) \leq c \tau_{-}(x,v)$ for any $(x,v) \in \Gamma_{+}$ 
which is a geometrical condition not satisfied if $\Omega$ is not strictly convex.\end{exe}


A key point is that, under Assumptions \ref{hyp1}, the following holds:
\begin{lemme}\phantomsection\label{lem:lpl1tau}
Assume $\mathsf{H} \in\mathscr{B}(\lp,\lm)$  satisfies Assumptions \ref{hyp1}. Then, for any $\varphi \in \lp$
$$\psi=\Rs(1,\mathsf{M}_{0}(\alpha\mathsf{R}))\mathsf{M}_{0}((1-\alpha)\mathsf{K})\varphi \in \Y^{+}_{1}$$
so that $\mathsf{H}\psi \in \Y^{-}_{1}.$
\end{lemme}
\begin{proof} Notice that, since $\sup_{x\in \pO} \alpha(x)=\alpha_{0} <1$, one has $\|\mathsf{M}_{0}(\alpha\mathsf{R})\|_{\mathscr{B}(\lp)} \leq \alpha_{0}< 1$ and $\Rs(1,\mathsf{M}_{0}(\alpha\mathsf{R}))$ is well-defined. From Assumption \ref{hyp1} \textbf{A1)},  $(1-\alpha(\cdot))K\varphi \in \Y^{-}_{1}.$ Then, from \eqref{eq:M0+Y}, $g=\mathsf{M}_{0}(1-\alpha)\mathsf{K}\varphi \in \Y^{+}_{1}.$ From Assumption \ref{hyp1} \textbf{A2)}, $\alpha\mathsf{R}g \in \Y^{-}_{1}$ and, from \eqref{eq:M0+Y}, $\mathsf{M}_{0}\alpha\mathsf{R}g \in \Y^{+}_{1}.$ More precisely, $\|\mathsf{M}_{0}\alpha\,\mathsf{R}\|_{\mathscr{B}(\Y^{+}_{1})} \leq \alpha_{0}< 1$
so that
$$\psi=\sum_{n=0}^{\infty}(\mathsf{M}_{0}(\alpha\mathsf{R}))^{n}g  \in \Y^{+}_{1}.$$
 Now, it is clear that $\mathsf{H}\psi \in \Y^{-}_{1}$ since $\mathsf{H}(\Y^{+}_{1}) \subset \Y^{-}_{1}$ (notice that $K$ maps any function in $\Y^{-}_{1}$ while $\mathsf{R}(\Y^{+}_{1})\subset \Y^{-}_{1})$).\end{proof}

We can now state our main existence and uniqueness result about invariant density:
\begin{theo}\phantomsection\label{theo:density} 
Let $\mathsf{H} \in\mathscr{B}(\lp,\lm)$ be a regular stochastic partly diffuse boundary operator and let Assumptions \ref{hyp1} and \ref{hyp2} be satisfied. Then,  $(\mathsf{T}_{\mathsf{H}},\D(\mathsf{T}_{\mathsf{H}}))$ is the generator of a \emph{stochastic} $C_{0}$-semigroup  $(U_{\mathsf{H}}(t))_{t\geq 0}$. Moreover, $(U_{\mathsf{H}}(t))_{t\geq 0}$ is irreducible and has a unique invariant density ${\Psi}_{\mathsf{H}} \in \D(\mathsf{T}_{\mathsf{H}})$ with 
$${\Psi}_{\mathsf{H}}(x,v) >0 \qquad \text{ for a. e. } (x,v) \in \Omega \times \R^{d}, \qquad \|{\Psi}_{\mathsf{H}}\|_{X}=1$$
and $\mathrm{Ker}(\mathsf{T}_{\mathsf{H}})=\mathrm{Span}({\Psi}_{\mathsf{H}}).$ Moreover, $(U_{\mathsf{H}}(t))_{t\geq 0}$ is ergodic, Eq. \eqref{eq:ergodic} holds and 
$X=\mathrm{Ker}({\mathsf{T}_{\mathsf{H}}}) \oplus \overline{\mathrm{Range}(\mathsf{T}_{\mathsf{H}})}.$
\end{theo}
\begin{proof} We begin with proving that, under Assumptions \ref{hyp1}, $1 \in \mathfrak{S}_{p}(\mathsf{M}_{0}\mathsf{H})$. Indeed, being both $\mathsf{M}_{0}$ and $\mathsf{H}$  stochastic, the spectral radius of $\mathsf{M}_{0}\mathsf{H}$ is $r_{\sigma}(\mathsf{M}_{0}\mathsf{H})=1 \in \mathfrak{S}(\mathsf{M}_{0}\mathsf{H})$. According to Theorem \ref{theo:ressM0}, one  has 
$$r_{\mathrm{ess}}(\mathsf{M}_{0}\mathsf{H}) < 1=r_{\sigma}(\mathsf{M}_{0}\mathsf{H}).$$
As well-known (see \cite[Theorem 2.1]{marek}), this implies that $1$ is an isolated eigenvalue of $\mathsf{M}_{0}\mathsf{H}$. Moreover, being $\mathsf{M}_{0}\mathsf{H}$ irreducible, we deduce from \cite[Theorem 2.2]{marek} the uniqueness and the strict positivity (almost everywhere) of a nonnegative eigenfunction $\varphi$. 

\  Let us consider now $\lambda \in \mathbb{C}$ with $\mathrm{Re}\lambda >0$. Considering the modulus operator $\left|\mathsf{M}_{\lambda}\mathsf{H}\right|$ (see \cite{chacon} for a precise definition) one has 
$$\left|\mathsf{M}_{\lambda}\mathsf{H}\right| \leq \mathsf{M}_{0}\mathsf{H} \qquad \text{ and } \qquad \left|\mathsf{M}_{\lambda}\mathsf{H}\right| \neq \mathsf{M}_{0}\mathsf{H}.$$
In particular, from \cite[Theorem 4.3]{marek}, $r_{\sigma}\left(\left|\mathsf{M}_{\lambda}\mathsf{H}\right|\right) < r_{\sigma}\left(\mathsf{M}_{0}\mathsf{H}\right)=1.$
Since moreover, $r_{\sigma}(\mathsf{M}_{\lambda}\mathsf{H}) \leq r_{\sigma}(\left|\mathsf{M}_{\lambda}\mathsf{H}\right|)$ according to \cite[Theorem 1]{chacon}, this proves that $r_{\sigma}(\mathsf{M}_{\lambda}\mathsf{H}) <1$, i.e.  $1 \notin \mathfrak{S}(\mathsf{M}_{\lambda}\mathsf{H})$. We conclude that $A={\mathsf{T}_{\mathsf{H}}}$ thanks to \cite[Theorem 4.5]{AL05}. Let us now show that the eigenfunction $\varphi$ lies in $\Y^{+}_{1}.$ Being $\mathsf{M}_{0}\mathsf{H}\varphi=\varphi$, we have $\varphi=\mathsf{M}_{0}(\alpha\mathsf{R})\varphi + \mathsf{M}_{0}(1-\alpha)\mathsf{K}\varphi$
so that, since $1-\mathsf{M}_{0}(\alpha\mathsf{R})$ is invertible, 
$$\varphi=\Rs(1,\mathsf{M}_{0}(\alpha\mathsf{R}))\mathsf{M}_{0}((1-\alpha)\mathsf{K})\varphi.$$
From Lemma \ref{lem:lpl1tau}, we get that $\varphi \in \Y^{+}_{1}$. We deduce then from Proposition \ref{propo:general} that there exists $\Psi_{\mathsf{H}} \in \D(\mathsf{T}_{\mathsf{H}})$ nonnegative and such that $\mathsf{T}_{\mathsf{H}}\Psi_{\mathsf{H}}=0.$ We conclude with Theorem \ref{theo:irred}.\end{proof}
\begin{nb}\phantomsection\label{rem:rem*}
The fact that $\mathsf{T}_{\mathsf{H}}$ is the generator of $(U_{\mathsf{H}}(t))_{t\geq0}$ does not depend on $\mathbf{A1)}$ and $\mathbf{A2)}$ in Assumptions \ref{hyp1}. 
\end{nb}

 \section{Asymptotic stability of collisionless kinetic semigroups}
\label{s:semigroups}

The object of this section is to complement Theorem \ref{theo:irred} and Theorem \ref{theo:density}
where a convergence in \emph{Cesar\`{o} means} of $\left(U_{\mathsf{H}}(t)\right)
_{t\geq 0}$ to its ergodic projection is given. Indeed, under a quite
weak additional assumption on the kernel of $K$ we will show that $\left(
U_{\mathsf{H}}(t)\right)_{t\geq 0}$ is asymptotically stable, i.e. $U_{\mathsf{H}}(t)f$
converges \emph{in norm}  as $t\to +\infty.$ In particular
\begin{equation}
\label{eq:asymp-stab}
\lim_{t \to \infty}
\left\|
U_{\mathsf{H}}(t)f-{\Psi}_{\mathsf{H}}\right\|_{X}=0  
\end{equation}
for any density $f \in X$, i.e. any nonnegative $f$ with $\|f\|_{X}=1$. For the sake of
simplicity, we restrict ourselves to the case in which $\bm{m}(\d v)=\d v$ is the Lebesgue measure  over 
$$V=\left\{v \in \R^{d}\;;\;m \leq |v| \leq M\right\}$$
where $0\leq m < M \leq \infty,$ although the surface Lebesgue measure on the unit sphere can also be dealt with, see Remark \ref{rem:measm} below. 

In order to prove asymptotic stability of $(U_{\mathsf{H}}(t))_{t\geq 0}$ 
we first describe the movement of particles as a piecewise
deterministic Markov process. Then we explain how the stochastic semigroup
$(U_{\mathsf{H}}(t))_{t\geq 0}$ can be defined by this process and finally 
we prove the asymptotic stability of this semigroup.

\subsection{Piecewise deterministic Markov process}
\label{ss:pdmp}
Consider the following stochastic process which
describes the movement of particles.
A particle is moving in the space $\overline{\Omega}$ with a constant velocity  and 
when it strikes the boundary $\partial \Omega$
a new direction is
chosen randomly from the directions that point back into the interior of $\overline{\Omega}$ and
the motion continues. 
We recall that if $x\in \partial \Omega$ and 
$v'\in \Gamma_+(x)$ then the distribution of velocity $v$ after reflection is  given by a probability measure
$P(x,v',\cdot)$ defined on Borel subsets $B$
of $\Gamma_-(x)$ by
\[
P(x,v',B)=\alpha(x)\delta_{\mathcal V(x,v')}(B)+
(1-\alpha(x))\int_B h(x,v,v')\,\bm{m}(\d v).
\]
where $\mathcal V(x,v)$ is the regular reflection law.
From the Assumptions \ref{hyp1} \textbf{A3)} it follows that there exists $\varepsilon_0=1-\sup_{x\in \partial\Omega}\alpha(x)>0$ such that  
$1-\alpha(x)\ge \varepsilon_0$ for all $x\in \partial\Omega$. This implies that
\begin{equation}
\label{est-jump}
P(x,v',\d v)\ge \varepsilon_0 h(x,v',v)\,\bm{m}(\d v).
\end{equation}
Let a particle starts at time $t=0$ from some point $x\in \Omega$ with some initial velocity 
$v\in V\setminus\{0\}$
or from  $x\in \partial \Omega$ with velocity 
$v\in \Gamma_-(x)$. 
Let $x(t)$  be the position and $v(t)$ be the velocity of the particle at time $t$
and let $t_1=t_+(x,v)$.
Then $x(t)=x+vt$ and $v(t)=v$ 
for $t\in [0,t_1)$.
Let $0<t_1<t_2<\dots$ be a sequence of times when a particle hits the boundary $\partial \Omega$. 
Then
\[
\Prob(v(t_n)\in B\,|\,x(t_n^{-})=x,\, v(t_n^{-})=v')=P(x,v',B)
\]
for every Borel subset $B$ of $\Gamma_+(x)$, where
$x(t_n^{-})$ and $v(t_n^{-})$ are the left-hand side limits of
$x(t)$ and $v(t)$, respectively, at the point $t_n$.
Moreover
\[
x(t)=x(t_n)+v(t_n)(t-t_n),\quad v(t)=v(t_n)\quad
\textrm{for $t\in [t_n,t_{n+1})$},
\]
$x(t_n)=x(t_n^{-})$ and $t_{n+1}=t_n+t_+(x(t_n),v(t_n))$ for $n\ge 1$.

It is easy to observe that 
$$\bm{\xi}(t)=(x(t),v(t)),  \qquad t\geq 0,$$
defines a piecewise deterministic Markov process 
\cite{RT-K-k} with values in the space
$$\mathcal E=\left(\Omega\times V\right)\cup\Gamma_-\,.$$
The process
$\left\{\bm{\xi}(t)\right\}_{t\geq0}$  has
\textit{c\`adl\`ag} sample paths, i.e., they are right-continuous with left limits.
Let $\mathcal P(t,x,v,B)$ be the transition probability function for this process, i.e.
$$\mathcal P(t,x,v,B)=\Prob(\bm{\xi}(t)\in B\,|\,\bm{\xi}(0)=(x,v)), $$ 
where $B$ are Borel subsets of $\mathcal E$.
The semigroup $(U_{\mathsf{H}}(t))_{t\geq 0}$ can be uniquely determined  by  
the transition probability function $\mathcal P(t,x,v,B)$ because
the following relation holds
$$\int_B  U_{\mathsf{H}}(t)f(y,w)\,\d y\otimes \bm{m}(\d w)=
\int_{\Omega\times V}
\mathcal P(t,x,v,B) f(x,v)\d x \otimes \bm{m}(\d v)$$
for all $f\in X$, Borel subsets $B$ of  $\Omega\times V$ and $t\ge 0$.

\begin{nb}\phantomsection
\label{str-collision} It should be noted that we do not assume here that 
$\Omega$ is a strictly convex set and it can happen that at some boundary points $x$ some outward or inward vectors belong to the tangent space $\Gamma_0(x)$. In such cases trajectories can be tangent to the boundary $\pO$, especially in the case when we consider the specular reflection (see Fig.~\ref{p:col-5}). 
But there is no need to consider such pathological trajectories because 
the set $\Gamma_0$ has  zero measure for $\mu$  and does not play any role in the definition of the boundary operator $\mathsf{H}$ .
\begin{figure}
\centerline{\includegraphics{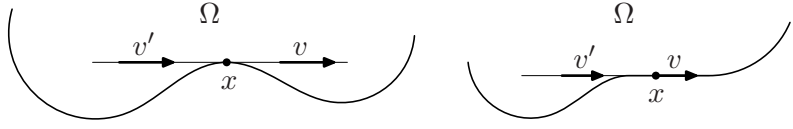}}
\centerline{
\begin{picture}(0,0)(0,0)
\put(-67,25){$x$}
\put(-100,38){$v'$}
\put(-40,38){$v$}
\put(65,33){$v'$}
\put(100,33){$v$}
\put(93,20){$x$}
\put(-75,50){$\Omega$}
\put(80,50){$\Omega$}
\end{picture}
}
\caption{Examples of pathological trajectories in the case of the specular reflection   $v'$ -- outward vector, $v$ -- inward vector.}     
\label{p:col-5}
\end{figure}  
\end{nb}

\subsection{Asymptotic stability}
\label{ss:as-pi}
Now we check that the semigroup $(U_{\mathsf{H}}(t))_{t\geq 0}$
is partially integral (see Appendix \ref{s:piss} for precise definition), i.e. that for some $t>0$ there  
there exists
an integrable function 
$q\colon \Omega\times V\times\Omega\times V\to [0,\infty)$, $q\not\equiv 0$,  
such that 
\begin{equation}
\label{WLD-c-sp}
\mathcal P(t,x,v,B)\ge \int_{B} q(x,v,y,w)\,\d y\otimes \bm{m}(\d w). 
\end{equation}
In order to prove this property we need a rather weak  assumption concerning 
function $h(x,v,v')$.
\begin{defi} \label{defi:wld}
Let $\mathsf{H} \in\mathscr{B}(\lp,\lm)$ be a stochastic partly diffuse boundary operator of the form \eqref{eq:partlydif}. We say that the boundary operator $\mathsf{H}$  is \emph{weakly
locally diffuse} (WLD) if for each point $x\in\partial \Omega$ and 
$v_0'\in \Gamma_+(x)$ there exists a 
$v_0\in \Gamma_-(x)$ and $\delta>0$ such that
\begin{equation}
\label{WLD-c}
k(x,v,v')>0 \quad\text{ for  $\bm{\mu}_{x}$-a. e. } 
v\in \Gamma_{-}(x)\cap B(v_0,\delta), \quad v' \in \Gamma_{+}(x)\cap B(v'_0,\delta).
\end{equation}
If we replace condition (\ref{WLD-c}) by a stronger one:
\begin{equation}
\label{SLD-c}
k(x,v,v')\ge \delta \quad\text{ for  all } 
v\in \Gamma_{-}(x)\cap B(v_0,\delta), \quad v' \in \Gamma_{+}(x)\cap B(v'_0,\delta).
\end{equation}
then the boundary operator $\mathsf{H}$  will be called \emph{strongly locally diffuse} (SLD).\end{defi}

\begin{lemme}
\label{part-integr}
Assume that the operator $\mathsf{H}$  is \emph{weakly
locally diffuse} and satisfies Assumptions \ref{hyp2} and \ref{hyp1}. Then the semigroup $(U_{\mathsf{H}}(t))_{t\geq 0}$
is partially integral. 
\end{lemme}
\begin{proof}
Let $(x,v)\in \mathcal{E}_{0}=\Omega \times (V \setminus\{0\})$ be the initial position and velocity of a particle. 
At time $t_1=t_+(x,v)$ it hits the point $x_1=x+t_+(x,v)v$ on the boundary $\partial \Omega$.
Then we choose a new velocity $\bar v_1\in \Gamma_-(x_1)$ and at time $t_1+t_+(x_1,\bar v_1)$ the particle hits 
the boundary for the second time at the point 
$\bar x_2=x_1+t_+(x_1,\bar v_1)\bar v_1$. 
We choose a new velocity $\bar v_2\in \Gamma_-(x_2)$.
Let $t>0$ satisfies inequalities 
\begin{equation}
\label{ineq-t}
t_1+t_+(x_1,\bar v_1)<t<t_1+t_+(x_1,\bar v_1)+t_+(x_2,\bar v_2)
\end{equation}
and let $\tau=t-t_1$.
We will find an neighborhood $U$ of $(\bar v_1,\bar v_2)\in V$ such that 
for $(v_1,v_2)\in U$ we have 
$v_1\in \Gamma_-(x_1)$, $v_2\in \Gamma_-(x_1+t_+(x_1, v_1) v_1) $
and (\ref{ineq-t}) is satisfied for $(v_1,v_2)\in U$.
Then
\[
\begin{aligned}
x(t)&=x_2+(\tau-t_+(x_1,v_1))v_2\\
&=x_1+t_+(x_1,v_1)v_1+
\big(\tau-t_+(x_1,v_1)\big)v_2  
\end{aligned}
\] 
and $v(t)=v_2$ is the position at time $t$ of the particle
if it starts from $x$ with velocity $v$, and after hitting the boundary 
we choose velocities $v_1$ and $v_2$  (see Fig.~\ref{p:col-3}).
\begin{figure}
\centerline{\includegraphics{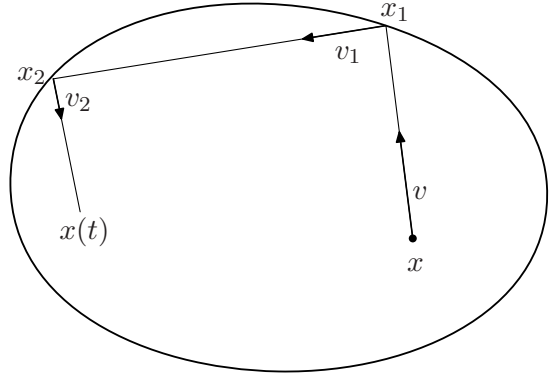}}
\centerline{
\begin{picture}(0,0)(0,0)
\put(48,52){$x$}
\put(50,79){$v$}
\put(38,148){$x_1$}
\put(21,132){$v_1$}
\put(-98,124){$x_2$}
\put(-80,114){$v_2$}
\put(-82,64){$x(t)$}
\end{picture}
}
\caption{Position $x(t)$ of the particle after two collisions with the boundary.} 
\label{p:col-3}
\end{figure}
We define the function $F\colon U\to V$ by   
$$F(v_1,v_2)=\left(x_1+t_+(x_1,v_1)v_1+\big(\tau-t_+(x_1,v_1)\big)v_2,v_2\right)=\left(F_{1}(v_{1},v_{2}),F_{2}(v_{1},v_{2})\right).$$
Now we check that if 
\begin{equation}
\label{cond-diff}
\nabla_{\bar v_1}t_+(x_1,\bar v_1)\cdot \bar v_2\ne
t_+(x_1,\bar v_1)+ \nabla_{\bar v_1}t_+(x_1,\bar v_1)\cdot \bar v_1
\end{equation}
then the function $F$ is a local diffeomorphism in some neighborhood of 
$(\bar v_1,\bar v_2)$. Indeed, let us denote the Jacobian matrix of $F$  by 
$\mathcal{J}_{F}=\big(\,\frac{\partial F}{\partial v_{1}}\,,\,\frac{\partial F}{\partial v_{2}}\big)$. One checks easily that
$$\det \mathcal{J}_{F}(v_{1},v_{2})=\det \left(\,\frac{\partial F_{1}}{\partial v_{1}}\right)=\det \Big(\dfrac{\partial}{\partial v_{1}}\Big(t_+(x_1,v_1)(v_1-v_2)\Big)=:\det M$$
where the matrix $M$ is given by
$$M=c\,\mathrm{Id} + A, \qquad A=a \otimes u=a^t u,$$
where $a$ is the vector $a=\partial_{v_{1}} t_+(x_1,v_1)$, $u=v_{1}-v_{2}$ and $c=t_{+}(x_{1},v_{1})$. We check then that \footnote{Indeed, let $p(z)$ be the characteristic polynomial of $A$, $p(z)=\mathrm{det}(A-z\mathrm{Id})$. Since the rank of $A$ is less or equal to 1, $z=0$ is a root of $p$ with multiplicity at least $d-1$ while 
the trace $\mathrm{tr}(A)=u\cdot a$ should also be a root of $p$.  Therefore,
$$p(z)=(-1)^{d}\,z^{d-1}(z-u\cdot a)$$ and, taking $z=-c$ gives the result.}
$$\det M=c^{d-1}\left(c+ a \cdot u\right)=t^d_+(x_1,v_1)+t^{d-1}_+(x_1,v_1)\nabla_{v_1}t_+(x_1,v_1)\cdot (v_1-v_2)
.$$ 
Consequently, if (\ref{cond-diff}) holds then the Jacobian matrix $\mathcal{J}_{F}(\bar{v}_{1},\bar{v}_{2})$ is non singular and $F$ is a diffeomorphism in some neighborhood of $(\bar v_1,\bar v_2)$.
Observe that condition (\ref{cond-diff})  does not hold
only on a $(2d-1)$-dimensional differentiable manifold in $\mathbb R^{2d}$ 
and we can change equality in (\ref{cond-diff}) to inequality
after a small perturbation of the vector $\bar v_2$. 
We have
\begin{align*}
\mathcal P(t,x,v,B)&=\Prob((x(t),v(t))\in B)\\
&\ge \Prob((v_1,v_2)\in U\colon F(v_1,v_2)\in B)\\
&\ge \int\limits_{F^{-1}(B)}
\varepsilon_0^2 k(x_1,v_1,v)k(x_1+v_1t_+(x_1,v_1),v_2,v_1)\,
\bm{m}(\d v_1)\, \bm{m}(\d v_2).
\end{align*}
Since $x_1=x+vt_+(x,v)$ we can define  
\[
\bm{\kappa}_{(x,v)}(v_1,v_2):=\varepsilon_0^2 k(x_1,v_1,v)k(x_1+v_1t_+(x_1,v_1),v_2,v_1).
\]
Since $k$ satisfies WLD, for each $(x_0,v_0)\in \mathcal E_0$ there exist 
$\delta'>0$ and $\bar v_1\in \Gamma_-(x_1)$, $\bar v_2\in \Gamma_-(x_1+t_+(x_1,\bar v_1)\bar v_1) $ 
such that 
$\bm{\kappa}_{(x,v)}(v_1,v_2)>0$ a. e. for
$(x,v)\in B((x_0,v_0),\delta')$ and  
$(v_1,v_2)\in B((\bar v_1,\bar v_2),\delta')$.
Without lost of generality we can assume that condition (\ref{cond-diff}) holds and $F$ is a diffeomorphism from $U_0=B((\bar v_1,\bar v_2),\delta')$
onto $F(U_0)$. Then
\begin{align*}
\mathcal P(t,x,v,B)
&\ge \int\limits_{F^{-1}(B)} \bm{\kappa}_{(x,v)}(v_1,v_2)\,
\bm{m}(\d v_1)\, \bm{m}(\d v_2)\\
&=
\int\limits_{B\cap F(U_0)} \bm{\kappa}_{(x,v)}(F^{-1}(y,w))
\left|\det \mathcal{J}_{F^{-1}}(y,w)\right|
\,\d y \otimes\bm{m}(\d w)
\end{align*}
where $\mathcal{J}_{F^{-1}}(y,w)=\left(\frac{\partial F^{-1}}{\partial y},\frac{\partial F^{-1}}{\partial w}\right)$ is the Jacobian matrix of $F^{-1}$. 
From the last inequality it follows that 
(\ref{WLD-c-sp}) holds for 
\[
q(x,v,y,w)=\mathbf 1_{F(U_0)}(y,w) \bm{\kappa}_{(x,v)}(F^{-1}(y,w))
\left|\det  \mathcal{J}_{F^{-1}}(y,w)\right| 
\]   
and the semigroup $(U_{\mathsf{H}}(t))_{t\geq 0}$
is partially integral. \end{proof}
\begin{nb} It is very likely that an analytical proof  based upon the Dyson-Phillips-like representation of the semigroup $(U_{\mathsf{H}}(t))_{t \geq 0}$ obtained in \cite{luisa,AL11} may replace the adopted probabilistic proof. Such a proof seems more involved than the probabilistic one given here and we did not investigate further on this point.\end{nb}

Combining Theorem~ \ref{theo:irred}, Theorem~\ref{asym-th2} 
and Lemma~\ref{part-integr} we obtain:
\begin{theo}
\label{th:asym-stab} Let the assumptions of Theorem  \ref{theo:irred} be satisfied. Assume moreover that $\mathsf{H}$  is \emph{weakly locally diffuse}, then the semigroup $(U_{\mathsf{H}}(t))_{t\geq 0}$ is asymptotically stable.
\end{theo}   
\begin{nb} In particular, under the conditions of Theorem \ref{theo:density}, the semigroup $(U_{\mathsf{H}}(t))_{t\geq0}$ is asymptotically stable.\end{nb}

\section{Sweeping properties of collisionless kinetic semigroups}
\label{ss:sweeping}
The asymptotic stability of the semigroup $(V_\mathsf{H}(t))_{t\ge 0}$ is strictly 
connected with the existence of an invariant density which was assumed in Theorem \ref{theo:irred} and proved in Theorem \ref{theo:density}. We investigate here the behaviour of $(U_{\mathsf{H}}(t))_{t\geq 0}$ when this semigroup has no invariant density. A crucial role is played by sweeping 
property (see Appendix \ref{app:sweeping}). We first establish the following which complements Lemma \ref{part-integr}:
\begin{lemme}\phantomsection\label{lem:SLD} If $\mathsf{H}$  is \emph{strongly locally diffuse} (in the sense of Definition \ref{defi:wld}), then, defining $\mathcal{E}_{0}=\Omega \times (V \setminus \{0\})$, for every $(x_0,v_0)\in \mathcal{E}_0$ there exist  $\varepsilon >0$, $t>0$,
and a measurable function 
$\eta_{0}\colon \mathcal{E}_0\to [0,\infty)$ such that 
$$\int_{\mathcal{E}_0} \eta(x,v)\, \d x \otimes \bm{m}(\d v)>0$$ 
and
\begin{equation}
\label{w-eta2}
q(x,v,y,w)\ge
\eta(y,w)\mathbf 1_{B_{0}(\varepsilon)}(x,v)\quad
\textrm{for $(y,w)\in \Omega\times V$},
\end{equation}
where
$B_{0}(\varepsilon)$ is the open ball in $\mathcal{E}_0$ centered at $(x_0,v_0)$
with radius $\varepsilon$. 
\end{lemme}
\begin{proof} The proof uses the notations introduced in the proof of Lemma \ref{part-integr}.  Recall that $(U_{\mathsf{H}}(t))_{t\geq0}$ is partially integral with
$$q(x,v,y,w)=\mathbf 1_{F(U_0)}(y,w) \bm{\kappa}_{(x,v)}(F^{-1}(y,w))
\left|\det \mathcal{J}_{F^{-1}}(y,w)\right|.$$
If the operator $\mathsf{H}$  is strongly
locally diffuse then there 
there exist 
$\delta'>0$ and $\bar v_1\in \Gamma_-(x_1)$, $\bar v_2\in \Gamma_-(x_1+t_+(x_1,\bar v_1)\bar v_1) $ 
such that 
$\bm{\kappa}_{(x,v)}(v_1,v_2)\ge \delta' $ for all
$(x,v)\in B((x_0,v_0),\delta')$ and  
$(v_1,v_2)\in B((\bar v_1,\bar v_2),\delta')$.
Now setting 
$$\eta(y,w)=\mathbf 1_{F(U_0)}(y,w)\left|\det  \mathcal{J}_{F^{-1}}(y,w)\right|$$ we check that $\eta$ satisfies the desired properties.
\end{proof}
\begin{nb}\label{rem:measm} We note that Lemma \ref{lem:SLD} and  Lemma \ref{part-integr} are also true when $\bm{m}(\d v)$ is a surface Lebesgue measure on a sphere but the proofs are slightly more technical. Indeed, instead of two reflections at the boundary (see Figure \ref{p:col-3}) we need one more reflection to achieve the property that the semigroup is partially integral. 
\end{nb}
According to Theorem~\ref{c:sweep2} and the  previous Lemma, we  have: 
\begin{theo}\phantomsection\label{prop:sweep} Let us assume that $(U_{\mathsf{H}}(t))_{t\geq 0}$ is stochastic and has no invariant density. If the boundary operator $\mathsf{H}$  is  strongly local diffuse then $(U_{\mathsf{H}}(t))_{t\geq 0}$  is sweeping from all compact subsets of $\mathcal E_0$.
\end{theo}
\begin{proof} Since $\left((\Omega\times V)\setminus \mathcal E_0\right)$ is of zero measure for the measure $\d x \otimes \bm{m}(\d v)$, we 
can assume that the semigroup $(U_{\mathsf{H}}(t))_{t\geq 0}$ is defined on the space 
$L^1(\mathcal E_0,\mathcal B(\mathcal E_0),\d x \otimes \bm{m}(\d v))$. Then, on this space, Lemma \ref{lem:SLD} exactly means that $(U_{\mathsf{H}}(t))_{t\geq 0}$ satisfies 
property $(K)$ of Theorem \ref{c:sweep2}.\end{proof}

\begin{nb} For any  $\varepsilon>0$ and $M>\varepsilon$ we define the set
\[
F_{\varepsilon,M}=\{(x,v)\in \Omega\times V\colon  \varepsilon \leq |v|\leq M, \,\, \mathrm{dist}(x,\partial\Omega)\ge \varepsilon\},    
\]
where $\mathrm{dist}(x,\partial\Omega)=\inf\{|x-y|\colon y\in \partial\Omega\}$.
Since the set  $F_{\varepsilon,M}$ is compact in the space $\mathcal E_0$, 
for every $f\in X$  we have
\begin{equation} 
\label{c:sw1-e-M}
\lim_{t\to \infty} \int_{F_{\varepsilon,M}}U_{\mathsf{H}}(t)\,f(x,v)\,\d x \otimes \bm{m}(\d v)=0. 
\end{equation}
This result has the following probabilistic interpretation. If the semigroup $(U_{\mathsf{H}}(t))_{t\geq 0}$ has no invariant density,  
the velocity of almost all particles converges to $0$ or to $\infty$ or particles get close to the boundary $\partial \Omega$
when time goes to infinity. 
\end{nb}
We complement Theorem \ref{prop:sweep} by a more precise sweeping result:\begin{theo}
\label{Foguel-alternative}
Assume that $\mathsf{H} \in\mathscr{B}(\lp,\lm)$ is a \emph{regular} stochastic partly diffuse operator given by \eqref{eq:partlydif} and satisfying Assumptions \ref{hyp2}. Assume moreover that $\mathsf{H}$ is weakly locally diffuse (WLD), $\sup_{x \in \partial\Omega}\alpha(x) <1$ and
\begin{equation}\label{eq:mu+k}
\mu_{+}\left(\left\{ (x,v^{\prime })\in \Gamma _{+};\ \int_{\Gamma
_{-}(x)}k(x,v,v^{\prime })\tau _{+}(x,v)\bm{\mu}_{x}(\d v)=+\infty \right\}\right) >0.\end{equation}
Then
\begin{equation}
\label{c:sweeping-spec}
\lim_{t\to\infty}\int_{\Omega \times V} \ind_{\{|v|\ge \varepsilon\}} U_{\mathsf{H}}(t)f(x,v)\,\d x \otimes \bm{m}(\d v)=0, \qquad \forall \varepsilon >0, \qquad \forall f \in X. 
\end{equation}
\end{theo}

\begin{proof} Note first that $\mathsf{T}_{\mathsf{H}}$ is the generator of $(U_{\mathsf{H}}(t))_{t\geq 0}$ (see Remark \ref{rem:rem*}). By virtue of Theorem \ref{KT}, the proof simply consists in showing that $(U_{\mathsf{H}}(t))_{t\geq0}$ has no invariant density and in constructing a function $\Psi=\Psi(x,v)$ such that
\begin{multline}\label{eq:Psi1}
0 < \Psi(x,v) < \infty \quad \text{a. e. on } \Omega \times V, \quad 
\int_{\Omega \times V}\Psi (x,v)\d x \otimes \bm{m}(\d v)=+\infty,\\ \int_{\Omega \times
V}	\ind_{\{|v|\geq \varepsilon \}}\Psi
(x,v)\d x \otimes \bm{m}(\d v)<+\infty \ \ (\varepsilon >0) 
\end{multline}
and 
\begin{equation}\label{eq:Psi2}
U_{\mathsf{H}}(t)\Psi \leq \Psi, \qquad \forall t \geq 0.\end{equation}

The proof will be given in several steps. First of all, according to Theorem~\ref{theo:ressM0}, there exists $\varphi \in \lp$ such
that 
\begin{equation}\label{eq:varphiM0H}
\mathsf{M}_{0}\mathsf{H}\varphi =\varphi ,\ \ \left\Vert \varphi \right\Vert_{\lp} =1.
\end{equation}
Since $\mathsf{M}_{0}\mathsf{K}$  is irreducible then  this $%
\varphi $ is unique.

\indent $\bullet$ \emph{First step:} The function $\Psi=\mathsf{\Xi}_{0}\mathsf{H}\varphi$ satisfies \eqref{eq:Psi1}. Indeed, one first notices that 
\begin{equation*}\begin{split}
\int_{\Gamma _{-}}\left[\mathsf{K}\varphi \right]
&(x,v)\tau _{+}(x,v)\d\mu _{-}(x,v)\\
&=\int_{\partial \Omega }\d\pi (x)\int_{\Gamma
_{-}(x)}\left( \int_{\Gamma _{+}(x)}k(x,v,v^{\prime })\varphi (x,v^{\prime
})\bm{\mu}_{x}(\d v^{\prime })\right) \tau _{+}(x,v)\bm{\mu}_{x}(\d v) \\
&=\int_{\partial \Omega }\d\pi (x)\int_{\Gamma_{+}(x)}\left( \int_{\Gamma
_{-}(x)}k(x,v,v^{\prime })\tau _{+}(x,v)\bm{\mu}_{x}(\d v)\right) \varphi
(x,v^{\prime })\bm{\mu}_{x}(\d v^{\prime }) \\
&=\int_{\Gamma _{+}}\left( \int_{\Gamma
_{-}(x)}k(x,v,v^{\prime })\tau _{+}(x,v)\bm{\mu}_{x}(\d v)\right) \varphi
(x,v^{\prime })\d\mu _{+}(x,v^{\prime }).\end{split}\end{equation*}
Therefore, under assumption \eqref{eq:mu+k} and 
because $\varphi (x,v^{\prime })>0$ a.e. on $\Gamma_{+}$, we have
$$\int_{\Gamma _{-}}\left[\mathsf{K}\varphi \right](x,v)\tau _{+}(x,v)\d\mu _{-}(x,v)=+\infty.$$ 
Using  Lemma~\ref{lem:motau} -- identity \eqref{eq:Xio} --  and with $\varepsilon_{0}=1-\sup_{x\in\partial\Omega}\alpha(x) >0$
\begin{equation*}\begin{split}
\int_{\Omega \times V}&\Psi (x,v)\d x \otimes \bm{m}(\d v)
=\int_{\Gamma _{-}}\left[\mathsf{H}\varphi \right] (x,v)\tau _{+}(x,v)\d\mu
_{-}(x,v) \\
&\geq \varepsilon_{0}\int_{\Gamma _{-}}\left[\mathsf{K}\varphi \right]
(x,v)\tau _{+}(x,v)\d\mu _{-}(x,v)=+\infty.\end{split}\end{equation*}
Hence, 
$$\int_{\Gamma_{-}}\left[\mathsf{H}\varphi\right](x,v)\,|v|^{-1}\d\mu_{-}(x,v)=\infty$$
since $\tau_{-}(x,v) \leq |v|^{-1}D$ (where we recall that $D$ is the diameter of $\Omega$). Thus, $\mathsf{H}\varphi \notin \Y_{1}^{-}$ and $0$ is not an eigenvalue of $\mathsf{T}_{\mathsf{H}}$ (associated to a nonnegative eigenvalue) according to Proposition \ref{propo:general}. Since $\mathsf{T}_{\mathsf{H}}$ is the generator of $(U_{\mathsf{H}}(t))_{t\geq0}$, this means that $(U_{\mathsf{H}}(t))_{t\geq0}$ \emph{has no invariant density.} Moreover, \begin{equation*}\begin{split}
\int_{\Omega \times V}\ind_{\{|v|\geq \varepsilon\}}\Psi (x,v)\d x \otimes \bm{m}(\d v) &=\int_{\Gamma _{-}}\ind_{\{|v| \geq \varepsilon\}}\left[\mathsf{H}\varphi \right](x,v)\tau
_{+}(x,v)\d\mu _{-}(x,v) \\
&\leq \frac{D}{\varepsilon }\left\Vert \mathsf{H}\varphi \right\Vert _{\lm}=\frac{D}{\varepsilon }\left\Vert \varphi \right\Vert _{\lp}.
\end{split}\end{equation*}%
Using that $\mathsf{M}_{0}\mathsf{H}\varphi =\varphi $, one has
$\varphi (x,v)=\left[ \mathsf{H}\varphi \right] (x-\tau _{-}(x,v)v,v),$ for any $(x,v)\in
\Gamma _{+}$ and, from the irreducibility of $\mathsf{M}_{0}\mathsf{H}$, we get $0<\varphi (x,v)<+\infty $ a.e. on $\Gamma _{+}$ which in turns implies that
$$0<\Psi (x,v)=\mathsf{\Xi}_{0}\mathsf{H}\varphi(x,v) =\left[ \mathsf{H}\varphi \right]
(x-t_{-}(x,v)v,v)<+\infty \quad \text{a.e. on }\Omega \times V.$$
This proves that $\Psi$ satisfies \eqref{eq:Psi1}. 

In order to prove that $\Psi$ satisfies also \eqref{eq:Psi2} we shall resort to Lemma \ref{lemme-1-add-sweeping} and for any $n \in \N$, introduce the regular diffuse operator given by $\mathsf{H}_{n}=\alpha\mathsf{R} + (1-\alpha)\mathsf{K}_{n}$ with $\mathsf{K}_{n}$ is defined as in Lemma \ref{lemme-1-add-sweeping}. As before, for any $n \in \N$, there exists $\varphi _{n}\in \lp$
such that 
$$
\mathsf{M}_{0}\mathsf{H}_{n}\varphi _{n}=\varphi _{n},\ \ \left\Vert \varphi _{n}\right\Vert_{\lp}
=1.$$

\indent $\bullet$ \emph{Second step:} $\lim_{n}\|\varphi_{n}-\varphi\|_{\lp}=0.$

To prove this, we notice that
\begin{equation}\label{eq:Hphin}
\mathsf{M_{0}H}\varphi_{n}=g_{n}+\varphi_{n}, \qquad g_{n}=\mathsf{M_{0}(H-H_{n})}\varphi_{n}\end{equation}
with $\|g_{n}\|_{\lp} \leq \|H-H_{n}\|_{\mathscr{B}(\lp)}$ 
since $\|\varphi_{n}\|_{\lp}=1$. In particular, 
$$\lim_{n \to \infty}\|g_{n}\|_{\lp}=0.$$ Now, denote by $\mathsf{P}$ the spectral projection associated to the (simple) eigenvalue $1$ of $\mathsf{M_{0}H}$ one has
$$\mathsf{M_{0}H}=\mathsf{M_{0}HP}+\mathsf{M_{0}H(I-P)}$$
with $\mathsf{M_{0}HP}$ compact (since $\mathsf{P}$ is of finite rank) and $1 \notin \mathfrak{S}(\mathsf{M_{0}H(I-P)})$. One can then write \eqref{eq:Hphin} as 
$\varphi_{n}-\mathsf{M_{0}H(I-P)}\varphi_{n}=g_{n}+\mathsf{M_{0}HP}\varphi_{n}$
or equivalently,
$$\varphi_{n}=\Rs(1,\mathsf{M_{0}H(I-P)})g_{n} + \mathcal{K}_{0}\varphi_{n}$$
with $\mathcal{K}_{0}=\Rs(1,\mathsf{M_{0}H(I-P)})\mathsf{M_{0}HP}$ compact. The sequence $\left(\mathcal{K}_{0}\varphi_{n}\right)_{n}$ is then relatively compact in $\lp$ and, since 
$$\lim_{n}\left\|\Rs(1,\mathsf{M_{0}H(I-P)})g_{n}\right\|_{\lp}=0,$$ 
the sequence $(\varphi_{n})_{n}$ is also relatively compact. There exists then a subsequence, still denoted $(\varphi_{n})_{n}$, which converges in $\lp$ to some $\psi \in \lp$ with unit norm. One clearly has then
$$\psi=\mathcal{K}_{0}\psi=\Rs(1,\mathsf{M_{0}H(I-P)})\mathsf{M_{0}HP}\psi$$
i.e. $\psi-\mathsf{M_{0}H(I-P)}\psi=\mathsf{M_{0}HP}\psi$ or equivalently 
$$\psi=\mathsf{M_{0}H}\psi.$$ We deduce from this that $%
\psi=\varphi$ by uniqueness. This shows finally that the whole sequence
 $\left( \varphi _{n}\right) _{n}$ converges to $\varphi$ in $\lp.$

 \indent $\bullet$ \emph{Third step:} Introducing the semigroup $(V_{\mathsf{H}_{n}}(t))_{t\geq0}$ associated to the boundary operator $\mathsf{H}_{n}$ $n \in \N$, it holds
 \begin{equation}
 \lim_{n}\left\|V_{\mathsf{H}_{n}}(t)f-U_{\mathsf{H}}(t)f\right\|_{X}=0, \qquad \forall t \geq 0, \quad f \in X.\end{equation}
 Indeed, for any $n \in \N$ the resolvent of the generator $\mathsf{T}_{\mathsf{H}_{n}}$ is given by%
$$\Rs(\lambda,\mathsf{T}_{\mathsf{H}_{n}})=\mathsf{\Xi}_{\lambda }\mathsf{H}_{n}\left( \lambda
-\mathsf{M}_{\lambda }\mathsf{H}_{n}\right) ^{-1}\mathsf{G}_{\lambda }+\mathsf{R}_{\lambda}$$
and it is easy to check, using again Lemma \ref{lemme-1-add-sweeping} and Eq. \eqref{eq:reso} that
$$\lim_{n}\left\|\Rs(\lambda,{\mathsf{T}_{\mathsf{H}_{n}}})f-\Rs(\lambda,A)f\right\|_{X}=0, \qquad \forall \lambda >0, \quad f \in X$$
where we recall that $(A,\D(A))$ is the generator of $(U_{\mathsf{H}}(t))_{t\geq0}.$ 
We deduce the second step from the Trotter-Kato approximation Theorem \cite[Theorem 3.19, p. 83]{davies}.

\indent $\bullet$ \emph{Fourth step.} Introduce then $\Psi_{n}=\mathsf{M}_{0}\mathsf{H}_{n}\varphi_{n}.$ According to Theorem~\ref{theo:density},  
$$V_{\mathsf{H}_{n}}(t)\Psi _{n}=\Psi _{n}, \qquad \forall n \in \N, t \geq 0.$$
On the other hand, since $\lim_{n}\|\mathsf{H}_{n}\varphi _{n}-\mathsf{H}\varphi\|_{\lm}=0$, we have, for any $\varepsilon >0$,
$$\lim_{n}\left\|\ind_{\left\{ \left\vert v\right\vert >\varepsilon \right\} }\mathsf{H}_{n}\varphi
_{n}-\ind_{\left\{ \left\vert v\right\vert >\varepsilon \right\}
}\mathsf{H}\varphi\right\|_{\lm}=0$$
and also
$$\lim_{n}\left\|\mathsf{\Xi}_{0}(\ind_{\left\{ \left\vert v\right\vert >\varepsilon \right\}
}\mathsf{H}_{n}\varphi _{n})-\mathsf{\Xi}_{0}(\ind_{\left\{ \left\vert v\right\vert
>\varepsilon \right\} }\mathsf{H}\varphi )\right\|_{\lp}=0$$
or equivalently%
$$\lim_{n}\left\|\ind_{\left\{ \left\vert v\right\vert >\varepsilon \right\} }\Psi
_{n}-\ind_{\left\{ \left\vert v\right\vert >\varepsilon \right\}
}\Psi\right\|_{\lp} =0.$$
Let then
$$\Psi _{n}^{\varepsilon }=\ind_{\left\{ \left\vert v\right\vert >\varepsilon
\right\} }\Psi _{n},\qquad \Psi ^{\varepsilon }=\ind_{\left\{ \left\vert
v\right\vert >\varepsilon \right\} }\Psi, \qquad n \in \N, \:\varepsilon >0,$$
we note that
$$V_{\mathsf{H}_{n}}(t)\Psi _{n}^{\varepsilon }\leq \Psi _{n} \qquad \text{ and } \quad \ind_{\left\{ \left\vert v\right\vert >\varepsilon \right\} }V_{\mathsf{H}_{n}}(t)\Psi
_{n}^{\varepsilon }\leq \Psi _{n}^{\varepsilon }$$
for any $n \in \N,$ $\varepsilon >0,$ $t >0.$ Using the \emph{Third step}, we can pass to the limit in norm in this inequality as $n\rightarrow
+\infty $ and get
$$\ind_{\left\{ \left\vert v\right\vert >\varepsilon \right\} }U_{\mathsf{H}}(t)\Psi
^{\varepsilon }\leq \Psi ^{\varepsilon }\leq \Psi.$$
Letting $\varepsilon \rightarrow 0$, the monotone convergence theorem
yields to $U_{\mathsf{H}}(t)\Psi \leq \Psi$, i.e. $\Psi$ satisfies \eqref{eq:Psi2} and the proof is concluded.\end{proof}

\appendix

\section{About the ballistic flow}\label{app:ballistic}

We establish in this appendix several important properties of the so-called  \emph{ballistic flow}
$$\bxi\::\:(x,v) \in \Gamma_{+} \mapsto \bxi(x,v)=(x-\tau_{-}(x,v)v,v) \in \pO \times V$$
which are fundamental for the proof of our main weak compactness result Theorem \ref{propo:weakcompact}. For the clarity of exposition, we postponed these results in an Appendix but strongly believe that the results stated here have their own mathematical interest. In this Appendix, we will use the following notations: for any element $\bm{z}=(x,v)$ of the extended phase space $\overline{\Omega} \times V$, we will call $x$ the space component of $\bm{z}$ and $v$ the velocity component of $\bm{z}$, writing $x=\bm{z}_{s}$ and $v=\bm{z}_{v}$.

With the notations of \cite{guo03}, $\bxi=x_{\mathbf{b}}.$ Notice that, as already observed in \cite{guo03,voigt}, in non convex domain this deterministic flow does not avoid the grazing set $\Gamma_{0}$, i.e. in full generality
$$\bxi(x,v) \in \Gamma_{-} \cup \Gamma_{0}$$
and -- as far as the regularity of $\bxi$ is concerned -- the set $\{(x,v) \in \Gamma_{+}\,;\,\bxi(x,v) \in \Gamma_{0}\}$ will be particularly relevant. Notice though that
$$\bm{\xi}\: \in \Gamma_{+} \cup \Gamma_{0} \rightarrow \Gamma_{-} \cup \Gamma_{0}$$
is invertible with inverse
$$\bm{\xi}^{-1}\::\:(x,v) \in \Gamma_{-} \cup \Gamma_{0} \longmapsto \bm{\xi}^{-1}(x,v)=\left(x+\tau_{-}(x,v)v,v\right) \in \Gamma_{+} \cup \Gamma_{0}.$$
Moreover, according to \eqref{10.52} with $\psi=\ind_{\Gamma_{0}}$ we see that
$$\int_{\Gamma_{+} \cup \Gamma_{0}}\ind_{\Gamma_{0}}(\bm{\xi}(x,v))\mu_{+}(\d x,\d v)=\mu_{-}(\Gamma_{0})=0$$
which proves that
\begin{equation}\label{eq:SB}
\mu_{+}\left(\left\{(x,v) \in \Gamma_{+} \cup \Gamma_{0}\,;\,\bm{\xi}(x,v) \in \Gamma_{0}\right\}\right)=0.\end{equation}
We introduce the following where we focus on velocity which are unit vectors (this is no loss of generality by virtue of \eqref{eq:scale}) 
\begin{defi} Let 
$$\widehat{\Gamma}_{\pm}=\left\{(x,\omega) \in \partial\Omega \times \mathbb{S}^{d-1}\,;\,(x,\omega) \in \Gamma_{\pm} \quad \text{and} \quad \bxi(x,\omega) \in \Gamma_{\mp}\right\}$$
and introduce, for any ${x} \in \partial \Omega$ the section
$$\widehat{\Gamma}_{\pm}({x})=\left\{\omega \in \mathbb{S}^{d-1}\,;\,({x},\omega) \in  \widehat{\Gamma}_{\pm}\right\}.$$
\end{defi}

\subsection{Regularity of the travel time}

The main result of this section is the following:
\begin{theo}\label{theo:C1} The set $\widehat{\Gamma}_{\pm}$ are open subsets  of $\:\Gamma_{\pm}$ and 
$$\tau_{\mp}\::\:(x,\omega) \in \widehat{\Gamma}_{\pm} \mapsto \tau_{\mp}(x,\omega) \in \R^{+}$$
is of class $\mathcal{C}^{1}.$\end{theo}

We will split the proof of the above in a series of Lemma -- dealing with $\tau_{-}$ but all the results have their counterpart for $\tau_{+}$:
\begin{lemme}\label{lem:cont} The set $\widehat{\Gamma}_{+}$ is an open subset  of $\:\Gamma_{+}$ and $\tau_{-}$ is continuous on $\widehat{\Gamma}_{+}.$
\end{lemme}
\begin{proof}  Let us fix $x_{0} \in \partial\Omega$ and $\omega_{0}\in \widehat{\Gamma}_{+}(x_{0})$, i.e. $\omega_{0} \cdot n(x_{0}) >0$ and $\omega_{0}\cdot n(\xi_{0})<0$ where $\xi_{0}=x_{0}-\tau_{-}(x_{0},\omega_{0})\omega_{0}$ is the space component of $\bxi(x_{0},\omega_{0}).$ For simplicity, set 
$$\tau_{0}=\tau_{-}(x_{0},\omega_{0}).$$ 
Let $(x_{n},v_{n}) \subset \Gamma_{+}$ be a given sequence such that $\lim_{n}(x_{n},v_{n})=(x_{0},\omega_{0})$. In particular, we can assume that $|v_{n}|\neq 0$ for any $n\in \mathbb{N}$. Set then $\omega_{n}=|v_{n}|^{-1}v_{n}  \in \mathbb{S}^{d-1}$ for any $n\in \mathbb{N}$ and
$$y_{n}=x_{n}-\tau_{-}(x_{n},\omega_{n})\omega_{n} \in \partial\Omega \qquad \forall n \in \mathbb{N}.$$
Taking a subsequence if necessary (recall that $\partial\Omega$ is compact), we may assume that $(y_{n})_{n}$ converges to some $y_{0} \in \partial\Omega.$ Then, since $\tau_{-}(x_{n},\omega_{n})=|y_{n}-x_{n}|$ we get that 
$$\lim_{n}\tau_{-}(x_{n},v_{n})=\lim_{n}\tau_{-}(x_{n},\omega_{n})|v_{n}|^{-1}=\lim_{n}\tau_{-}(x_{n},\omega_{n})=|y_{0}-x_{0}|=:\tau_{1}$$ and, consequently, letting $n$ goes to infinity in the definition of $y_{n}$ yields
$$y_{0}=x_{0}-\tau_{1}\omega_{0} \in \partial\Omega.$$
This in particular shows that $\tau_{0}\leq \tau_{1}$. To prove that $\tau_{1}=\tau_{0}$, let us argue by contradiction and assume that $\tau_{1} > \tau_{0}$. Since both $y_{0}=x_{0}-\tau_{1}\omega_{0}$ and $\xi_{0}=x_{0}-\tau_{0}\omega_{0}$ belong to $\partial\Omega$ and since $\omega_{0}\cdot n(\xi_{0}) <0$, the set 
$$\bigg\{t \in \big(\tau_{0}\,,\,\tau_{1} \big)\,;\,x_{0}-t\omega_{0} \notin \overline{\Omega}\bigg\}$$
is open and not empty. Therefore, there exists $\delta >0$ such that $\tau_{1} >\delta+\tau_{0}$ and
\begin{equation}\label{tau1}
x_{0}-t\omega_{0} \notin \overline{\Omega} \qquad \forall  t \in (\tau_{0}\,,\, \delta+\tau_{0}).\end{equation}
Notice that $x_{n}-t\omega_{n}\in \Omega$ for all $t \in (0,\tau_{-}(x_{n},\omega_{n}))$ and any $n\in\mathbb{N}$. Since $\lim_{n}\tau_{-}(x_{n},\omega_{n})=\tau_{1}$, we get that,  for $n \in \mathbb{N}$ large enough, 
$$x_{n}-t\omega_{n} \in \Omega \qquad \forall t \in (0,\tau_{0}+\delta).$$
Letting then $n$ goes to infinity, we obtain $x_{0}-t\omega_{0} \in \overline{\Omega}$ for any $t \in (0,\tau_{0}+\delta)$ which contradicts \eqref{tau1}. Therefore, $\tau_{1}=\tau_{0}$ which proves the continuity of $\tau_{-}$ on $\widehat{\Gamma}_{+}.$ Let us now show that $\widehat{\Gamma}_{+}$ is open. We keep the previous notations, fixing $(x_{0},\omega_{0}) \in \widehat{\Gamma}_{+}$. Let us assume that there exists a sequence $(x_{n},\omega_{n}) \subset \Gamma_{+}$ such that $\lim_{n}(x_{n},\omega_{n})=(x_{0},\omega_{0})$ where $\omega_{n} \in \mathbb{S}^{d-1}$ for any $n\in \mathbb{N}$ but $(x_{n},\omega_{n}) \notin \widehat{\Gamma}_{+}.$ This means that
$$\omega_{n}\cdot n(x_{n}-\tau_{-}(x_{n},\omega_{n})\omega_{n})=0, \qquad \forall n\in \mathbb{N}.$$
From the previous part of the proof, we know that $\lim_{n}x_{n}-\tau_{-}(x_{n},\omega_{n})\omega_{n}=x_{0}-\tau_{-}(x_{0},\omega_{0})\omega_{0}$ and, since $n(\cdot)$ is continuous, we get 
$$\omega_{0}\cdot n(x_{0}-\tau_{-}(x_{0},\omega_{0})\omega_{0})=0,$$
which contradicts the assumption that $(x_{0},\omega_{0}) \in \widehat{\Gamma}_{+}$. Therefore, no sequence with the above properties can exist and $\widehat{\Gamma}_{+}$ is open.
\end{proof}

\begin{lemme}\phantomsection\label{lem:lem25} For any $x \in \partial \Omega$, the mapping
$$\omega \in \widehat{\Gamma}_{+}(x) \longmapsto \tau_{-}(x,w) \in \R^{+}$$
is differentiable and 
\begin{equation}\label{eq:contipw}
{\nabla_{ \omega}}\tau_{-}\::\:(x,\omega) \in \widehat{\Gamma}_{+} \longmapsto \tau_{-}(x,\omega) \in \R^{+}\end{equation}
is continuous.
\end{lemme}
\begin{proof} As before, let us fix $x_{0} \in \partial\Omega$ and $\omega_{0}\in \widehat{\Gamma}_{+}(x_{0})$. Since the normal vector $n(\cdot)$ is continuous on $\partial \Omega$, we deduce from Lemma \ref{lem:cont} that there exists a radius $r >0$ such that 
\begin{equation}\label{eq:11}
\omega \cdot n(x) >0 \qquad \text{ and } \qquad \omega \cdot n(\bxi_{s}(x,\omega)) < 0 \qquad \forall (x,\omega) \in \dU(x_{0}) \times \cU(\omega_{0}),\end{equation}
where $\dU(x_{0})=\mathds{B}(x_{0},r) \cap \partial \Omega$ is an open neighbourhood of $x_{0}$ and $\cU(\omega_{0})=\mathds{B}(\omega_{0},r) \cap \mathbb{S}^{d-1}$ is an open neighbourhood of $\omega_{0}$. The continuity of $\tau_{-}$ implies that there exists $t_{0} >0$ such that
\begin{equation}\label{eq:t0}
\tau_{-}(x,\omega) \geq t_{0} >0 \qquad \forall (x,\omega) \in \dU(x_{0}) \times \cU(\omega_{0}).\end{equation}
Since the mapping $\bm{\xi}\::\:(x,\omega )\in \widehat{\Gamma }_{+}\mapsto \bm{\xi}(x,\omega)=(x-\tau _{-}(x,\omega
)\omega ,\omega )\in \Gamma _{-}$ is continuous, invertible with  inverse
$$\bm{\xi}^{-1}\::\:
(y,\omega )\in \Gamma _{-}\mapsto \bm{\xi}^{-1}(y,\omega)=(y+\tau _{+}(y,\omega )\omega ,\omega
)\in \widehat{\Gamma }_{+}$$
and since $\bm{\xi}(\widehat{\Gamma}_{+}) \subset \widehat{\Gamma}_{-}$, one has $\bm{\xi}^{-1}$ continuous. In particular
$$
\mathbf{\ }\left\{ (x-\tau _{-}(x,\omega )\omega ,\omega ),\ (x,\omega )\in \dU(x_{0})\times \cU(\omega _{0})\right\}$$%
is an open neighbourhood of $\bm{\xi}(x_{0},\omega_{0})$
and
$$\mathds{W}(z_{0}):=\left\{ x-\tau _{-}(x,\omega )\omega ,\ (x,\omega )\in 
\dU(x_{0})\times \cU(\omega _{0})\right\} $$
is an open neighbourhood of $z_{0}=x_{0}-\tau _{-}(x_{0},\omega _{0})\omega _{0}=\bm{\xi}_{s}(x_{0},\omega_{0})\in \partial \Omega.$  
Since $\partial \Omega $ is of class $\mathcal{C}^{1}$ then (up to choosing a smaller neighbourhood $\mathds{W}(z_{0})$ if necessary), there exists a $\mathcal
{C}^{1}$ bijective mapping%
$${\Psi}\::\:{y} \in (-1,1)^{d-1} \longmapsto  {\Psi}(y)\in \mathds{W}(z_{0})$$
with ${\Psi}(0)=z_{0}$ and such that  
such that the range of the differential $\d{\Psi}({y})$ has dimension $d-1$ for any ${y} \in (-1,1)^{d-1}.$
We introduce open pieces of $\partial \Omega$ indexed by $x \in \pO$ 
$$\mathcal{S}_{x}=\left\{ x-\tau _{-}(x,\omega )\omega ,\ \omega \in \cU(\omega _{0})\right\} \subset \mathds{W}(z_{0}).$$
Define then
$$\mathcal{O}_{x}=\Psi^{-1}(\mathcal{S}_{x}),\ x\in \dU(x_{0}),$$
one sees that, for any $x \in \pO$, the mapping $\Psi\::\:y \in \mathcal{O}_{x} \mapsto \Psi(y) \in \mathcal{S}_{x}$ is a parametrization of $\mathcal{S}_{x}$. Namely, given $(x,\omega) \in \dU(x_{0}) \times \cU(\omega_{0})$, there is a unique $y \in \mathcal{O}_{x}$ such that $x-\tau_{-}(x,\omega)\omega=\Psi(y)$. Thus, 
\begin{equation}\label{eq:tau-omega}
\tau _{-}(x,\omega )=\left\vert x-\Psi(y)\right\vert \qquad \text{ and } \qquad  \omega =\frac{x-\Psi(y)}{\left\vert x-\Psi(y)\right\vert }.\end{equation}
In particular $x_{0}-\tau _{-}(x_{0},\omega _{0})\omega _{0}=z_{0}=\Psi(0),$ and 
$\omega _{0}=\frac{x_{0}-\Psi(0)}{\left\vert x_{0}-\Psi(0)\right\vert }.$
Introduce the $\mathcal{C}^{1}$ mapping
\begin{equation}\label{eq:defH}
H\::\:(x,y) \in \dU(x_{0}) \times (-1,1)^{d-1} \mapsto \mathsf{H}(x,y)=\frac{x-\Psi(y)}{|x-\Psi(y)|} \in \cU(\omega_{0})\end{equation}
and, for any $z \in \R^{d} \setminus \{0\}$, let $\mathds{P}_{z}$ denote the orthogonal projection on the hyperplane orthogonal to $z$, 
$$\mathds{P}_{z}h=h_{z}^{\perp}:=h-\langle h,\,\bar{z}\rangle\,\bar{z}, \qquad \bar{z}=\frac{z}{|z|} \in \mathbb{S}^{d-1}, h \in \R^{d}.$$
Notice that $\mathds{P}_{z}=\mathds{P}_{\bar{z}}$ for any $z \in \R^{d}\setminus \{0\}.$ 
Because the differential of the mapping $z \in \R^{d} \setminus\{0\} \mapsto \frac{z}{|z|}$ is given by
$$h \in \R^{d} \mapsto -|z|^{-1}\mathds{P}_{z}h, \qquad h \in \R^{d}$$
it follows that the differential $\d_{y}{H}(x,y)$ of $H$  is given by
\begin{equation}\label{eq:defdH}
\d_{y}{H}(x,y)\::\:h \in \R^{d} \longmapsto  -\frac{\mathds{P}_{x-\Psi(y)}\left(\d\Psi(y)h\right)}{|x-\Psi(y)|}=-\frac{\mathds{P}_{\omega}\left(\d\Psi(y)h\right)}{|x-\Psi(y)|},\end{equation}
where $\omega={H}(x,y)=|x-\Psi(y)|^{-1}\left(x-\Psi(y)\right)$. 
Note that the differential $\d_{y}{H}(x,y)$ depends \emph{
continuously} on $(x,y)\in \dU(x_{0})\times (-1,1)^{d-1}.$ Let us assume for a while that 
\begin{equation}\label{eq:range}
\mathrm{Rank}\left(\d_{y}{H}(x_{0},0)\right)=d-1.\end{equation}
Then, the dimension of the range of $\d_{y}{H}(x,y)$  remains of dimension $d-1$ for $x$ close enough to $0$. Recalling that $\omega _{0}=|x_{0}-\Psi(0)|^{-1}\left(x_{0}-\Psi(0)\right)$, we deduce from the \emph{local inverse function theorem} that, in some open neighbourhood $\dU'(x_{0}) \times (-\delta,\delta)^{d-1}$ of $(x_{0},0)$ and a neighbourhood $\cU'(\omega_{0})$ of $\omega_{0}$ such that  the equation
$$\omega ={H}(x,y) \qquad (x,y) \in \dU'(x_{0}) \times (-\delta,\delta)^{d-1}$$
is solved uniquely as
$$y=G(x,\omega ), \qquad (x,\omega) \in \dU'(x_{0}) \times\cU'(x_{0})$$
where $G(x,\cdot)$ is a $\mathcal{C}^{1}$ mapping on a neighbourhood $\cU'(\omega_{0})$ of $\omega_{0}$ and the mapping $(x,\omega )\mapsto G(x,\omega )$ is
continuous on $\dU'(x_{0}) \times \cU'(\omega_{0}).$ It follows that, for $x \in \dU'(x_{0})$ the mapping 
$$\omega \in \cU'(\omega_{0}) \longmapsto 
\tau _{-}(x,\omega )=|x-\Psi(y)|=|x-\Psi(G(x,\omega))|$$
is  differentiable with differential $\d_{\omega }\tau _{-}(x,\omega )$ given by%
\begin{equation}\label{eq:domegatau}\d_{\omega}\tau_{-}(x,\omega) \::\:h \in \R^{d} \longmapsto  -\frac{\langle x-\Psi (G(x,\omega )),\d\Psi (G(x,\omega
))\d_{\omega}G(x,\omega )h\rangle}{|x-\Psi
(G(x,\omega ))|}.\end{equation}
Since $\d_{\omega}G(x,\omega)=\left(\d_{y}{H}(x,G(x,\omega))\right)^{-1}$ and the mapping $(x,\omega )\mapsto G(x,\omega )$ is continuous then so is
$$(x,\omega ) \in \dU'(x_{0}) \times \cU'(\omega_{0}) \longmapsto  \d_{\omega }\tau _{-}(x,\omega)$$
which proves the Lemma under assumption \eqref{eq:range}. It only remains to prove \eqref{eq:range}. Notice that 
$$\mathrm{Range}(\d\Psi(0))=\left\{ \d\Psi (0)h;\ h\in
\R^{d-1}\right\}$$
is the $(d-1)$-dimensional tangent space of $\pO$  at $z_{0}=\Psi(0)$ with $x_{0}-\Psi(0)=\tau _{-}(x_{0},\omega _{0})\omega _{0}
$
and%
$$\left\{\mathds{P}_{x_{0}-\Psi(0)}\left(\d\Psi(0)h\right)\,;\,h \in \R^{d-1}\right\}=\mathds{P}_{\omega_{0}}(\mathrm{Range}(\d\Psi(0))$$
is the orthogonal projection of $\mathrm{Range}(\d\Psi(0))$  on the orthogonal hyperplane to $\omega _{0}.$ One sees that 
$$
\omega _{0}\cdot n({\Psi (0)})<0 \Longrightarrow 
\omega _{0}\notin \mathrm{Range}(\d\Psi(0)))
$$%
and consequently $\mathds{P}_{\omega_{0}}(\mathrm{Range}(\d\Psi(0)))$ coincides with the orthogonal hyperplane to $\omega _{0}$. In particular, it has dimension $d-1$ which is exactly \eqref{eq:range}.\end{proof}


\begin{lemme}\label{lem:26} For any $\omega \in \mathbb{S}^{d-1}$, the mapping
$$x \in \widehat{\Gamma}_{+}(\omega) \mapsto \tau_{-}(x,w) \in \R^{+}$$
is differentiable and 
$$\nabla_{x}\tau_{-}\::\:(x,\omega) \in \widehat{\Gamma}_{+} \mapsto \tau_{-}(x,\omega) \in \R^{+}$$
is continuous.
\end{lemme}
\begin{proof} As in the previous Lemma, we fix $x_{0} \in \partial\Omega$ and $\omega_{0}\in \widehat{\Gamma}_{+}(x_{0})$ and consider an open neighbourhood $\dU(x_{0}) \times \cU(\omega_{0})$ of $(x_{0},\omega_{0})$ on which \eqref{eq:11} and \eqref{eq:t0} hold. For any $\omega \in \cU(\omega _{0})$, define 
$$\mathcal{S}_{\omega }=\left\{ x-\tau _{-}(x,\omega )\omega ;\ x\in \dU(x_{0})\right\} $$
and, with the notation of the previous Lemma, $\mathcal{S}_{\omega} \subset \mathds{W}(z_{0})$ for any $\omega \in \cU(\omega_{0})$ where $\mathds{W}(z_{0})$ is the image of the $\mathcal{C}^{1}$ function 
$$\Psi\::\: y \in (-1,1)^{d-1} \longmapsto \Psi (y)\in \mathds{W}(z_{0})$$
with $\Psi(0)=z_{0}$ and $\mathrm{Rank}(\d\Psi(y))=d-1$ for any $y \in (-1,1)^{d-1}.$ This allows to introduce, as in the previous Lemma, $\mathcal{O}_{\omega}=\Psi^{-1}(\mathcal{S}_{\omega})$ and $\mathcal{S}_{\omega}$ is parametrized by $\Psi$ (defined now on $\mathcal{O}_{\omega})$, i.e. given $(x,\omega) \in \dU(x_{0}) \times \cU(\omega_{0})$, there is a unique $y \in \mathcal{O}_{\omega}$ such that $x-\tau_{-}(x,\omega)\omega=\Psi(y)$ and \eqref{eq:tau-omega} and \eqref{eq:defH} still hold. We have seen in the proof of Lemma \ref{lem:lem25} that ${H}$  is $\mathcal{C}^{1}$ with differential $\d_{y}{H}(x,y)$ given by \eqref{eq:defdH} and depending continuously on $(x,y)\in \dU(x_{0})\times (-1,1)^{d-1}.$ In particular, as seen earlier, at $(x,y)=(x_{0},0)$ the differential $\d_{y}H$ is  given by 
$$\d_{y}{H}(x_{0},0)\::\:h \in \R^{d-1} \longmapsto  -\frac{\mathds{P}_{\omega_{0}}\left(\d\Psi(0)h\right)}{|x_{0}-\Psi(0)|}=-\frac{\mathds{P}_{\omega_{0}}\left(\d\Psi(0)h\right)}{\tau_{-}(x_{0},\omega_{0})}$$
and has ($d-1$)-dimensional range. As before, from the \emph{implicit function theorem}, there is a neighbourhood $(-\delta,\delta)^{d-1} \times \dU''(x_{0})$ of  $(0,x_{0})$ and a neighbourhood $\cU''(\omega_{0})$ of $\omega_{0}$ on which the equation $\omega =|x-\Psi(y)|^{-1}\left(x-\Psi(y)\right)$ with $(y,x) \in (-\delta,\delta)^{d-1} \times \dU''(x_{0})$
is   solved uniquely as 
$$y=\widehat{G}(\omega ,x), \qquad (\omega,x) \in \cU''(\omega_{0}) \times \dU''(x_{0})$$
where $\widehat{G}(\omega ,\cdot)$ is a $\mathcal{C}^{1}$ mapping and $(\omega,x) \in \cU''(\omega_{0}) \times \dU''(x_{0}) \mapsto \widehat{G}(\omega ,x)$ is continuous. 
It follows that the mapping $x \in \dU''(x_{0}) \mapsto \tau _{-}(x,\omega )=\left\vert \Psi (\widehat{G}(\omega ,x))-x
\right\vert $ is differentiable for any $\omega \in \cU''(\omega_{0})$ with differential given by 
\begin{equation}\label{eq:dx}
\d_{x}\tau _{-}(x,\omega )\::\: h \in \mathcal{T}_{x} \longmapsto  -\frac{\langle  \Psi (\widehat{G}(\omega ,x))-x,\d \Psi (G(\omega ,x))\d_{x}\widehat{G}(\omega ,x)h-h\rangle 
}{|\Psi (\widehat{G}(\omega ,x))-
x|}
\end{equation}
where $\mathcal{T}_{x}$ is the tangent space of $\pO$ at $x\in \partial \Omega .$ Let us prove now the continuity of $\d_{x}\tau_{-}(\cdot,\cdot)$. Because $\omega+\frac{\Psi(\widehat{G}(\omega,x))-x}{|\Psi(\widehat{G}(\omega,x))-x|}=0$, differentiating with respect to $x$ along the direction $h$ tangent at $\pO$ at $x$ yields
$$\frac{\mathds{P}_{\Psi(\widehat{G}(\omega,x))-x}\left(\d\Psi(\widehat{G}(\omega ,x))\d_{x}\widehat{G}(\omega ,x)h-h\right)}{|x-\Psi(\widehat{G}(\omega,x))|}=0
$$
i.e.
$$\mathds{P}_{\omega}\left(\d\Psi(\widehat{G}(\omega ,x))\d_{x}\widehat{G}(\omega ,x)h\right)=\mathds{P}_{\omega}h, \qquad \forall h \in \mathcal{T}_{x}$$
where we used that $\omega$ is the unit vector in the direction of $(\Psi(\widehat{G}(\omega,x))-x)$. 
This implies that
\begin{equation}\label{eq:PsiG}
\d\Psi(\widehat{G}(\omega ,x))\d_{x}\widehat{G}(\omega ,x)h=\mathds{P}_{\omega}h + \left\langle \d\Psi(\widehat{G}(\omega ,x)))\d_{x}\widehat{G}(\omega ,x)h,\omega\right\rangle \omega.\end{equation}
Since $\d\Psi(\widehat{G}(\omega ,x))\d_{x}\widehat{G}(\omega ,x)h$ is a tangent vector to $\pO$ at $\Psi(\widehat{G}(\omega ,x))=\bm{\xi}_{s}(x,\omega)=x-\tau_{-}(x,\omega)\omega$, taking the inner product of the above identity with the normal unit vector $n(\bm{\xi}_{s}(x,\omega))$ yields
$$\left\langle \d\Psi(\widehat{G}(\omega ,x))\d_{x}\widehat{G}(\omega ,x)h,\omega\right\rangle\,\left\langle \omega,n(\bm{\xi}_{s}(x,\omega))\right\rangle=-\left\langle \mathds{P}_{\omega}h,n(\bm{\xi}_{s}(x,\omega)) \right\rangle.$$
Inserting this into \eqref{eq:PsiG} and since $\omega \cdot n(\bm{\xi}_{s}(x,\omega)) \neq0$ we get
\begin{equation*} 
\d\Psi(\widehat{G}(\omega ,x))\d_{x}\widehat{G}(\omega ,x)h=\mathds{P}_{\omega}h -\frac{\left\langle \mathds{P}_{\omega}h,n(\bm{\xi}_{s}(x,\omega)) \right\rangle}{\langle \omega,n(\bm{\xi}_{s}(x,\omega))\rangle}\omega\end{equation*} 
which, plugged into \eqref{eq:dx}, yields
\begin{equation*}\begin{split}
\d_{x}\tau_{-}(x,\omega)h&=-\langle \omega,\mathds{P}_{\omega}h - \frac{\langle \mathds{P}_{\omega}h\,,\,n(\bm{\xi}_{s}(x,\omega)) \rangle}{\langle \omega\,,\,n(\bm{\xi}_{s}(x,\omega)) \rangle}\omega-h\rangle\\
&=\langle h\,;\,\omega\rangle -\frac{\langle \mathds{P}_{\omega}h\,;\,n(\bm{\xi}_{s}(x,\omega)) \rangle}{\langle \omega\,;\,n(\bm{\xi}_{s}(x,\omega)) \rangle}\end{split}\end{equation*}
i.e.
\begin{equation}\label{eq:n(z)} \d_{x}\tau_{-}(x,\omega)h=\frac{\langle h,n(\bm{\xi}_{s}(x,\omega))\rangle}{\langle \omega,n(\bm{\xi}_{s}(x,\omega))\rangle}\end{equation}
This gives directly the continuity of the mapping $(x,\omega) \mapsto \d_{x}\tau_{-}(x,\omega)$ since $\bm{\xi}_{s}$ is continuous on $\widehat{\Gamma}_{+}.$\end{proof}

\begin{nb}\phantomsection\label{nb:guo} Notice that \eqref{eq:n(z)} allows to recover the expression
$$\nabla_{x}\tau_{-}(x,\omega)=\frac{1}{\omega \cdot n(\bm{\xi}_{s}(x,\omega))}\,n(\bm{\xi}_{s}(x,\omega)), \qquad (x,\omega) \in \widehat{\Gamma}_{+}$$
which was obtained in \cite[Lemma 3]{guo03} for some special structure of $\Omega$. Moreover, using \eqref{eq:domegatau} and using the range of the differential of $\Psi(G(x,\omega))$ is the orthogonal of $\mathrm{Span}(n(G(x,\omega))$ we can prove 

$$\nabla_{\omega}\tau_{-}(x,\omega)=\frac{\tau_{-}(x,\omega)}{\omega \cdot n(\bm{\xi}_{s}(x,\omega))}\,n(\bm{\xi}_{s}(x,\omega)), \qquad (x,\omega) \in \widehat{\Gamma}_{+}$$
which, again,  is a result obtained in a special case in \cite[Lemma 3]{guo03}.
\end{nb}
\begin{proof}[Proof of Theorem \ref{theo:C1}] The above three Lemmas give directly the proof of Theorem \ref{theo:C1} for $\tau_{-}$ and $\widehat{\Gamma}_{+}$. The proof for $\tau_{+}$ and $\widehat{\Gamma}_{-}$ is done similarly. 
\end{proof}

An immediate consequence of Theorem \ref{theo:C1} is the following regularity of the ballistic flow:
\begin{cor}\label{cor:ba} The ballistic flow:
$$\bxi\::\:(x,\omega) \in \widehat{\Gamma}_{+} \longmapsto \bxi(x,\omega)=(x-\tau_{-}(x,\omega)\omega,\omega) \in \Gamma_{-}$$
is a $\mathcal{C}^{1}$ diffeomorphism from $\widehat{\Gamma}_{+}$ onto its image and 
$$\bxi^{-1}\::\:(x,\omega) \in \widehat{\Gamma}_{-} \longmapsto \bxi^{-1}(x,\omega)=(x+\tau_{+}(x,\omega)\omega,\omega) \in \Gamma_{+}$$
is a $\mathcal{C}^{1}$ diffeomorphism from $\widehat{\Gamma}_{-}$ onto its image.
\end{cor}
\subsection{Further non degeneracy results} \label{sec:degen}

We introduce a local polar parametrization of the boundary which will turn useful later on. Here and everywhere in the text, $\bxi_{s}^{-1}$ denotes the position component of the inverse $\bxi$, i.e. $\bxi^{-1}_{s}=(\bxi^{-1})_{s}$ \footnote{this should not be confused with the inverse of the position component $(\bxi_{s})^{-1}$}:
\begin{propo}\label{prop:diff}
For any $x\in \pO$, there is a closed subset ${S}(x) \subset \widehat{\Gamma}_{-}(x)$
 with zero  surface Lebesgue measure $\d\sigma$ and such that the mapping 
 $$\bxi_{s}^{-1}(x,\cdot)\::\:\omega \in \widehat{\Gamma}_{-}(x) \setminus S(x) \mapsto  x+\tau_{+}(x,\omega)\omega \in \pO$$
has a differential of rank $d-1$. 
\end{propo}
\begin{proof} For any $x \in \pO$, we choose an orthonormal basis $\left\{\bm{e}_{1}(x),\ldots,\bm{e}_{d-1}(x),\bm{e}_{d}(x)\right\}$  -- depending continuously on $x \in\pO$ -- where
$$\bm{e}_{d}(x)=-n(x).$$
Let us write the components of $\omega \in \widehat{\Gamma}_{-}(x)$ in this basis using polar coordinates
\begin{equation}\label{eq:omeg}
\begin{cases}
\omega_1&=\sin \theta_{d-1}\ldots\sin \theta_3  \sin \theta_2 \sin \theta_1,\\
\omega_2&=\sin \theta_{d-1}\ldots\sin \theta_3  \sin \theta_2 \cos \theta_1,\\
\omega_3&=\sin \theta_{d-1}\ldots\sin \theta_3  \cos \theta_2,\\
&\vdots\\
\omega_{d-1}&=\sin\theta_{d-1}\cos\theta_{d-2},\\
\omega_{d}&=\cos \theta_{d-1}
\end{cases}
\end{equation}
with $\theta_1 \in [0,2\pi]$ and $\theta_2,\ldots,\theta_{d-1} \in [0,\pi]$. Notice that the assumption $\omega \in \widehat{\Gamma}_{-}(x)$ actually implies that $\theta_{d-1} \in [0,\pi/2).$ We will write $\bm{\theta}=(\theta_{1},\ldots,\theta_{d-2},\theta_{d-1})$ and $U:=[0,2\pi] \times [0,\pi]^{d-3}\times \left[0,\frac{\pi}{2}\right],$ so that $\omega \in \widehat{\Gamma}_{-}(x) \implies \bm{\theta} \in U$. Notice that the set $U$ is independent of $x \in \pO$. Within this frame and with the above set of coordinates, we write
\begin{multline*}
T(x,\theta_{1},\ldots,\theta_{d-2},\theta_{d-1}):=\tau_{+}(x,\omega) \qquad \text{ and } \quad \\
 G(x,\theta_{1},\ldots,\theta_{d-2},\theta_{d-1})=x+T(x,\theta_{1},\ldots,\theta_{d-2},\theta_{d-1})\omega, \qquad \forall \omega \in \widehat{\Gamma}_{-}(x).\end{multline*}
From Theorem \ref{theo:C1}, for any $x \in \pO$, the mapping $\bm{\theta} \in U \mapsto T(x,\bm{\theta})$   is of class $\mathcal{C}^{1}$. It is then clear that the set of $\omega \in \mathbb{S}^{d-1}$ given by \eqref{eq:omeg} such that the  vectors 
$$\left(\partial_{\theta_{j}}G(x,\bm{\theta})\right)_{j=1,\ldots,d-1}$$ are \emph{linearly independent} coincide with the set at which  the differential of the mapping $\omega \in \widehat{\Gamma}_{-}(x) \mapsto  x+\tau_{+}(x,\omega)\omega \in \pO$ is of full rank $d-1$. 

One has
$$\partial_{\theta_{j}}G(x,\bm{\theta})=\partial_{\theta_{j}}T(x,\bm{\theta})\omega + T(x,\bm{\theta})\partial_{\theta_{j}}\omega$$ and, because $|G(x,\bm{\theta})-x|=T(x,\bm{\theta})$, it is easy to check that 
 $\partial_{\theta_{j}}T(x,\bm{\theta})=\omega \cdot \partial_{\theta_{j}}G(x,\bm{\theta}),$ i.e.
$$\partial_{\theta_{j}}G(x,\bm{\theta})-(\partial_{\theta_{j}}G(x,\bm{\theta})\cdot \omega)\omega=T(x,\bm{\theta})\partial_{\theta_{j}}\omega.$$ 
Notice that $\partial_{\theta_{j}}G(x,\bm{\theta})-(\partial_{\theta_{j}}G(x,\bm{\theta})\cdot \omega)\omega$ is nothing but the projection of $\partial_{\theta_{j}}G(x,\bm{\theta})$ on the hyperplane $\omega^{\perp}$.
%
Recalling that $T(x,\bm{\theta}) >0$, we see that $\left(\partial_{\theta_{j}}G(x,\bm{\theta})-(\partial_{\theta_{j}}G(x,\bm{\theta})\cdot \omega)\omega\right)_{j=1,\ldots,d-1}$ are independent  if and only if $(\partial_{\theta_{j}}\omega)_{j=1,\ldots,d-1}$ are independent, or equivalently, if the Gram matrix 
$\mathcal{J}_{\bm{\theta}}(\omega)=\big(\,\partial_{\theta_{i}}\omega\,,\,\partial_{\theta_{j}}\omega\big)_{i,j}$ is not singular.  It is well known that 
\begin{equation}\label{eq:detJtheta}\mathrm{det}\left(\mathcal{J}_{\bm{\theta}}(\omega)\right)=\sin^{d-2}\theta_{d-1}\,\sin^{d-3}\theta_{d-2}\,\ldots\sin\theta_{2}.\end{equation}
In other words, if $\mathrm{det}\left(\mathcal{J}_{\bm{\theta}}(\omega)\right) \neq 0$, then $\left(\partial_{\theta_{j}}G(x,\bm{\theta})-(\partial_{\theta_{j}}G(x,\bm{\theta})\cdot \omega)\omega\right)_{j=1,\ldots,d-1}$ are independent and one deduces easily that then 
 $\left(\partial_{\theta_{j}}G(x,\bm{\theta})\right)_{j=1,\ldots,d-1}$ are also independent. We define then
 $$S(x)=\{\omega \in \widehat{\Gamma}_{-}(x)\,;\,\mathrm{det}\left(\mathcal{J}_{\bm{\theta}}(\omega)\right)=0\}$$
 the mapping $\omega \in \widehat{\Gamma}_{-} \setminus S(x) \mapsto x+\tau_{+}(x,\omega)\omega$ has a differential of rank $d-1$. It is clear that $S(x)$ is closed. Let us now prove that indeed  $S(x)$ has a zero  surface Lebesgue measure. Using then \eqref{eq:detJtheta}, we get that 
 $$\omega \in S(x) \qquad \text{  if and only if } \qquad \theta_{j} \in \{0,\pi\} \text{ for some } \:\:j=2,\ldots,d-1.$$
The conditions $\theta_{d-1}\in \{0,\pi\}$ only means $\theta_{d-1}=0$ (recall that $\theta_{d-1} \leq \pi/2$) which means that $\omega=(0,\ldots,0,1).$ Then, for $d \geq 3,$ the condition $\theta_{d-2}\in \{0,\pi\}$ means that 
$$\omega= (0,\ldots,0,\pm \sin\theta_{d-1},\cos\theta_{d-1}),$$ i.e $\omega$ belongs to some (half) unit circle of $\mathbb{S}^{d-1}$.  More generally, the condition $\theta_{d-j}\in \{0,\pi\}$ for $2\leq j \leq d-2$ describes a unit $(j-1)$-dimensional (half)-spheres of $\mathbb{S}^{d-1}$. This means that $S(x)$ can be written as 
$S(x)=\bigcup_{j=2}^{d-2}C_{j}$ where $C_{j}$ is a closed set with positive $(d-j)$--Lebesgue measure \footnote{namely $C_{j}=\{\omega \in \widehat{\Gamma}_{+}(x)\;;\;\theta_{d-j}=0 \text{ or } \theta_{d-j}=\pi\}$} $(2 \leq j \leq d-2).$ Therefore, $\sigma(S(x))=0$ (where we recall that $\d\sigma$ is the Lebesgue surface measure over $\mathbb{S}^{d-1}$) and the conclusion follows.
\end{proof}
\begin{nb}\label{nb:Sepsi}
For any $x \in \pO$ and any $\varepsilon >0$, introduce the set $S_{\varepsilon}(x)$ of all $\omega \in \widehat{\Gamma}_{-}(x) \subset \mathbb{S}^{d-1}$ whose polar coordinates $\bm{\theta}=(\theta_{1},\ldots,\theta_{d-1})$ in the basis $\{\bm{e}_{1}(x),\ldots,\bm{e}_{d}(x)\}$ are such that
$$\mathrm{det}\left(\mathcal{J}_{\bm{\theta}}(\omega)\right) \leq \varepsilon$$
where we recall that the determinant $\mathrm{det}\left(\mathcal{J}_{\bm{\theta}}(\omega)\right)$ is given by \eqref{eq:detJtheta} with $\omega=\omega(\bm{\theta})$ given by \eqref{eq:omeg}. Notice that $\mathrm{det}\left(\mathcal{J}_{\bm{\theta}}(\omega)\right)$ is actually \emph{independent} of $x \in \pO$. Then, the surface Lebesgue measure $\sigma(S_{\varepsilon}(x))$ of $S_{\varepsilon}(x)$ is given by
$$\int_{\mathbb{S}^{d-1}}\ind_{S_{\varepsilon}(x)}(\omega)\sigma(\d\omega)=\int_{U_{\varepsilon}}\mathrm{det}\left(\mathcal{J}_{\bm{\theta}}(\omega)\right)\d\bm{\theta} \leq \varepsilon\,\pi^{d-2}$$
where $U_{\varepsilon}=\{\bm{\theta}\in U\;;\;\sin^{d-2}\theta_{d-1}\,\sin^{d-3}\theta_{d-2}\,\ldots\sin\theta_{2} \leq \varepsilon\}.$ Therefore
$$\lim_{\varepsilon\to0^{+}}\sup_{x \in \pO}\sigma\left(S_{\varepsilon}(x)\right)=0.$$

\end{nb}

We have then the following result {(which is not used in the core of the paper but has its own interest):}
\begin{propo}\label{propo:Uk}
Assume that, for $\pi$-a. e. $x \in \partial\Omega$, $\mathcal{V}(x,\cdot)\::\:\Gamma_{-}(x) \to \Gamma_{+}(x)$ is a field of measurable mappings associated to a pure reflection
boundary operator as in Definition \ref{defi:reflec} and let 
\begin{equation*}
\begin{cases}
\U\::\:(x,v) \in \Gamma_{+} \cup \Gamma_{0}\longmapsto \U(x,v)&=\left(x-\tau_{-}(x,v)v,\mathcal{V}(x-\tau_{-}(x,v)v,v)\right)\\
&=(\bxi_{s}(x,v),\mathcal{V}(\bxi(x,v))) \in \Gamma_{+} \cup \Gamma_{0}.\end{cases}\end{equation*}
For any $k \in \mathbb{N}$ there exists a  subset $\bm{\gamma}_{k} \subset \Gamma_{-}$ such that:
\begin{enumerate}
\item $\bm{\gamma}_{k}$ is a closed subset of $\Gamma_{-}$ with $\mu(\bm{\gamma}_{k})=0$.
\item $\U^{-k}\circ \bm{\xi}^{-1}\left(\Gamma_{-}\setminus \bm{\gamma}_{k}\right)$ is an open subset of $\Gamma_{+}$ and 
$$ \U^{-k}\circ \bm{\xi}^{-1}\::\:\Gamma_{-}\setminus \bm{\gamma}_{k} \to \Gamma_{+}$$
is a $\mathcal{C}^{1}$ diffeomorphism from $\Gamma_{-}\setminus \bm{\gamma}_{k}$ onto its image.\end{enumerate}
\end{propo}
\begin{proof} We first notice that, thanks to Corollary \ref{cor:ba}, $\U\::\:\widehat{\Gamma}_{+} \to \Gamma_{+}$ is a $\mathcal{C}^{1}$ diffeomorphism from $\widehat{\Gamma}_{+}$ onto its image $\U(\widehat{\Gamma}_{+})$ which is an open set of $\Gamma_{+}.$  Let us introduce
$$\Delta:=\Gamma_{+}\setminus \widehat{\Gamma}_{+}=\left\{(x,v) \in \Gamma_{+} \cup \Gamma_{0}\,;\,\bm{\xi}(x,v) \in \Gamma_{0}\right\}$$
As already noticed in \eqref{eq:SB},
\begin{equation}\label{eq:delta}
\mu(\Delta)=0,\end{equation}
i.e. $\Delta$ is a closed set of $\Gamma_{+}$ of zero $\d\mu$-measure. Since 
$$\U\left(\Gamma_{+} \cup \Gamma_{0}\right)=\U(\Gamma_{0}) \cup \U(\widehat{\Gamma}_{+}) \cup \U(\Delta)$$
we have that $\mu\left(\U(\widehat{\Gamma}_{+})\right)=\mu\left(\U\left(\Gamma_{+} \cup \Gamma_{0}\right)\right)=\Gamma_{+}\cup \Gamma_{0}$
since $\U$ is $\mu$-preserving as a composition of the $\mu$-preserving mapping $\bm{\xi}$ and $\mathcal{V}$. Therefore
$$\mu\left(\Gamma_{+}\setminus \U(\widehat{\Gamma}_{+})\right)=0.$$
Introduce then $\Gamma_{+}^{(1)}=\widehat{\Gamma}_{+}$,
$$\Gamma_{+}^{(2)}=\left\{(x,\omega) \in \Gamma_{+}^{(1)}\,;\,\U(x,\omega) \in \widehat{\Gamma}_{+}\right\}$$
and
$$\Lambda_{1}=\Gamma_{+}^{(1)} \setminus \Gamma_{+}^{(2)}.$$
One has $\Lambda_{1}=\U^{-1}(\Delta)$ is a closed subset of $\widehat{\Gamma}_{+}$ with $\mu(\Lambda_{1})=0$. Moreover,
$$\U^{2}\::\:\Gamma_{+}^{(2)} \to \Gamma_{+}$$
is  a $\mathcal{C}^{1}$ diffeomorphism from $\Gamma_{+}^{(2)}$ onto its image. Since $\U^{2}$ is $\mu$-preserving, writing the disjoint unions
$$\Gamma_{+}^{(1)}=\Gamma_{+}^{(2)} \cup \Lambda_{1}, \qquad \U^{2}(\Gamma_{+}^{(1)})=\U^{2}(\Gamma_{+}^{(2)}) \cup \U^{2}(\Lambda_{1})$$
we see that $\mu(\U^{2}(\Gamma_{+}^{(2)}))=\mu(\U^{2}(\Gamma_{+}^{(1)}))=\mu(\U(\Gamma_{+}^{(1)}))$ and
$$\mu\left(\Gamma_{+} \setminus \U^{2}(\Gamma_{+}^{(2)})\right)=0.$$
By induction, assuming that there is $\Gamma_{+}^{(k-1)}  \subset \widehat{\Gamma}_{+}$ such that that
$$\U^{k-1}\::\:\Gamma_{+}^{(k-1)} \rightarrow \Gamma_{+}$$
is a $\mathcal{C}^{1}$ diffeomorphism from $\Gamma_{+}^{(k-1)}$ onto its image $\U^{k-1}(\Gamma_{+}^{(k-1)})$ which is of full $\mu$-measure, i.e.
$$\mu\left(\Gamma_{+}\setminus \U^{k-1}\left(\Gamma_{+}^{(k-1)}\right)\right)=0,$$
then define
\begin{equation}
\label{eq:deltak}
\Gamma_{+}^{(k)}=\{(x,\omega) \in \Gamma_{+}^{(k-1)}\;;\;\U^{k-1}(x,\omega) \in \widehat{\Gamma}_{+}\}\end{equation}
so that 
$$\Lambda_{k-1}=\Gamma_{+}^{(k)}\setminus \Gamma_{+}^{(k-1)}=\left(\U^{k-1}\right)^{-1}(\Delta)$$
is a closed subset of $\Gamma_{+}$ with $\mu(\Lambda_{k-1})=0$ while
$$\U^{k}\::\:\Gamma_{+}^{(k)}\rightarrow \Gamma_{+}$$
is a $\mathcal{C}^{1}$ diffeomorphism from $\Gamma_{+}^{(k)}$ onto its image $\U^{k}\left(\Gamma_{+}^{(k)}\right).$ As before, writing the disjoint unions
$$\Gamma_{+}^{(k-1)}=\Gamma_{+}^{(k)}\cup \Lambda_{k-1}, \qquad \U^{k}\left(\Gamma_{+}^{(k-1)}\right)=\U^{k}\left(\Gamma_{+}^{(k)}\right) \cup \U^{k}(\Lambda_{k-1})$$
we see that $\mu\left(\U^{k}\left(\Gamma_{+}^{(k)}\right)\right)=\mu\left(\U^{k}\left(\Gamma_{+}^{(k-1)}\right)\right)=\mu\left(\U^{k-1}\left(\Gamma_{+}^{(k-1)}\right)\right)$ so that
$$\mu\left(\Gamma_{+}\setminus \U^{k}\left(\Gamma_{+}^{(k)}\right)\right)=0.$$
On the other hand, according to Corollary \ref{cor:ba},
$$\bm{\xi}\::\:\widehat{\Gamma}_{+} \rightarrow \Gamma_{-}$$
is a $\mathcal{C}^{1}$ diffeomorphism from $\widehat{\Gamma}_{+}$ onto its image. Using the definition \eqref{eq:deltak} for $k+1$ we see that
$$\Lambda_{k}=\Gamma_{+}^{(k)}\setminus \Gamma_{+}^{(k+1)}=\left(\U^{k}\right)^{-1}(\Delta)$$
is a closed subset of $\Gamma_{+}$ with $\mu(\Lambda_{k})=0$ and 
$$\bm{\xi} \circ \U^{k}\::\:\Gamma_{+}^{(k+1)} \rightarrow \Gamma_{-}$$
is a $\mathcal{C}^{1}$ diffeomorphism from $\Gamma_{+}^{(k+1)}$ onto its image $(\bm{\xi} \circ \U^{k})\left(\Gamma_{+}^{(k+1)}\right)$. Arguing as before and writing
$$\Gamma_{+}^{(k)}=\Gamma_{+}^{(k+1)}\cup \Lambda_{k}, \qquad \left(\bm{\xi} \circ \U^{k}\right)\left(\Gamma_{+}^{(k)}\right)=\left(\bm{\xi} \circ \U^{k}\right)\left(\Gamma_{+}^{(k+1)}\right) \bigcup
\left(\bm{\xi} \circ \U^{k}\right)(\Lambda_{k})$$
and since $\bm{\xi} \circ \U^{k}$ is $\mu$-preserving, we have that 
$$\mu\left(\left(\bm{\xi} \circ \U^{k}\right)\left(\Gamma_{+}^{(k+1)}\right)\right)=\mu\left(\left(\bm{\xi} \circ \U^{k}\right)\left(\Gamma_{+}^{(k)}\right)\right)=\mu\left(\U^{k}\left(\Gamma_{+}^{(k)}\right)\right)$$ where we used that $\bm{\xi}$ is also $\mu$-preserving. Therefore
$$\mu\left(\Gamma_{-}\setminus \left(\left(\bm{\xi} \circ \U^{k}\right)\left(\Gamma_{+}^{(k+1)}\right)\right)\right)=0.$$
Since $\left(\bm{\xi} \circ \U^{k}\right)(\Gamma_{+}^{(k+1)})$ is an open subset of $\Gamma_{-}$ then setting
$$\bm{\gamma}_{k}:=\Gamma_{-}\setminus \left(\left(\bm{\xi} \circ \U^{k}\right)\left(\Gamma_{+}^{(k+1)}\right)\right)$$
one sees that $\bm{\gamma}_{k}$ is a closed subset of $\Gamma_{-}$ with $\mu(\bm{\gamma}_{k})=0$ and 
$$\U^{-k} \circ \bm{\xi}^{-1}\::\:\Gamma_{-}\setminus \bm{\gamma}_{k} \rightarrow \Gamma_{+}^{(k+1)}$$
is a $\mathcal{C}^{1}$ diffeomorphism from $\Gamma_{-}\setminus \bm{\gamma}_{k}$ onto its image.\end{proof}

{In a previous version of the paper \cite[Lemma A.11]{MKLR}, we claimed that, with the above notations, for any $x \in \pO$ and any $k \in \mathbb{N}$, introducing
$$\mathcal{O}_{k}(x)=\left\{\omega \in \widehat{\Gamma}_{-}(x)\;;\;\bm{\xi}^{-1}(x,\omega) \in \U^{k}(\Gamma_{+}^{(k+1)})\right\}$$
where $\Delta_{k+1}$ is defined thanks to \eqref{eq:deltak}, it holds that, for any $\omega \in \mathcal{O}_{k}(x) \setminus S(x)$,  the differential
$$\d_{\omega}\left[\U^{-k} \circ \bxi^{-1}\right]_{s}(x,\omega)$$
has rank $(d-1).$ As pointed out by an anonymous referee, this result is not true. We provide here a simple counterexample for bounce-back boundary conditions:
\begin{exe}\label{exe:bounce} Assume that $\mathcal{V}$ is associated to bounce-back boundary conditions (see Example \ref{exe:specu}), i.e.
$$\mathcal{V}(x,v)=-v \qquad \forall (x,v) \in \Gamma_{-}.$$
Then, one checks easily that
$$\bm{\xi} \circ \mathcal{U}^{k}(x,v)=\begin{cases} (x,-v) \qquad &\text{ if $k$ is odd}\\
\bm{\xi}(x,v)=(x-\tau_{-}(x,v)v,v) \qquad &\text{ if $k$ is even}\end{cases}$$
which results easily in 
$$\left[\U^{-k} \circ \bxi^{-1}\right]_{s}(x,\omega)=\begin{cases} x \qquad &\text{ if $k$ is odd}\\
\bm{\xi}_{s}^{-1}(x,\omega)=x+\tau_{+}(x,\omega)\omega \qquad &\text{ if $k$ is even}\end{cases}.$$
In particular, the rank of the differential is zero for odd $k$ while it is $d-1$ for even $k$ provided $\omega \in \widehat{\Gamma}_{+}(x) \setminus S(x)$ (see Proposition \ref{prop:diff}). We aim also to point out that, with this choice of the boundary condition and for a diffuse boundary operator $\mathsf{K}$ of the form
$$\mathsf{K}\varphi(x,v)=F(v)\int_{\Gamma_{+}(x)}\varphi(x,v')\bm{\mu}_{x}(\d v')$$
as considered in the proof of Theorem \ref{propo:weakcompact}, one has, for any $k,\ell$ odd, 
$$\mathsf{K}(\mathsf{M}_{0}\mathsf{R})^{k}\mathsf{M}_{0}\mathsf{K}(\mathsf{M}_{0}\mathsf{R})^{\ell}\mathsf{M}_{0}\mathsf{K}\varphi(x,v)=\tilde{F}(x)\mathsf{K}\varphi(x,v), \qquad \varphi \in \lp, \quad (x,v) \in \Gamma_{-}$$
where $\tilde{F}(x)=\int_{\Gamma_{+}(x)}F(-v')\bm{\mu}_{x}(\d v')$, $x \in \pO.$ In particular, one sees that such an operator is not weakly compact (see Remark \ref{nb:weaknb}).
 \end{exe}
}

\section{Reminders on partially integral stochastic semigroups}\label{sec:appB}

We collect here several results on partially integral stochastic semigroups in $L^{1}(E,\Sigma,m)$ where  $(E,\Sigma,m)$ is a given $\sigma$-finite measure space. \subsection{Partially integral stochastic semigroup}\phantomsection
\label{s:piss}

We begin with the following definition
\begin{defi}
A stochastic semigroup $\{P(t)\}_{t\ge 0}$ on the space $L^1(E,\Sigma,m)$ is
called \textit{partially integral} if there exists a measurable function 
$k\colon (0,\infty)\times E\times E\to[0,\infty)$, called a \textit{kernel},
such that for every $y\in E$ all nonnegative $f\in L^1(E,\Sigma,m)$ we have
\begin{equation}
\label{con:pit}
P(t)f(y)\ge\int_E k(t,x,y)f(x)\,m (\d x)
\end{equation}
and 
\begin{equation*}
\int_E\int_E k(t,x,y)\,m(\d x)\otimes m(\d y)>0
\end{equation*}
for some $t>0$.
\end{defi}
We have then the following (see \cite{PR})
\begin{theo}
\label{asym-th2} 
Let $\{P(t)\}_{t\ge 0}$ be a partially integral stochastic
semigroup. Assume that the semigroup $\{P(t)\}_{t\ge 0}$ has a unique
invariant probability density $f_*$. If $f_*>0$ a.e., then the semigroup
 $\{P(t)\}_{t\ge 0}$ is asymptotically stable.
\end{theo}
Let  $\mathcal P(t,y,B)$ be a probability transition function for the semigroup
$\{P(t)\}_{t\ge 0}$, i.e.
\[
\int\limits_B  P(t)f(y)\, m(\d y)=
\int\limits_{E}
\mathcal P(t,x,B) f(x)\,m(\d x)
\]
for all $f\in L^1(E,\Sigma,m)$, $B\in \Sigma$ and $t\ge 0$.
Then inequality (\ref{con:pit}) can be rewritten as
\[
\mathcal P(t,x,\d y)\ge k(t,x,y)\,m(\d y).
\]

\subsection{Sweeping property}
\label{app:sweeping}

We define now the sweeping property for stochastic semigroups:
\begin{defi}
A stochastic semigroup $\{P(t)\}_{t\ge 0}$ on the space $L^1(E,\Sigma,m)$ is
sweeping from a set $A$ if 
\[
\lim_{t\to\infty}\int_A P(t)f(x)\,m(\d x) =0
\]
for each density $f$.  

If moreover $(E,\rho)$ is a 
metric space and $\Sigma=\mathcal B(E)$ is the
$\sigma$--algebra of Borel subsets of $E$, a partially integral semigroup $\{P(t)\}_{t\geq0}$ with kernel $k(t,\xi,z)$ is said to satisfy the property $(K)$ on $E$ if the following holds:
\begin{itemize}
\item[$(K)$] for every $\xi_0\in E$ there exist $\varepsilon >0, t>0$,
and a measurable function 
$\eta\colon E \to [0,\infty)$ such that 
$$\int_{E} \eta(\xi)m(\d \xi) >0$$ and
\begin{equation*}
\label{w-eta}
 k(t,\xi,z)\ge \eta(z)\mathbf 1_{B(\xi_0,\varepsilon)}(\xi), \qquad (\xi,z) \in E \times E,
\end{equation*}
\end{itemize}
where $B(\xi_0,\varepsilon)=\{\xi\in E\colon \rho(\xi,\xi_0)<\varepsilon\}$. \end{defi}
 
 We have the following which is a simple  consequence of a general result concerning asymptotic decomposition of stochastic semigroups (see \cite[Corollary 2]{PR-JMMA2016}):
 \begin{theo}
\label{c:sweep2}
Let $\{P(t)\}_{t\ge0}$ be a stochastic semigroup  on  $L^1(E,\Sigma,m)$, where 
$E$ is a separable metric space,  $\Sigma=\mathcal B(E)$, and $m$ is a $\sigma$-finite measure on $(E,\Sigma)$.
Assume that $\{P(t)\}_{t\ge0}$ has the property {\rm (K)} and has no invariant density.
Then $\{P(t)\}_{t\ge0}$ is sweeping from all compact sets. 
\end{theo}

\subsection{Foguel alternative}
\label{app:sweeping-psi} 
If a stochastic semigroup has no invariant density but we are able to find 
a subinvariant function $f_*>0$, then we can precisely point out all sets having the sweeping property \cite{R-b95}.  We start with some general description.
  
Let a stochastic semigroup $\{P(t)\}_{t\geq 0}$ be given and assume that
this semigroup is partially integral.   
If the kernel $k(t,x,y)$ satisfies 
\begin{equation*}
\int_{E}\int_0^{\infty }k(t,x,y)\,\d t\,m(\d x)>0\quad \text{$y$ \thinspace --\thinspace a.e.},
\end{equation*}%
then $\{P(t)\}_{t\geq 0}$ is called a \textit{pre-Harris semigroup}. 
In particular, if a semigroup is partially integral and irreducible 
then it is pre-Harris semigroup. 
The following condition plays a crucial role in studying sweeping.

\begin{description}
\item[(KT)] There exists a measurable function $f_*$ such that: 
$0<f_*<\infty $ a.e., $P(t)f_*\le f_*$ for $t\ge 0$, $f_*\notin L^1$ and 
$\int_A f_*(x) \,m(\d x)<\infty$.
\end{description}

In \textbf{(KT)} we have written $P(t)f_{\ast }$ for a
non-integrable function. We can use such notation because any substochastic
operator $P$ may be extended beyond the space $L^1$ (see \cite{Foguel} Chap.
I). If $f$ is an arbitrary non-negative measurable function, then we define 
$Pf$ as a pointwise limit of the sequence $Pf_{n}$, where $(f_{n})$ is any
monotonic sequence of non-negative functions from $L^1$ pointwise convergent
to $f$ almost everywhere.

\begin{theo}[\protect\cite{R-b95}, Corollary 3]\phantomsection
\label{KT}
 Let $\{P(t)\}_{t\ge 0}$ be a pre-Harris stochastic semigroup
which has no invariant density. Assume that the semigroup $\{P(t)\}_{t\ge 0}$
and a set $A\in\Sigma$ satisfy condition $\mathbf{(KT)}$. Then the semigroup 
$\{P(t)\}_{t\ge 0}$ is sweeping with respect to~$A$.
\end{theo}


\begin{thebibliography}{99}

\bibitem{aoki}
\textsc{K. Aoki, F. Golse,} On the speed of approach to equilibrium for a collisionless gas, \textit{Kinet. Relat. Models} {\bf 4} (2011), 87--107.


\bibitem{nagel}
\textsc{W. Arendt, A. Grabosch, G. Greiner, U. Groh, H.P. Lotz, U. Moustakas, R. Nagel, F. Neubrander, U. Schlotterbeck,} \textit{\textbf{One-parameter Semigroups of Positive Operators}}, Lecture Notes in Math., vol. 1184, Springer-Verlag, Berlin, 1986. 

\bibitem{AN}
\textsc{L. Arkeryd, A. Nouri}, Boltzmann asymptotics with diffuse reflection boundary conditions,
\textit{Monatsh. Math.}, {\bf 123} (1997), 285--298.



\bibitem{luisa}
\textsc{L. Arlotti}, Explicit transport semigroup associated to abstract boundary conditions, \textit{Discrete Contin. Dyn. Syst. A}, 2011, \emph{Dynamical systems, differential equations and applications. 8th AIMS Conference. Suppl. Vol. I,} 102--111.

\bibitem{AL05}
\textsc{L. Arlotti, B. Lods}, Substochastic semigroups for transport equations with conservative boundary conditions, \textit{J. Evol. Equations}, {\bf 5} (2005) 485--508.

\bibitem{AL11}
\textsc{L. Arlotti,  B. Lods,} Transport semigroup associated to positive boundary conditions of unit norm: a Dyson-Phillips approach, \textit{Discrete Contin. Dyn. Syst. Ser. B}, {\bf 19} (2014)  2739--2766.

\bibitem{mjm1}
\textsc{L. Arlotti, J. Banasiak, B. Lods}, A new approach to transport equations associated to a regular field: trace results and well--posedness, \textit{Mediterr. J. Math.}, {\bf 6} (2009), 367--402.

\bibitem{mjm2}
\textsc{L. Arlotti, J. Banasiak, B. Lods}, On general transport equations with abstract boundary conditions.
The case of divergence free force field, \textit{Mediterr. J. Math.}, {\bf 8} (2011), 1--35.

\bibitem{bardos}
\textsc{C. Bardos,} {Probl\`emes aux limites pour les \'equations aux
d\'eriv\'ees partielles du premier ordre \`a coefficients r\'eels;
th\'eor\`emes d'approximation; application \`a l'\'equation de
transport,}  \textit{Ann. Sci. \'{E}cole Norm. Sup.} \textbf{3} (1970), 185--233.

\bibitem{beals}
\textsc{R. Beals, V. Protopopescu}, {Abstract time-dependent
transport equations},  \textit{J. Math. Anal. Appl.} {\bf 121} (1987),
370--405.


\bibitem{brezis}
\textsc{H. Br\'ezis}, \textit{\textbf{Functional analysis, Sobolev spaces and partial differential equations,}} Universitext, Springer, New York, 2011.


\bibitem{bogachev}
\textsc{V. I. Bogachev}, \textit{\textbf{Measure Theory, Vol. I}},  Springer-Verlag, Berlin, 2007.


 

 
\bibitem{ces1}
\textsc{M. Cessenat}, Th\'eor\`emes de traces $L_p$ pour les espaces
de fonctions de la neutronique, \textit{C. R. Acad. Sci. Paris.},
Ser. I {\bf 299} (1984), 831--834.
\bibitem{ces2}
\textsc{M. Cessenat}, Th\'eor\`emes de traces pour les espaces de
fonctions de la neutronique, \textit{C. R. Acad. Sci. Paris.}, Ser.
I {\bf 300} (1985), 89--92.

\bibitem{chen}
\textsc{I-Kun Chen, C.-H. Hsia, D. Kawagoe,} Regularity for diffuse reflection boundary problem to the stationary linearized Boltzmann equation in a convex domain, \textit{Ann. I. H. Poincar\'e -- AN,} in press, 2018.



\bibitem{chacon}
\textsc{R. V. Chacon, U. Krengel}, Linear modulus of linear operator, 
\textit{Proc. Amer. Math. Soc.}, {\bf 15} (1964), 553--559. 

\bibitem{CPSV}
\textsc{F. Comets, S. Popov, G. M. Sch\"utz, M. Vachkovskaia}, Billiards in a General Domain with Random Reflections, \textit{Arch. Ration. Mech. Anal. } {\bf 191} (2009), 497--537.

\bibitem{dautray}
{\sc R. Dautray, J. L. Lions}, \textit{\textbf{Mathematical analysis
and numerical methods for science and technology. Vol. 6: Evolution
problems II,}} Berlin, Springer, 2000.


\bibitem{davies}
\textsc{E. B. Davies}, \textit{\textbf{One-parameter Semigroups}}, Academic Press, 1980.

 \bibitem{evans}
\textsc{S. N. Evans}, Stochastic billiards on general tables, 
\textit{Ann. Appl. Probab.} {\bf 11} (2001), 419--437.

\bibitem{Foguel} 
\textsc{S. R. Foguel,} \textit{\textbf{The Ergodic Theory of Markov Processes}}, Van Nostrand Reinhold Comp., New York, 1969.



\bibitem{guo03}
\textsc{Y. Guo,} Decay and continuity of the Boltzmann equation in bounded domains, \textit{Arch. Ration. Mech. Anal.}, {\bf 197} (2010) 713--809. 






\bibitem{liu}
\textsc{H. W. Kuo, T. P. Liu, L. C. Tsai,} Free molecular flow with boundary effect, \textit{Comm. Math. Phys.} {\bf 318}
(2013), 375--409.


\bibitem{kim}
\textsc{C. Kim,} Formation and propagation of discontinuity for Boltzmann equation in non-convex domains, \textit{Comm. Math. Phys.} {\bf 308} (2011), 641--701. 

\bibitem{linden}
\textsc{J. Lindenstrauss, L. Tzafriri,} Classical Banach spaces, I. Sequence spaces. Ergebnisse der Mathematik und ihrer Grenzgebiete, Vol. 92. Springer-Verlag, Berlin-New York, 1977.



\bibitem{L-MK}
\textsc{B. Lods, M. Mokhtar-Kharroubi,} Algebraic convergence to equilibrium for the transport equation with partly diffuse boundary conditions, work in progress.

\bibitem{MKLR}
\textsc{B. Lods,  Mokhtar-Kharroubi, R. Rudnicki}, {Invariant density and time asymptotics for collisionless kinetic equations with partly diffuse boundary operators}, \texttt{https://arxiv.org/abs/1812.05397}, 2018.



\bibitem{L-MK-R} 
\textsc{B. Lods,  Mokhtar-Kharroubi, R. Rudnicki}, On stochastic billiards, work in preparation.

\bibitem{marek}
\textsc{I. Marek,} Frobenius theory of positive operators: Comparison theorems and applications, \textit{SIAM J. Appl. Math.} {\bf 19} (1970), 607--628. 
 
\bibitem{mischler}
\textsc{S. Mischler,} Kinetic equations with Maxwell boundary conditions, \textit{Ann. Sci. \'{E}cole Norm. Sup.} {\bf 43} (2010), 719--760


\bibitem{Mkst}
\textsc{M. Mokhtar-Kharroubi,} On collisionless transport semigroups with boundary operators of norm one, \textit{J. Evol. Equ.} {\bf 8} (2008), 327--352.




\bibitem{M-KR}
\textsc{M. Mokhtar-Kharroubi, R. Rudnicki}, 
On asymptotic stability and sweeping of collisionless kinetic equations. \textit{Acta Appl. Math.} \textbf{147} (2017), 19--38. 

\bibitem{MK-seifert}
\textsc{M. Mokhtar-Kharroubi, D. Seifert}, Rates of convergence to equilibrium for collisionless kinetic equations in slab geometry, \textit{J. Funct. Anal.} {\bf 275} (2018),  2404--2452. 


\bibitem{pelci}
\textsc{A. Pelczy\'{n}ski,} On strictly singular and strictly cosingular operators. II. Strictly singular and strictly cosingular operators in $L^{1}(\nu)$-spaces, \textit{Bull. Acad. Polon. Sci. Sér. Sci. Math. Astronom. Phys.} {\bf 13} (1965) 37--41.


\bibitem{PR} 
\textsc{K. Pich\'{o}r, R. Rudnicki,} Continuous Markov semigroups and
stability of transport equations, \textit{J. Math.\ Anal.\ Appl.} \textbf{249} (2000), 668--685.

\bibitem{PR-JMMA2016}
\textsc{K. Pich\'{o}r, R. Rudnicki,}
 Asymptotic decomposition of substochastic operators and semigroups, 
\textit{J. Math. Anal. Appl.} \textbf{436} (2016), 305--321.

\bibitem{R-b95} \textsc{R. Rudnicki,} On asymptotic stability and sweeping
for Markov operators, \textit{Bull. Pol. Ac.: Math.} \textbf{43} (1995), 245--262.



\bibitem{RT-K-k} 
\textsc{R. Rudnicki, M. Tyran-Kami\'nska}, 
\textit{\textbf{Piecewise Deterministic Processes in Biological Models}}, 
SpringerBriefs in Applied Sciences and Technology, Mathematical Methods, Springer, Cham, Switzerland 2017.


\bibitem{stroock}
\textsc{D. W. Stroock}, \textit{\textbf{Essentials of integration theory for analysis}}, Springer, 2011.






\bibitem{voigt}
\textsc{J. Voigt}, 	\textit{\textbf{Functional analytic treatment of the
initial boundary value problem for collisionless gases}},
Habilitationsschrift, M\"unchen, 1981.


\end{thebibliography}
\end{document}